\documentclass[reqno,11pt]{amsart}
 \usepackage{bbm}
\usepackage{url}
\usepackage{soul}
\usepackage{epstopdf}
\usepackage[ruled,vlined]{algorithm2e}
\usepackage[a4paper,bindingoffset=0.5cm,left=2cm,right=2cm,top=2.5cm,bottom=2cm,footskip=.8cm]{geometry}

\usepackage{caption}
\usepackage{pgfplots}
\usepackage{grffile}
\newlength\figureheight
\newlength\figurewidth
\pgfplotsset{%
        tick label style={font=\scriptsize},
        label style={font=\footnotesize},
        legend style={font=\footnotesize},
        every axis plot/.append style={very thick}
}

 \renewcommand{\leq}{\leqslant}
\renewcommand{\geq}{\geqslant}

\allowdisplaybreaks

\makeatletter
\renewcommand{\fnum@figure}[1]{\textbf{\figurename~\thefigure}. }
\renewcommand{\fnum@table}[1]{\textbf{\tablename~\thetable}. }
\makeatother

\usepackage[utf8]{inputenc}
\usepackage{graphicx}
\usepackage{type1cm}        
\usepackage{multicol}        % used for the two-column index
\usepackage[bottom]{footmisc}% places footnotes at page bottom

\usepackage{newtxtext}       %
\usepackage{newtxmath}       % selects Times Roman as basic font

\newcommand{\vb}{\vspace{3.2mm}}

\usepackage[british,UKenglish,USenglish,english,american,dutch]{babel} %Remark: use the same language for uvamath and babel
\usepackage{mathrsfs}
\usepackage{lipsum}
\usepackage{verbatim}
\usepackage{amsmath, amsfonts, amsthm,mathtools}
\usepackage{syntonly}
\usepackage{float}
\usepackage{subcaption}
\usepackage{enumitem}
\usepackage{hyperref}

\usepackage{bbm}
\usepackage{eufrak}
\usepackage{tikz-cd} 
\usepackage{adjustbox}

\newtheorem{lemma}{Lemma}
\newtheorem{corollary}{Corollary}
\newtheorem{assumption}{Assumption}
\newtheorem{theorem}{Theorem}
\newtheorem{remark}{Remark}

\newtheorem{definition}{Definition}
\newtheorem{proposition}{Proposition}

\DeclareMathOperator*{\esssup}{ess\,sup}
\DeclareMathOperator*{\essinf}{ess\,inf}

\begin{document}
\selectlanguage{USenglish}

        \title[Spatiotemporal Hawkes processes  with a graphon-induced connectivity structure]{Spatiotemporal Hawkes processes\\ with a graphon-induced connectivity structure}       
%        \title[Delayed Hawkes birth-death processes: excitation delayed by sojourn times]{Delayed Hawkes birth-death processes: excitation delayed by sojourn times}
\author{Justin Baars, Roger J.~A. Laeven, and Michel Mandjes}
       
        \begin{abstract}
We introduce a spatiotemporal self-exciting point process $(N_t(x))$, boundedly finite both over time $[0,\infty)$ and space $\mathscr X$, with excitation structure determined by a graphon $W$ on $\mathscr{X}^2$. 
This \emph{graphon Hawkes process} generalizes both the multivariate Hawkes process and the Hawkes process on a countable network, and despite being infinite-dimensional, it is surprisingly tractable. 
After proving existence, uniqueness and stability results, we show, both in the \emph{annealed} and in the \emph{quenched} case, that for compact, Euclidean $\mathscr X\subset\mathbb R^m$, any graphon Hawkes process can be obtained as the suitable limit of $d$-dimensional Hawkes processes $\tilde N^d$, as $d\to\infty$. 
Furthermore, in the stable regime, we establish an FLLN and an FCLT for our infinite-dimensional process on compact $\mathscr X\subset\mathbb R^m$, while in the unstable regime we prove divergence of $N_T(\mathscr X)/T$, as $T\to\infty$. 
Finally, we exploit a cluster representation to derive fixed-point equations for the Laplace functional of $N$, for which we set up a recursive approximation procedure. 
We apply these results to show that, starting with multivariate Hawkes processes $\tilde N^d_t$ converging to stable graphon Hawkes processes, the limits $d\to\infty$ and $t\to\infty$ commute.

\vb


\vb

\noindent
{\sc Affiliations.} 
JB and RL are with the Dept.~of Quantitative Economics, University of Amsterdam. RL is also with E{\sc urandom}, Eindhoven University of Technology, and with C{\sc ent}ER, Tilburg University. MM is with the Mathematical Institute, Leiden University, and is also affiliated with Korteweg-de Vries Institute for Mathematics, University of Amsterdam; E{\sc urandom}, Eindhoven University of Technology, Eindhoven; and Amsterdam Business School, University of Amsterdam. 
The research of JB and RL is 
funded in part by the Netherlands Organization for Scientific Research under an NWO VICI grant (2020--2027).
The research of MM is funded in part by the NWO Gravitation project N{\sc etworks}, grant number 024.002.003.

\vb

\noindent
{\sc Email.} \url{j.r.baars@uva.nl}, \url{r.j.a.laeven@uva.nl}, \url{m.r.h.mandjes@math.leidenuniv.nl}.

\vb

\noindent
\textit{Date}: \today.

        \end{abstract}

        \maketitle
 
\newpage

\section{Introduction}
In a wide variety of applied fields, self-exciting point processes are used to model natural, economic and social phenomena that interact locally with their own history, in such a way that past events can trigger events in the future. 
Besides modeling self-exciting behavior in the temporal dimension, nowadays spatiotemporal models are often used, where the excitation occurs both in time and in space. 
This is usually done by introducing a mutually exciting, or multivariate Hawkes, process on $d$ coordinates, interpreting the coordinates as locations in space. 
The basic self- and mutually exciting point process %, or Hawkes process, 
was introduced in 1971, see \cite{Hawkes, Hawkes2}. 

Such multivariate Hawkes processes are employed to capture excitation behavior over time and space in, e.g.: seismology (\cite{seismology, DaleyVereJones, Ikefuji}), where earthquakes tend to trigger sequences of aftershocks both at the same location and at nearby or adjacent fault lines; neuroscience (\cite{neuroscience}), where high-dimensional processes can be used to assess neuronal spike trains; epidemiology (\cite{covid}), where outbreaks of infectious diseases spread contagiously both over time and over different geographical locations; finance (\cite{ACL15, finance2, Large}), where financial shocks and transactions, such as stock trades or order arrivals in electronic markets, cluster over time and across markets; social network analysis (\cite{social media2, social media}), for the prediction of online user activity, or to explain the contagious nature of information diffusion and the spread of online content in social media platforms; and criminology (\cite{criminology}), where crimes like burglaries or gang-related activities tend to cluster in time and across neighborhoods. 

In these modeling approaches, the respective space is typically divided into $d$ locations, and one measures the number of events $N_t=(N_t(1),\ldots,N_t(d))$ that have occurred by time $t$ at those $d$ locations. %Under some regularity assumptions, the probabilistic behavior of a counting process can be characterized by the conditional intensity process $(\lambda_t(1),\ldots,\lambda_t(d))$, see \cite{DaleyVereJones}, meaning that the components $N_t(i)$ satisfy, as $\Delta t\downarrow0$,
%\begin{align*}
%\mathbb P(N_{t+\Delta t}(i)-N_t(i)=0|\mathcal F_t)&=1-\lambda_t(i)\Delta t+o(\Delta t);\\
%\mathbb P(N_{t+\Delta t}(i)-N_t(i)=1|\mathcal F_t)&=\lambda_t(i)\Delta t+o(\Delta t),
%\end{align*}
%where $(\mathcal F_t)_{t\in\mathbb R}$ is the $\sigma$-algebra generated by $N$. In case of a multivariate Hawkes process, we set 
%$$
%\lambda_t(i)=\phi_i\left(\sum_{j=1}^d\sum_{\substack{(s,y)\in N(j)\\s<t}}h_{ij}(t-s)\right),
%$$
%where the \emph{excitation functions} $h_{ij}:[0,\infty)\to\mathbb R$ model how events in coordinate $j$ affect the future evolution in coordinate $i$, and where the functions $\phi_i:\mathbb R\to[0,\infty)$ allow for nonlinear effects. 
By not distinguishing between more granular locations \emph{within} each of those $d$ locations, one essentially divides the space into $d$ approximating, `homogeneous' locations or `subpopulations', possibly neglecting important information.
The main advantage of assuming this homogeneity within subspaces is that the probabilistic behavior of the corresponding (finite-dimensional) multivariate Hawkes process is well understood; see e.g., \cite{Stabilitypaper, Massoulie, Zhu, Zhu2}. 
Furthermore, it is highly tractable:
moments can readily be determined for Markovian models (\cite{elementary, matrix method, Infinite server queues});  
transforms can be approximated iteratively (\cite{multivariateKLM}); 
heavy-tailed and heavy-traffic asymptotics have been identified (\cite{multivariateKLM, Infinite server queues}); a broad range of scaling limits have been studied (\cite{Bacry, Horst, Rosenbaum1, Rosenbaum2}); and recently results on the distribution of the cluster duration of the Hawkes process have become available along with the combinatoric structure inherent to its branching representation (\cite{Conditional uniformity}). 
Moreover, a substantial amount of research has been done on statistical inference, both for the parametric case (\cite{parametrics}) and, more recently, for the nonparametric case (\cite{Boswijk, Hansen, Kirchner}). 

From this well-understood realm, the focus within the applied probability literature has gradually shifted towards understanding Hawkes processes in high-dimensional settings; see \cite{Galves, Massoulie} for early research in this direction.  
More recently, Hawkes processes on large, \textit{countable} networks have been defined in the influential paper \cite{Delattre16}, where the connectivity structure is provided by an infinite directed graph. 
In such a model, each node still represents a location or subpopulation. 
When a spatial structure is absent, and a complete graph on $d$ nodes is imposed for the interaction structure, the large population limit yields an inhomogeneous Poisson process with a \emph{deterministic} intensity function solving a certain convolution equation. 
This model is still quite tractable: its large-time behavior can be described in the mean-field setting and in a nearest-neighbor version of the model exact asymptotics have been provided (\cite{Delattre16}); and recently, a large deviations principle for the mean-field limit has been established (\cite{LD mean field Hawkes}).

The idea of \cite{Delattre16} has been extended in \cite{Agathe-Nerine}, where the mean-field prelimit becomes a model on $d$ nodes sampled from a Euclidean space $I=[0,1]$ or $I=\mathbb R^d$, and where the connectivity between those nodes is modeled by a graphon $W:I^2\to\mathbb R_+$. 
In contrast to \cite{Delattre16}, the coordinates in \cite{Agathe-Nerine} are sampled from an \emph{uncountable} space. 
In particular, the mean-field limit is an uncountable collection of inhomogeneous Poisson processes, % $\{\bar Z_i:i\in I\}$, 
whose intensity functions satisfy a convolution equation. 
Interestingly, in this setting, the large-time behavior can still be analyzed, and is related to the spectral radius of an operator corresponding to this convolution equation (\cite{Agathe-Nerine}). 
%We emphasize that \cite{Agathe-Nerine} analyzes a Hawkes process in a high, but finite-dimensional setting.

In different subfields of applied probability, there has been a rise in research on describing, or rather approximating, excitation behavior in an infinite network with a graphon-induced (\cite{Lovasz}) connectivity structure. 
Typically, finite-dimensional models with a graph-induced connectivity structure approximate infinite-dimensional models with a graphon-induced connectivity structure, as the dimension grows large. 
Such approximations have been investigated for macro-level, deterministic SEIR models (\cite{Naldi}), micro-level, stochastic SIR models (\cite{Pang}), and for linear threshold models (\cite{Erol}). 
In the case of Hawkes processes, spatiotemporal and ETAS models have been studied in several applied fields, including seismology and statistics (\cite{Fox, Ilhan, Kwon, Reinhart}), but to the best of our knowledge, there is no account in the literature of a Hawkes process with a graphon-induced connectivity structure. 
Perhaps the closest to such a model are $d$-dimensional ($d\in\mathbb N$) Hawkes processes with connectivity between nodes sampled from a graphon (\cite{Agathe-Nerine, Xu}), but those mutually exciting processes themselves are \emph{finite-dimensional}.

In the present work, we introduce, and formally establish, a spatiotemporal self-exciting point process $N_t(x)$ on $[0,\infty)\times \mathscr X$, where $(\mathscr X,\mathscr A,\mu)$ is some $\sigma$-finite measure space that is locally finite in both space and time, with cross-excitation determined by a graphon $W:\mathscr X^2\to\mathbb R_+$.
We define our process through a \emph{conditional intensity density} $\lambda_t(x)$, which is such that the simple point process $N$ satisfies 
$$\mathbb P(N_{t+\mathrm dt}(x+\mathrm dx)-N_t(x)=1|\mathcal F_t)=\lambda_t(x)\ \mathrm dt\mathrm dx,$$
where $(\mathcal F_t)_{t\in\mathbb R}$ is the $\sigma$-algebra generated by $(N_t)_{t\in\mathbb R}$, i.e., $\mathcal F_t=\sigma(N_s(A):s\leq t,A\in\mathscr A)$.
In the linear case, for some baseline intensity $\lambda_\infty$, a mark stochastic process $B$ and an excitation function $h$, we set $$\lambda_t(x)=\lambda_\infty(x)+\sum_{\substack{(s,y)\in N\\s<t}}B_{xy}(s)W(x,y)h(t-s).$$ 
In this way, $\lambda_t(x)$ can be interpreted as the \emph{spatial density} of the conditional intensity. 
In particular, this representation of the process implies that it is locally finite over space, so while being \emph{infinite-dimensional}, it generates a finite number of events over a bounded subset of time and space. 
We coin this process a \emph{graphon Hawkes process}.

If we partition our space $\mathscr X$ into finite measure sets $\mathscr X_k$, and take the parameters of the graphon Hawkes process constant on $\mathscr X_k$, then the location \emph{within} $\mathscr X_k$ becomes irrelevant, and our process reduces to a Hawkes process on a discrete spatial set. 
In this way, our novel process generalizes both the multivariate Hawkes process and the process studied in \cite{Delattre16}. 
Note that our graphon Hawkes process is fundamentally different from the process studied in \cite{Agathe-Nerine}, which is a Hawkes process on \emph{finitely many} coordinates with a sampled connectivity graph, and its resulting mean-field limit intensity function, though defined on a continuum of space, is \textit{deterministic} and not boundedly finite over space. 
%Thus, this paper contributes to the literature by generalizing the model studied in \cite{Delattre16} to a continuous-space setting.

We outline our main results and contributions to the literature. 
After having formally introduced the graphon Hawkes process in Section~\ref{sectiondefs}, we prove corresponding existence, uniqueness and stability results %for possibly nonlinear models 
in Theorem~\ref{thm stability general X spectral radius} of Section~\ref{section existence etc}, under a condition on the spectral radius $\rho(T_{\mathrm{hom}})$ of a linear operator $T_{\mathrm{hom}}$ related to the excitation part of the conditional density intensity function. 
This provides %more 
general, weak stability conditions. % than typically found in the literature, see e.g., \cite{Massoulie}. 
Furthermore, our general infinite-dimensional framework requires functional-analytic arguments in later sections of the paper that would not be required should one work in a finite-dimensional setting. % with the stability conditions of \cite{Massoulie}. 
Whereas the results of Section~\ref{section existence etc} can be proved by suitably leveraging and generalizing existing techniques from \cite{Massoulie}, the results of Sections~\ref{section convergence results}--\ref{section6 fixed point theorems} require genuinely novel techniques and ideas.

In Section~\ref{section convergence results}, Theorems~\ref{annealed theorem} and~\ref{quenched theorem} establish that (under suitable regularity assumptions) the $d$-variate process $\tilde N^d$ obtained by averaging the parameters of a linear graphon Hawkes process $N$ --- defined on compact $\mathscr X=[\boldsymbol a,\boldsymbol b]\subset\mathbb R^m$ --- over the elements of partitions $\mathcal P^d$ of $\mathscr X$, converges uniformly on compact sets in probability to $N$, as $d\to\infty$. 
We prove this both in the \emph{annealed} and in the \emph{quenched} case. 
This shows that the graphon Hawkes process is a natural continuous-space generalization of the multivariate Hawkes process. %, and those results may be compared to \cite{Erol, Naldi, Pang}, which consider different applied probability models. % , or to the mean-field limits from \cite{Agathe-Nerine, Delattre16}. 
In passing, we prove in Lemma~\ref{lemma offspring size} uniform boundedness for summations over operator norms of iterates of the kernel operators corresponding to steppings of a graphon (cf.\ \cite{Lovasz}, $\S$7.5 and $\S$9.2), a result that may be of interest in its own right, and might be of use when infinite-dimensional systems are approximated by finite-dimensional systems in different subfields of applied and numerical mathematics. 
In Lemma~\ref{TVlemma4}, we prove a version of the Poincaré inequality that might be interesting in its own right as well. 
Relying on those functional-analytic arguments and on stochastic coupling, we can give %particularly 
insightful proofs for our convergence results in Theorems~\ref{annealed theorem} and~\ref{quenched theorem}.

In Section~\ref{large time behavior section}, we provide a functional law of large numbers (FLLN) in the stable case $\rho(T_{\mathrm{hom}})<1$ in Theorem~\ref{FLLN theorem} and we prove divergence of this prelimit in the unstable case $\rho(T_{\mathrm{hom}})>1$ in Theorem~\ref{FLLN theorem 2}; 
hence, we provide a dichotomy in large-time behavior between those two cases. %These results resemble \cite[\S4]{Agathe-Nerine}, which, in contrast to our situation, works with a determinstic mean-field limit. 
In the stable regime, we also describe a functional central limit theorem (FCLT) in Theorem~\ref{FCLT theorem}. 
In proving our FLLN and FCLT, we show how to essentially reduce to a finite-dimensional setting, allowing for relatively simple proofs of functional limit theorems for infinite-dimensional systems. 
Given the difficulty of proving such results, our approach may be of interest for other infinite-dimensional models as well.

Finally, we describe in Section~\ref{section6 fixed point theorems} the probabilistic behavior of the graphon Hawkes process by characterizing its Laplace functional through fixed-point equations, which provides an iterative approximation procedure in the transform domain, non-trivially extending results from \cite{multivariateKLM, Infinite server queues} to our infinite-dimensional setting. 
As an application, we show that, starting with multivariate Hawkes processes $\tilde N^d_t$ converging to stable graphon Hawkes processes, we can interchange the limits $d\to\infty$ and $t\to\infty$. 

By not discretizing space, we obtain a natural model for phenomena occurring in time and space that is more general, and possibly more realistic, than the classical multivariate Hawkes process, especially in applications where heterogeneous behavior over space is expected. 
The inaccuracy originating from discretizing space can be reduced by making the approximating, `homogeneous' sublocations smaller, i.e., by working in a higher-dimensional setting. 
In the present work, this is taken a step further by working on a continuous spatial set. 
In applications where excitation or infectivity depends continuously on the distance between two locations or on other geographical characteristics, a continuous-space model provides a clear advantage over finite-dimensional models. 
For example, the likelihood of an aftershock in location $x$ following an earthquake with epicenter $y$ naturally depends continuously on the distance between $x$ and $y$, and may also depend on geographical characteristics. 
Similarly, the likelihood of infection for COVID-19 depends on the distance between an individual and a patient, and on characteristics such as the amount and type of restrictive measures, which vary over space. 
Moreover, when applying Hawkes processes to crime modeling (\cite{criminology}), the likelihood of a crime in location $y$ following a crime in $x$ may depend on the distance between $x$ and $y$, as well as on sociological factors, such as the level of law enforcement, of social control, or of education.
Therefore, the model developed in the present work may be of importance to applied fields where multivariate Hawkes processes are used and spatial aspects are relevant. 

%\noindent
\subsection*{Conventions} 
Throughout this paper, $\mathbb N:=\{1,2,3,\ldots\}$ and $\mathbb N_0:=\mathbb N\cup\{0\}$. 
Furthermore, the positive half-line is denoted by $\mathbb R_+:=[0,\infty)$. 
In general, the temporal variable is denoted by $t$, whereas the spatial variable is denoted by $x$.

When writing `a multivariate Hawkes process', we mean a Hawkes process on $d\in\mathbb N$ coordinates; in particular, a multivariate Hawkes process is finite-dimensional, and we do not refer to Hawkes processes on countable networks or to graphon Hawkes processes by this term.

\subsection*{Appendix}
The proofs that are not given in the main text and some supplementary results are collected in four appendices.

\section{Definitions of the graphon Hawkes process}\label{sectiondefs}
In this section, we formally define graphon Hawkes processes, i.e., spatiotemporal self-exciting point processes on an uncountable spatial set, with a connectivity structure defined through a graphon $W$. 
We provide two constructions: one through a \emph{conditional intensity density}, and one using a \textit{cluster process representation}, i.e., by describing the dynamics as a spatial Poisson branching process. %In this article, apart from Section \ref{section6 fixed point theorems}, we will mainly use the representation through the conditional intensity density.
The conditional intensity density representation allows for a form of nonlinearity, whereas the cluster representation does not. 
In the linear case the two representations are equivalent.
In Section~\ref{section existence etc}, we mainly use the conditional intensity density representation; in Section~\ref{section6 fixed point theorems}, we mainly use the cluster representation; in Sections~\ref{section convergence results} and~\ref{large time behavior section}, we use both.

We start by defining a spatiotemporal point process on $[0,\infty)\times\mathscr X$ through a conditional intensity density, for which we need a formal definition. 
Let $L^p_{\mathrm{loc}}(\mathscr X)$ be the set of measurable functions $f$ on a $\sigma$-finite measure space $(\mathscr X,\mathscr A,\mu)$ such that $\int_Af\ \mathrm d\mu<\infty$ for all $A\in\mathscr A$ of finite measure.

\begin{definition}\label{conditional intensity density}
Let $(\mathscr X,\mathscr A,\mu)$ be a $\sigma$-finite measure space. 
Let $(N_t)_{t\geq0}$ be a spatiotemporal point process, where to each event at time $t$, a spatial coordinate $x\in\mathscr X$ is attached. 
Let $(\mathcal F_t)_{t\geq0}$ be the natural filtration generated by $N$, i.e., $\mathcal F_t:=\sigma(N_s(A):s\leq t, A\in\mathscr A)$. 
Then any $(\mathcal F_t)_{t\geq0}$-predictable process $(t,x)\mapsto\lambda_t(x)$ is called a \emph{conditional intensity density}, if: $\lambda_t(\cdot)\in L^1_{\mathrm{loc}}$ for all $t\geq0$; and for all $A\in\mathscr A$ with $\mu(A)<\infty$, $N$ satisfies, as $\Delta t\downarrow0$,
\begin{align*}
\mathbb P(N_{t+\Delta t}(A)-N_t(A)=0|\mathcal F_t)&=1-\int_A\lambda_t(x)\ \mathrm d\mu(x)\Delta t+o(\Delta t);\\
\mathbb P(N_{t+\Delta t}(A)-N_t(A)=1|\mathcal F_t)&=\int_A\lambda_t(x)\ \mathrm d\mu(x)\Delta t+o(\Delta t).
\end{align*}
\end{definition}
Note that the probabilistic structure of $N$ is determined uniquely by the conditional intensity density $\lambda_t(x)$. 
This can be seen by treating spatial coordinates as marks, and applying \cite{DaleyVereJones}, Proposition~7.3.IV. 
Also, it follows from \cite{DaleyVereJones}, Definition~7.3.II, that any \emph{regular} spatiotemporal point process (see \cite{DaleyVereJones}, Proposition~7.3.I) admits a conditional intensity density. 
Furthermore, note that when there is an event in a finite measure set $A\in\mathscr A$ at time $t$, the conditional intensity density can be rescaled to a probability density function determining the spatial coordinate: the spatial coordinate $x$ follows the law $\mathbb P(x\in\mathrm dx)=\lambda_t(\mathrm dx)/\lambda_t(A)$.

Before we define a graphon Hawkes process through its conditional intensity representation, we need some regularity assumptions. 

\begin{assumption}\label{ass1}
Let $(\mathscr X,\mathscr A=\mathcal B(\mathscr X),\mu)$ be a topological $\sigma$-finite measure space. % equipped with a topology such that $\mathscr A=\mathcal B(\mathscr X)$. 
Let $\lambda_\infty\in L^1_{\mathrm{loc}}(\mathscr X)$, $h\in L^1(\mathbb R_+)$, let $W:\mathscr X^2\to\mathbb R_+$ be a measurable (di)graphon, i.e., a (possibly non-symmetric) measurable function $(x,y)\mapsto W(x,y)$, and let $(x,y)\mapsto B_{xy}$ be a stochastic process on $\mathscr X^2$ that is separable w.r.t.\ the class $\mathcal U$ of open subsets of $\mathscr X^2$.  
Also assume that there are $p,q\in[1,\infty]$ with $p^{-1}+q^{-1}=1$ such that $W(\cdot,y)\in L^p_+(\mathscr X)$ and $\mathbb E[B_{\cdot y}]\in L^q_+(\mathscr X)$ a.s., for $\mu$-a.e.\ $y\in\mathscr X$.
\end{assumption}
The definition of \emph{separability} of a stochastic process can be found in \cite{Neveu}, \S III.4.

\begin{definition}[Conditional intensity density representation]\label{def graphon Hawkes cond dens}
Grant Assumption~\ref{ass1}. 
Furthermore, assume that we are given some history $\mathcal F_0$ on $(-\infty,0]$ of a marked spatiotemporal point process $N$ with $\mathscr X$-valued random locations. 
Let $\mathscr X\times[0,\infty)\to[0,\infty):(x,y)\mapsto f_x(y)$ be a measurable function, %with $f_x(0)=0$ for each $x\in\mathscr X$, 
and assume that $f_x(\cdot)$ is $c_x$-Lipschitz, for each $x\in\mathscr X$.
When $N$ is generated on $\mathbb R_+$ by the conditional intensity density
\begin{equation}\label{eq cond int dens def}
\lambda_t(x)=f_x\left(\lambda_\infty(x)+\sum_{\substack{(s,y,B_{xy})\in N\\s<t}}B_{xy}W(x,y)h(t-s)\right),
\end{equation}
we say that $N$ is a \emph{graphon Hawkes process} on $\mathbb R_+\times \mathscr X$.
\end{definition}

We illustrate the dynamics of the graphon Hawkes process in Figure~\ref{3DfidCID}, where a realization of the conditional intensity density for a linear, unmarked process on $\mathscr X=[0,1]$ is shown. 
Here, we use a graphon $W$ with values $W(x,y)$ depending negatively on the distance between $x$ and $y$ on the unit circle. %The self-exciting spatiotemporal dynamics create a figure that resemble a mountain range right-skewed in the time dimension.

In this paper, we mainly focus on \emph{linear} processes for which $f_x$ is the identity mapping, for all $x\in\mathscr X$.
We assume throughout that we work in this setting, unless explicitly stated otherwise.

\begin{remark}\label{remark multivariate hawkes}
When we take $\mathscr X$ compact, e.g., $\mathscr X=[0,1]$, and parameters $\lambda_\infty,B,W$ that are piecewise constant on the elements of some partition $\mathcal P^d$ of $[0,1]$ into $d$ disjoint measurable sets $[0,1]=\bigcup_{i=1}^d\mathscr X_i$, then the locations within the sets $\mathscr X_i$ are irrelevant, and the process is essentially a $d$-variate Hawkes process. 
We employ this idea in Section~\ref{section convergence results}.

Similarly, when $\mathscr X$ is unbounded, $\sigma$-finite and partitioned into finite measure sets $(\mathscr X_k)_{k\in\mathbb N}$, then a process with parameters constant on each $\mathscr X_k$ is essentially a Hawkes process on a countable network; cf.\ \cite{Delattre16}.
\end{remark}

\begin{figure}
\centering
   \includegraphics[width=1\linewidth]{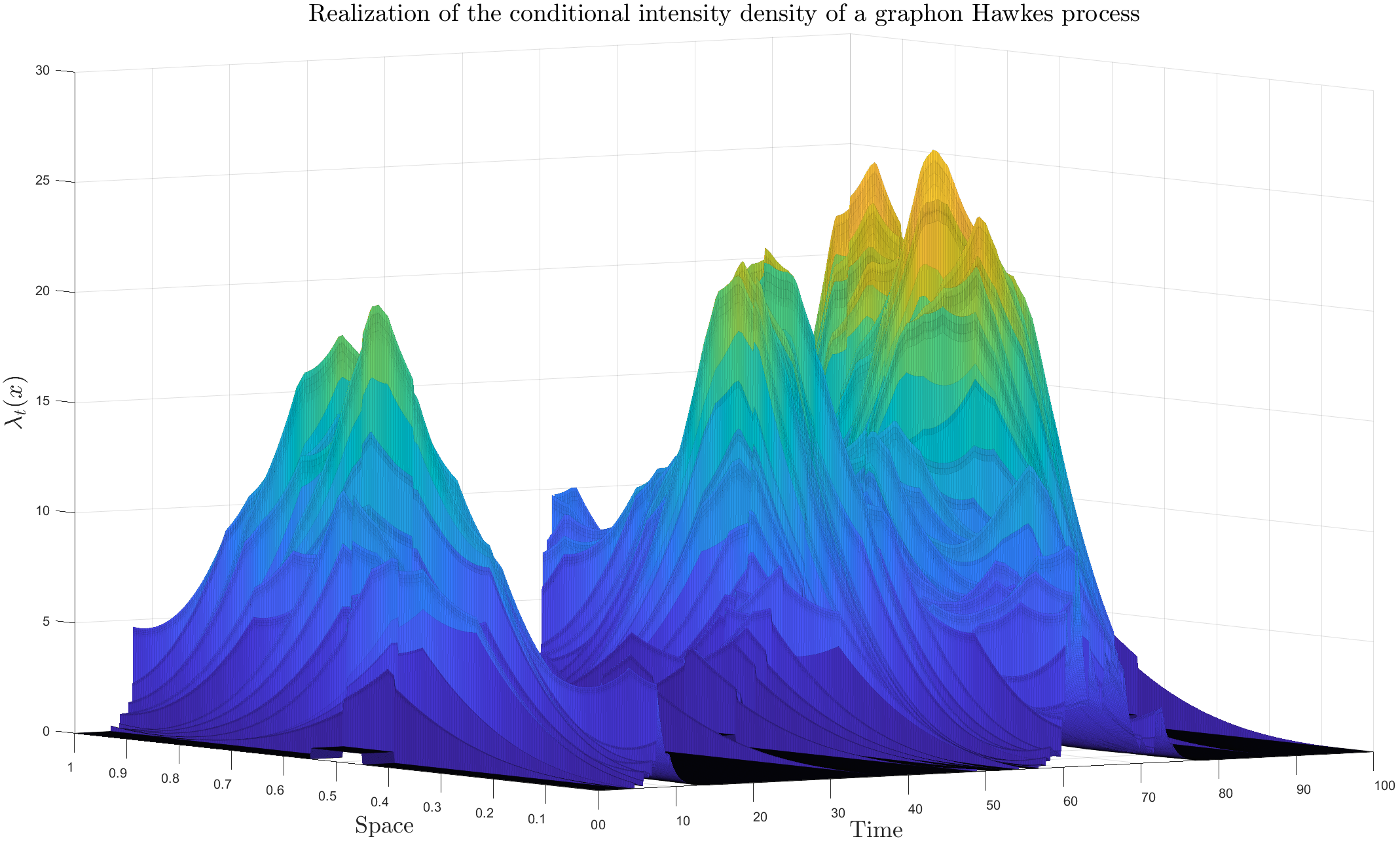}
   \caption{A realization of $\lambda_t(x)$ for a linear, unmarked model on $\mathscr X=[0,1]\ni x$ and $[0,100]\ni t$. 
   We use $\lambda_\infty(x)=\mathbf1\{x\in[.45,.55]\}$ %and $W(x,y)=30(1/2 - \min\{|x-y|,x+1-y,y+1-x\})^3$, i.e. 
   and a graphon $W(x,y)=30(1/2 - d(x,y))^3$, where $d(x,y)=\min\{|x-y|,x+1-y,y+1-x\}$ describes the distance on the unit circle.} 
   \label{3DfidCID}
\end{figure}
%, in the case of compact, finite measure spaces $(\mathscr X,\mathscr A,\mu)$. Here we restrict ourselves to compact, Euclidean $(\mathscr X,\mathscr A,\mu)=(\mathscr X,\mathcal B(\mathscr X),\mathrm{Leb}|_{\mathscr X})$ with $\mathscr X\subset\mathbb R^m$. For $\sigma$-finite measure spaces $(\mathscr X,\mathscr A,\mu)$, Definition \ref{clusterdef} below can be generalized in a straightforward manner, by partitioning $\mathscr X=\bigsqcup_{i\in\mathbb N}\mathscr X_i$ into sets of finite measure, i.e., $\mu(\mathscr X_i)<\infty$, for all $i\in\mathbb N$, and by generating events for each site separately. However, we do not need this level of generality in this article.
%Since we do not need this level of generality later on, we state Definition \ref{clusterdef}  for the simpler case $(\mathscr X,\mathscr A,\mu)=([0,1],\mathcal B[0,1],\mathrm{Leb}([0,1]))$.

As is the case for multivariate Hawkes processes, a linear conditional intensity density specification allows us to give a definition of the graphon Hawkes process in terms of a Poisson branching process. 
In the linear case, events $(t,x)$ arrive over time and space at rate
\begin{equation}
\lambda_t(x)=\lambda_\infty(x)+\sum_{\substack{(s,y,B_{xy})\in N\\s<t}}B_{xy}(s)W(x,y)h(t-s).\label{linearcondintsection2}
\end{equation} 
With this linear rate specification, one can say that there are two types of events: those generated by the \emph{baseline intensity} $\lambda_\infty$, which are coined \emph{immigrant events}, and those generated by the second term, which consists of intensity increases caused by previous events; therefore, those events are coined \emph{offspring events}. 
Note that the immigrant events create clusters that are i.i.d., modulo the time shift corresponding to the immigrant arrival times. 
Furthermore, each event creates offspring according to the same iterative procedure, with different clusters being independent of each other, implying that there is a branching structure, exhibiting self-similarity. 
This structure is made explicit by the following cluster representation, which is fully exploited in Section~\ref{section6 fixed point theorems}, when we derive fixed-point equations in the transform domain.

\begin{definition}[Cluster representation]\label{clusterdef}
Grant Assumption~\ref{ass1}.
Consider the time interval $[0,T]$, where $T\in[0,\infty]$. 
We generate the events of the graphon Hawkes process according to the following procedure.
\begin{enumerate}[label=(\roman*)]
\item Generate immigrant arrival times according to a time-homogeneous Poisson process $I(\cdot)$ on $[0,T]\times\mathscr X$ with intensity $\lambda_\infty(x)$, and write $(T^{(0)}_1,x_1^{(0)}),\ldots,(T^{(0)}_{I(t)},x^{(0)}_{I(t)})$ for the pairs of event times and locations strictly before time $t\in[0,T]$. %For each of those immigrant events, sample spatial coordinates (`locations') $x_1^{(0)},\ldots,x^{(0)}_{I(t)}\in\mathscr X$ according to the law having Lebesgue density proportional to $\lambda_\infty(\cdot)$. 
Sample the corresponding random mark functions $B_{\cdot x_i^{(0)}}(T_i^{(0)})$ according to the law of $B$. 
\item Set $n=0$. 
For each arrival with characteristics $(T_i^{(n)},x_i^{(n)},B_{\cdot x_i^{(n)}}(T_i^{(n)}))$, sample next-generation offspring events $(T^{(n+1)}_1,x_1^{(n+1)}),\ldots,(T^{(n+1)}_{K_i^{(n)}(t)},x_{K_i^{(n)}(t)}^{(n+1)})$ strictly before time $t$ according to a time-inhomogeneous Poisson process $K_i^{(n)}$ on $[0,T]\times\mathscr X$ having intensity at time $t\geq T_i^{(n)}$ equal to $B_{x x_i^{(n)}}(T_i^{(n)})W(x,x_i^{(n)})h(t-T_i^{(n)}).$  %Sample spatial coordinates $x_1^{(n+1)},\ldots,x_{K_i^{(n)}(t)}^{(n+1)}$ according to the law having Lebesgue density proportional to $B_{\cdot x_i^{(n)}}(T_i^{(n)})W(\cdot,x_i^{(n)})$. 
Sample the corresponding random mark functions $B_{\cdot x_i^{(n+1)}}(T_i^{(n+1)})$ according to the law of $B$. 
\item Iterate for $n\in\mathbb N$, %until no more events are created
obtaining the event sequence $$E_n(t):=\left\{\left(T_i^{(n)},x_i^{(n)},B_{\cdot x_i^{(n)}}\left(T_i^{(n)}\right)\right):T_i^{(n)}\leq t,\ n\in\mathbb N_0\right\}.$$
\end{enumerate}
Then the process $N(\cdot)$ given by $N(t)=\bigcup_{n\in\mathbb N_0}E_n(t)$ constitutes a graphon Hawkes process. 
\end{definition}

As for multivariate Hawkes processes, with linear models one can decompose the conditional intensity density into a baseline intensity, first-generation offspring intensity, and so on, implying that particles arrive at the same rates according to Definitions~\ref{def graphon Hawkes cond dens} and~\ref{clusterdef}, given that their spatial coordinates coincide. 
It is then intuitively clear that the intensity-based and the cluster representation-based definitions are equivalent for linear processes on compact $\mathscr X$ starting on an empty history, though we need to check that the different procedures used to sample the spatial coordinates are equivalent. 
This is the content of the next elementary lemma. 
Part (i), corresponding to the cluster representation definition, describes a sampling method where the intensity is decomposed, whereas part (ii), corresponding to the conditional intensity density definition, does not use such a decomposition. 
The proof of Lemma~\ref{sampling lemma} can be found in Appendix~\hyperref[app A]{A}.
\begin{lemma}\label{sampling lemma}
Grant Assumption~\ref{ass1}.
Suppose that we sample spatial coordinates using an intensity function $\lambda(x)=\lambda^1(x)+\lambda^2(x)$. 
Then the following two procedures are equivalent.
\begin{enumerate}[label=(\roman*)]
\item First, sample whether $\lambda^i$ caused the event, where $\lambda^i$ gets probability $\frac{\|\lambda^i(\cdot)\|_{L^1(\mathscr X)}}{\|\lambda^1(\cdot)+\lambda^2(\cdot)\|_{L^1(\mathscr X)}}$. 
Then, sample the location according to the law having Lebesgue density $\frac{\lambda^i(\cdot)}{\|\lambda^i(\cdot)\|_{L^1(\mathscr X)}}$.
\item Sample the location according to the law having Lebesgue density $\frac{\lambda(\cdot)}{\|\lambda(\cdot)\|_{L^1(\mathscr X)}}$.
\end{enumerate}
\end{lemma}

For most results in this paper (except for the existence, uniqueness and stability results of Section~\ref{section existence etc}), we assume compactness of our space $\mathscr X$ %, i.e., that our space is not `too big'. For similar reasons, for those results we need 
and uniform boundedness of the parameters of our model, as follows.

\begin{assumption}\label{ass cw cb}
It holds that $W(x,y)$, $\mathbb E[B_{xy}]$ are bounded by constants $C_W,C_B\geq0$, respectively, uniformly over $x,y\in\mathscr X$. 
Furthermore, the Lipschitz constants $c_x$ are bounded by $C_{\mathrm{Lip}}$, uniformly over $x\in\mathscr X$, 
and $\lambda_\infty(\cdot)\in L^\infty_{\mathrm{loc}}(\mathscr X)$.
\end{assumption}

\section{Existence, uniqueness and stability}\label{section existence etc}
%In Section~\ref{sectiondefs}, we defined the graphon Hawkes process in Definitions \ref{def graphon Hawkes cond dens} and \ref{clusterdef} through specifying its conditional intensity density and, in the linear case, through a cluster representation. 
In this section, we consider the general nonlinear model and the corresponding conditional intensity density representation given by Definition~\ref{def graphon Hawkes cond dens}, %which was stated for $\mathscr X$-valued locations,
where $(\mathscr X,\mathcal B(\mathscr X),\mu)$ is assumed to be some (topological) $\sigma$-finite measure space. 
We still need to prove that this is a `good' definition, in the sense that there exists a unique process $N$ satisfying the dynamics given by \eqref{eq cond int dens def}, and that such a solution is stable in distribution, meaning that there exists a unique stationary distribution, to which any `sufficiently regular' solution converges weakly, as $t\to\infty$. 
We consider the model under relatively weak stability conditions, reminiscent of those given in \cite{Stabilitypaper}, Theorem~7, for multivariate Hawkes processes. %In this setting, existence and uniqueness can be proved by a modification of the proof of \cite[Theorem 2]{Massoulie}. 
The proofs of this section leverage methodology from \cite{Massoulie} and can be found in Appendix~\hyperref[app A]{A}.
In Appendix~\hyperref[app A]{A}, we also show how to apply \cite{Massoulie}, Theorem~2, more directly under %somewhat
the stronger stability conditions considered there. 
When one considers the case of compact, Euclidean $\mathscr X$, e.g., $\mathscr X=[0,1]$, the proofs and methodology become somewhat simpler. 

In this section, we assume that the spatial coordinate takes values in some $\sigma$-finite measure space $(\mathscr X,\mathscr A,\mu)$, meaning that there exists a countable, measurable partition $\mathscr X=\bigsqcup_{i\in\mathbb N}\mathscr X_i$ into sets of finite measure, i.e., $\mu(\mathscr X_i)<\infty$, for all $i\in\mathbb N$. 
We refer to an element $\mathscr X_i$ of this partition as a \emph{site}.
Of course, $\mathbb N$ may be replaced by any other countable set, e.g., $\mathbb Z^m$. 
Note that when $\mathscr X=[0,1]$, there is no need to introduce multiple sites.

In the finite-dimensional case, existence, uniqueness and stability of Hawkes processes is related to the spectral radius of a matrix consisting of magnitudes of excitation, see \cite{Stabilitypaper}, Theorem~7. 
In the current infinite-dimensional setting, we are interested in the linear transformation
\begin{equation}\label{operator Thom marked general X}
T_{\mathrm{hom}}:L^1(\mathscr X,\mu)\to L^1(\mathscr X,\mu):f(\cdot)\mapsto\|h\|_{L^1(\mathbb R_+)}c_{\cdot}\int_{\mathscr X}\mathbb E[B_{\cdot y}]W(\cdot,y)f(y)\ \mathrm d\mu(y),
\end{equation}
which, heuristically, maps `a particle in $y$', i.e., the Dirac delta function $\delta_y$, into its expected (first-generation) offspring size density throughout the space $\mathscr X$. 
Recall that $c_x$ is the Lipschitz constant of $f_x$; see Definition~\ref{def graphon Hawkes cond dens}.
In the linear case where $f_x\equiv1$ and $c_x\equiv1$, the operator $T_{\mathrm{hom}}$ corresponds to the homogeneous part of the conditional intensity density, whence the subscript.
We work in the regime where $\rho(T_{\mathrm{hom}})$, the spectral radius of $T_{\mathrm{hom}}$, is smaller than $1$. 
Here, 
\begin{align*}\rho(T_{\mathrm{hom}})&:=\sup\{|\alpha|:\alpha\in\sigma(T_{\mathrm{hom}})\},\\
\sigma(T_{\mathrm{hom}})&:=\{\alpha\in \mathbb C:T_{\mathrm{hom}}-\alpha I\in B(L^1(\mathscr X,\mu))\text{ is not invertible}\},
\end{align*} 
with $I:L^1(\mathscr X,\mu)\to L^1(\mathscr X,\mu)$ the identity mapping, and with $B(\mathscr Z)$ the set of bounded linear operators on a Banach space $\mathscr Z$. 
The spectral radius is related to the operator norm through Gelfand's formula, see \cite{Conway}, Proposition\ VII.3.8:
\begin{equation}\label{gelfands formula}\rho(T_{\mathrm{hom}})=\lim_{n\to\infty}\|T_{\mathrm{hom}}^n\|^{1/n}=\inf_{n\in\mathbb N}\|T_{\mathrm{hom}}^n\|^{1/n}.
\end{equation}

%In the proof of Theorem \ref{thm stability general X spectral radius}, we assume uniform boundedness of $W(x,y)$, $\mathbb E[B_{xy}]$.

We now describe a two-step construction of (possibly nonlinear) graphon Hawkes processes. 
In the first step, for each $i\in\mathbb N$, we define a Poisson random measure $\bar N^i$ corresponding to site $\mathscr X_i$, which is used to construct events in $\mathscr X_i$. 
Next, for an arrival in $\mathscr X_i$, we use uniform marks to set up an acceptance-rejection method for determining the spatial coordinate $x\in \mathscr X_i$. 
To this end, for $i\in\mathbb N$, let $(L^i,\mathcal L^i,\mathcal Q^i)$ be a factor space of two spaces $(L^i_j,\mathcal L^i_j,\mathcal Q^i_j)$, $j=1,2$, the first space modeling the mark stochastic processes $B_{xy}^i$ $(x\in\mathscr X,y\in\mathscr X_i)$, and the second one supporting i.i.d.\ sequences of couples of independent $\mathrm{Uni}(\mathscr X_i,\mu_i)$ and $\mathrm{Uni}(0,1)$ marks; denote such a sequence by $U^i=(U^i_{1,k},U^i_{2,k})_{k\in\mathbb N}$. 
Here, the $\mathrm{Uni}(\mathscr X_i,\mu_i)$-distribution is defined by $\mathbb P(U^i_{1,k}\in\mathrm dx)=\mu(\mathrm dx)/\mu(\mathscr X_i)$. 
For later use, for $j=1,2$, let $(L_j,\mathcal L_j,\mathcal Q_j)$ be the product space $\prod_{i\in\mathbb N}(L^i_j,\mathcal L^i_j,\mathcal Q^i_j)$, and let $(L,\mathcal L,\mathcal Q)$ be the product of $(L_1,\mathcal L_1,\mathcal Q_1)$ and $(L_2,\mathcal L_2,\mathcal Q_2)$.

Our acceptance-rejection method is as follows. 
When there is an arrival at time $t$ in site $\mathscr X_i$, we have $\lambda_t(x)\in L^\infty(\mathscr X_i)$ a.s., since $T_{\mathrm{hom}}$ is a bounded operator under Assumption~\ref{ass1} (use Hölder's inequality).
In step $k$, $U^i_{1,k}$ is used as a proposal. 
We accept the proposal with probability proportional to $\lambda_t(U^i_{1,k})$. 
This means that we accept the proposal if and only if $$\frac{\lambda_t(U^i_{1,k})}{\|\lambda_t(\cdot)\|_{L^\infty(\mathscr X_i)}}\geq U^i_{2,k}.$$ 
If this condition is not satisfied, we reject, and try again for $k+1$. 
Now, given an arrival in site $\mathscr X_i$ at time $t$, the location within this site is determined by the realization of $U^i$, since $\lambda_t(\cdot)$ is known. 
We refer to this location as $X_t^i(U_t)$, with the understanding that $U_t$ is the realization of $U^i $ at time $t$.

\begin{remark}\label{Remark uni simulation}
In case $\mathscr X=[0,1]$ and $\mu$ is the Lebesgue measure on $[0,1]$, we need only a single $U\sim\mathrm{Uni}(0,1)$ random variable. 
Indeed, given an event at time $t$, the location can be determined by $$\mathbb P(X_t\in A)=\frac{\int_A\lambda_t(x)\ \mathrm dx}{\|\lambda_t(\cdot)\|_{L^1[0,1]}},$$ for $A\in\mathcal B[0,1]$, meaning that the location is given by 
\begin{equation}
X_t(U)=\inf\left\{z\in[0,1]:\frac{\int_0^z\lambda_t(x)\ \mathrm dx}{\int_0^1\lambda_t(x)\ \mathrm dx}\geq U\right\}. \label{locationuniforms}
\end{equation}
\end{remark}

We are now equipped to define the graphon Hawkes process through the Poisson random measures $\bar N^i$ on $\mathbb R\times L^i\times\mathbb R_+$ having intensity $\mathrm dt\times\mathcal Q^i(\mathrm dz)\times\mathrm ds$. 
Let $S_tN_-$ be the history of $N$ at time $t$; here, $N_-$ denotes the history at time $0$, and $S_t$ is the left-shift operator w.r.t.\ the first variable. 
Let $N^i(\mathrm dt\times\mathrm dz)$ be the process counting events in site $i$, with $z\in L$. 
Given an initial condition $S_0N_-$, we assume the following dynamics on $\mathbb R_+$:
\begin{align}
N^i(\mathrm dt\times\mathrm dz)&=\bar N^i(\mathrm dt\times\mathrm dz\times[0,\psi(S_tN_-,i)]);\nonumber\\
%\equiv N(\mathrm dt\times\mathrm dz\times\mathrm di)
\psi(S_tN_-,i)&=\Lambda_t^i=\int_{\mathscr X_i}\lambda_t(x)\ \mathrm d\mu(x)=\|\lambda_t(\cdot)\|_{L^1(\mathscr X_i,\mu)};\label{N wrt PRM}\\
\lambda_t(x)&=f_x\left(\lambda_{\infty}(x)+\sum_{j\in\mathbb N}\sum_{\substack{(s,B_{xX_s^j(U_s)},U_s)\in N^j\\s<t}}B_{xX_s^j(U_s)}W(x,X_s^j(U_s))h(t-s)\right).\nonumber
\end{align}

For the definition of strongly regular point processes, which we use in the next result, we refer to \cite{Massoulie}, Definition~2. 
%The proof of Theorem~\ref{thm stability general X spectral radius} can be found in Appendix~\hyperref[app A]{A}.

\begin{theorem}\label{thm stability general X spectral radius}
Grant Assumption~\ref{ass1}. % and \ref{ass cw cb}.
Consider a graphon Hawkes process $N$ on a $\sigma$-finite measure space $(\mathscr X,\mathscr A,\mu)$ having (possibly nonlinear) conditional intensity density specification 
\begin{equation}
\lambda_t(x)=f_x\left(\lambda_\infty(x)+\int_{(-\infty,t)}\int_{\mathscr X}B_{xy}(s)W(x,y)h(t-s)\ \mathrm dN_s(y)\ \mathrm d\sigma(s)\right),
\end{equation} where $\sigma$ denotes the counting measure on $\mathscr X$.
Assume that we have a partition of $\mathscr X$ into finite, non-null measure sets $\bigsqcup_{i\in\mathbb N}\mathscr X_i$ that is such that the following conditions hold:
\begin{align}
C:=\sup_{t\geq0,i\in\mathbb N}\|f_\cdot(\lambda_\infty(\cdot)+\eta(t,\cdot))\|_{L^1(\mathscr X_i,\mu)}&<\infty\label{Cnonlinear};\\
\rho(T_{\mathrm{hom}})&<1;\label{rho specrad}
\end{align}
where 
\begin{equation}\label{def eta stability}
\eta(t,x):=\sum_{\substack{(s,y)\in N\\s<0}}B_{xy}W(x,y)h(t-s).
\end{equation} 
% where  for $x\in\mathscr X_i$
% \begin{equation}\eta_i(t,x):=\sum_{\substack{(s,y)\in N\\s<0}}h(t-s)B_{xy}W(x,y).
% \end{equation}
Then there exists a unique strongly regular solution $N$ to \eqref{N wrt PRM} such that $$\sup_{t\geq0,i\in\mathbb N}\mathbb E\psi(S_tN_-,i)<\infty.$$
\end{theorem}

\begin{remark}
In the linear case, \eqref{Cnonlinear} is equivalent to the following two conditions.
\begin{align}
\alpha:=\sup_{i\in\mathbb N}\int_{\mathscr X_i}\lambda_\infty(x)\ \mathrm d\mu(x)&<\infty;\label{alpha specrad}\\
\sup_{t>0,i\in\mathbb N}\int_{\mathscr X_i}\eta(t,x)\ \mathrm d\mu(x)&<\infty.\label{eta specrad}
\end{align}
\end{remark}

Besides the proof of Theorem~\ref{thm stability general X spectral radius}, we provide in Appendix~\hyperref[app A]{A} some results that more directly follow the work of \cite{Massoulie}. 
First, in Proposition~\ref{thstationarity mathscr X}, we give a stability result, which is \cite{Massoulie}, Theorem~4, specified to the current setting; there we also use \cite{Massoulie}, Remarks~3 and~4. 
Next, in Proposition~\ref{existence mathscr X}, we apply \cite{Massoulie}, Theorem~2, directly to prove existence and uniqueness under a stability condition %$\rho<1$ 
related to the operator norm of $T_{\mathrm{hom}}$, which we calculate to be
\begin{equation}\label{expression operator norm thom}
\|T_{\mathrm{hom}}\|=\|h\|_{L^1(\mathbb R_+)}\sup_{y\in\mathscr X}\int_{\mathscr X}c_x\mathbb E [B_{xy}]W(x,y)\ \mathrm d\mu(x).
\end{equation}
\begin{remark}\label{rmk5}
When we work with a linear model on a finite measure space $\mathscr X$, the parameter $\rho$ from equation \eqref{rho mathscr X} is given by $\rho=\|T_{\mathrm{hom}}\|$. 
In this case, condition~\eqref{rho mathscr X} implies $\rho(T_{\mathrm{hom}})<1$ (cf.~\eqref{rho specrad}) by Gelfand's formula.
\end{remark}

\begin{comment}
\begin{remark}\label{remark stability delattre}
Inspired by \cite{Delattre16}, Assumption 4, it is interesting to investigate whether existence and uniqueness can also be guaranteed under the following mixing conditon: Let $p:\mathscr X\to\mathbb R_+$ be a weight function, let $C>0$ be some constant, and assume that 
\begin{enumerate}[label=(\alph*)]
\item $\int_{\mathscr X}p(x)\lambda_\infty(x)\ \mathrm dx<\infty$;
\item for all $y\in\mathscr X$, $\int_{\mathscr X}p(x)\mathbb E[B_{xy}]W(x,y)\ \mathrm dx\leq Cp(y)$.
\end{enumerate}
Essentially, a higher likelihood of an event in $x$ induces a lower maximum allowed weight $p(x)$ assigned to a spatial coordinate $x$. We did not succeed in proving existence and uniqueness in this regime.
\end{remark}
\end{comment}

\section{Convergence of mutually exciting processes to graphon Hawkes processes}\label{section convergence results}
In this section, we show that the \emph{linear} graphon Hawkes process on a compact, convex, Euclidean spatial set $\mathscr X\subset \mathbb R^m$, which is an infinite-dimensional model, occurs as the limit of suitably chosen multivariate Hawkes processes on $d$ coordinates, as $d\to\infty$. 
More specifically, we show that, under some regularity assumptions, partitioning $\mathscr X$ and averaging the parameters of the graphon Hawkes process over the sets of this partition yields a piecewise constant graphon Hawkes process --- i.e., a multivariate Hawkes process, see Remark~\ref{remark multivariate hawkes} --- that is close to the original graphon Hawkes process, for large $d$. 
This averaging procedure can be followed for infinite Euclidean spaces $\mathscr X=\mathbb R^m$ as well, by partitioning the space into hyperrectangles of uniformly bounded size. %, in which case the model with averaged parameters relates to the model from \cite{Delattre16} instead of to a multivariate Hawkes process. 
The proofs of the present section, however, do not carry over to this setting, due to the possibility that the excitation of the prelimit `comes from infinity', meaning that an event may be caused by excitation stemming from an event that is arbitrarily far away.  %The results of this section demonstrate that the graphon Hawkes process generalizes both the multivariate Hawkes and a Hawkes process on a countable network, and that graphon Hawkes processes on compact, convex subsets of $\mathbb R^m$ can actually be obtained as the limit of multivariate Hawkes processes.

We start by introducing some definitions and assumptions (Section~\ref{sec:defass}), the prelimit model $(\tilde N^d)_{d\in\mathbb N}$ (Section~\ref{section prelimit}), and establish some required technical results (Section~\ref{sec:technical lemmas}).
Then we consider the \emph{annealed} case (Section~\ref{sec:ann}), where the prelimit is a multivariate Hawkes process with a complete weighted connectivity graph. 
Next, we treat the \emph{quenched} case (Section~\ref{sec:que}), where we sample edges from this complete connectivity graph, yielding a Hawkes process on a finite network with an unweighted, directed connectivity graph as prelimit. 
Note that for such a sampling procedure, we need $W$ to map into $[0,1]$, which can be accomplished by a rescaling such as $\breve W(x,y)\breve B_{xy}=\left(W(x,y)/C_W\right)\left(C_WB_{xy}\right)$, using Assumption~\ref{ass cw cb}.

After providing definitions, the prelimit model and some technical preliminaries on $[0,T]$, we work from Section~\ref{sec:ann} onwards on the finite time interval $[0,1]$. 
By obvious modifications, everything caries over to bounded time intervals $[0,T]$.

\subsection{Definitions and assumptions}\label{sec:defass}

To prove convergence, we use the notion of \emph{uniform convergence on compacts in probability} (ucp convergence) and we first need to define a metric on the space of marked spatiotemporal point processes. 
Next, %before defining our prelimit $(\tilde N^d)_{d\in\mathbb N}$ in the next subsection, 
we provide some assumptions.

\begin{definition}\label{ucp def}
Let $\mathscr X\subset\mathbb R^m$ be a convex, $m$-dimensional manifold.
Let $(N^k)_{k\in\mathbb N}$ be a sequence of spatiotemporal point processes on $[0,T]\times\mathscr X$, marked by i.i.d.\ realizations of a stochastic process $B$ on $\mathscr X^2$ satisfying Assumption~\ref{ass1}.  
Furthermore, let $N$ be another such process. 
Write $\mathscr Z$ for the space of realizations of such processes. 
Let $\mathfrak d(\cdot,\cdot)$ be a metric on $\mathscr Z$. 
We say that $N^k$ converges uniformly on compact sets in probability to $N$, as $k\to\infty$, if for every nondegenerate compact, convex $m$-dimensional manifold $\mathscr A\subset \mathscr X$ and 
for every $\delta>0$, 
\begin{equation}\label{UCP convergence}
\mathbb P\left(\sup_{\substack{\mathscr A'\subset\mathscr A\\\mathscr A'\textnormal{ compact, convex}}}\mathfrak d\left(N^k_{\mathscr A'}, N_{\mathscr A'}\right)>\delta\right)\to0,\qquad\text{as }k\to\infty,
\end{equation}
in which case we write $N^k\stackrel{\mathrm{ucp}}\longrightarrow N$. 
Here, $N_{\mathscr Y}$ denotes $N$ restricted to the spatial set $\mathscr Y$.
\end{definition}

We aim to define a metric with the property that two point processes are `far apart' whenever there is an event of one process that does not coincide with an event of the other; and such that two simultaneous events are `close' if and only if their spatial coordinates and their marks are close. 
This motivates the following metric.

\begin{definition}\label{metric on spatiotemporal}
%Let $\mathscr X\subset\mathbb R^m$ be a convex, $m$-dimensional manifold.
For some $T\geq0$, let $N,M$ be two spatiotemporal point processes on $[0,T]\times\mathscr X$, marked by i.i.d.\ realizations of a stochastic process $B$ on $\mathscr X^2$ satisfying Assumption \ref{ass1}. 
Write $\mathcal T_L:=\{t\in\mathbb R_+:\exists y,B\text{ s.t. }(t,y,B)\in L\}$ for the event times of $L=N,M$.
Then we define a metric on $\mathscr Z$ (see Definition~\ref{ucp def}) by
\begin{align}
\mathfrak d(N,M)&:=\sum_{\substack{t\in \mathcal T_N\cap\mathcal T_M\\(t,y^N,B^N)\in N\\(t,y^M,B^M)\in M}}\left(\|y^N-y^M\|_{\mathbb R^m}+\|B^N_{\cdot y^N}-B^M_{\cdot y^M}\|_{L^1(\mathscr X)}\right)\nonumber\\
&+\sum_{\substack{t\in\mathcal T_N\vartriangle\mathcal T_M\\(t,y,B)\in \mathcal T_N\cup\mathcal T_M}}\left(1+\|y\|_{\mathbb R^m}+\|B_{\cdot y}\|_{L^1(\mathscr X)}\right).\label{eq metric on spatiotemporal}
\end{align}
\end{definition}
\begin{remark}
The choice of the Euclidean norm on $\mathbb R^m$ is unimportant, since all norms on a finite-dimensional vector space are equivalent. 
Also, we could just as well work with the $L^p(\mathscr X)$-norm, $p\in(1,\infty)$, by stating an analog of Lemma~\ref{TVlemma4} below, using a Poincaré inequality in $L^p(\mathscr X)$ and assumptions on the $p$-variation of $B$. %We do not pursue this route here.
\end{remark}

%The metric of Definition \ref{metric on spatiotemporal} has the following interpretation. The distance between two processes is the sum of the distances between their events. For simultaneous events, the distance between two events is the spatial distance plus the distance between the marks. We compare each non-simultaneous event to $0$, meaning that we add $1$ to the distance, plus the `size' of each event, which means the norms of the spatial coordinate $y$ and of the mark $B$. 

\begin{assumption}\label{ass7}
Let $\mathscr X=[\boldsymbol a,\boldsymbol b]:=\prod_{i=1}^m[a_i,b_i]\subset\mathbb R^m$ be some nondegenerate hyperrectangle in $(\mathbb R^m,\mathcal B(\mathbb R^m),\mathrm{Leb}^m)$, where $\mathrm{Leb}^m$ is the Lebesgue measure on $(\mathbb R^m,\mathcal B(\mathbb R^m))$. % Let $\mathcal P$ be a partition of $\mathscr X$ into $d$ hyperrectangles $(\mathscr X_n)_{n\in[d]}$ satisfying
%$\inf_{n\in\mathbb N}\mu(\mathscr X_n)>0$, 
%$\min_{n\in[d]}\mathrm{Leb}^m(\mathscr X_n)>0$. %, and  $\sup_{n\in[d]}\mathrm{Leb}^m(\mathscr X_n)<\infty$.
Assume that we are given a sequence of partitions $(\mathcal P^d)_{d\in\mathbb N}$ of $\mathscr X$ into $d$ hyperrectangles $(\mathscr X_n^d)_{n\in[d]}$ with %$\mathrm{mesh}(\mathcal P^d)<\infty$ for all $d\in\mathbb N$ and 
$\mathrm{mesh}(\mathcal P^d)\to0$ monotonically as $d\to\infty$, where $\mathrm{mesh}(\mathcal P^d):=\max_{n\in[d]}\mathrm{diam}(\mathscr X_n^d)$. %Assume that each $\mathcal P^{d+1}$ is a refinement of $\mathcal P^d$. 
% Also assume that there is some $\gamma\in(0,1]$, independent of $d$, such that $\min_{n\in[d]}\mathrm{diam}(\mathscr X_n^d)\geq\gamma\ \mathrm{mesh}(\mathcal P^d)$ for all $d\in\mathbb N$.
\end{assumption}

Next, we need some regularity assumptions on the parameters of our graphon Hawkes process in terms of their \emph{total variation}; cf.\ \cite{Leoni}.

\begin{definition}\label{TVmultivariate}
Let $\Omega\subset\mathbb R^m$ be an open region, and let $f\in \mathcal L^1_{\mathrm{loc}}(\Omega)$. 
We define the total variation of $f$ in $\Omega$ by 
\begin{equation}\label{TVmultivariate eq}
\mathrm{Var}(f,\Omega):=\sup\left\{\sum_{i=1}^m\int_\Omega\frac{\partial\Phi_i}{\partial x_i}f\ \mathrm dx:\Phi\in C_c^1(\Omega,\mathbb R^m),\|\Phi\|_{L^\infty(\Omega)}\leq1\right\},
\end{equation}
where $C_c^1(\Omega,\mathbb R^m)$ is the set of continuously differentiable functions $\Phi:\Omega\to\mathbb R^m$. 
When working on equivalence classes, i.e., for $f\in L^1_{\mathrm{loc}}(\Omega)$, we set 
\begin{equation}\|f\|_{\mathrm{TV}(\Omega)}:=\mathrm{Var}(f,\Omega):=\inf_{\substack{g\in\mathcal L_{\mathrm{loc}}^1(\mathbb R^m)\\g=f\text{ a.e.}}}\mathrm{Var}(g,\Omega).\end{equation}
Write $\mathrm{BV}(\Omega)$ for the set of (equivalence classes of) integrable functions $f\in L^1_{\mathrm{loc}}(\Omega')$ defined on a subset $\Omega'\supset\Omega$ of $\mathbb R^m$ such that $f|_\Omega$ has finite total variation.
\end{definition}
 
\begin{remark}
For univariate functions $g:\mathbb R\to\mathbb R$, \eqref{TVmultivariate eq} is equivalent to 
\begin{equation}\label{def TV univariate}
 \mathrm{Var}(g,(a,b))=\sup\left\{\sum_{i=1}^n|g(x_{i+1})-g(x_i)|:n\in\mathbb N,a=x_0<x_1<\cdots<x_n=b\right\}.
\end{equation}
\end{remark}

\begin{assumption}\label{ass6}
Assume that $N$ is a linear graphon Hawks process defined on $\mathscr X=[\boldsymbol a,\boldsymbol b]$, satisfying Assumption~\ref{ass7}.
Suppose that $\|\lambda_\infty\|_{\mathrm{TV}(\mathscr X)}<\infty$.
%there exists some constant $C_{\mathrm{TV}}\geq0$ such that $\|\lambda_\infty\|_{\mathrm{TV}[\mathbf a',\mathbf b']}\leq C_{\mathrm{TV}}\mathrm{Leb}^m([\mathbf a',\mathbf b'])$, for all $\mathbf a',\mathbf b'\in[\mathbf a,\mathbf b]$, $\mathbf a'<\mathbf b'$, meaning $a_i'<b_i'$ for all $i\in[m]$. 
Suppose that $B_{\cdot y}$ is of bounded variation a.s.\ for a.a.\ $y\in\mathscr X$, with $\sup_{y\in\mathscr X}\mathbb E\|B_{\cdot y}\|_{\mathrm{TV}(\mathscr X)}<\infty$. %Also assume that $\esssup_{y\in[\mathbf a,\mathbf b]}\mathbb E\|B_{\cdot y}\|_{\mathrm{TV}(\mathscr X)}<\infty$ for all $\mathbf a<\mathbf b$.
Furthermore, suppose that $W(\cdot, y)$ is of bounded variation for all $y\in\mathbb R$, with $\sup_{y\in \mathscr X}\|W(\cdot,y)\|_{\mathrm{TV}(\mathscr X)}<\infty$.
\end{assumption}

Finally, in Proposition~\ref{graphon + process convergence simultaneously} below, we require continuity of $W$, as follows. 
\begin{assumption}\label{ass2}
%Suppose $\mathscr X=[\mathbf a,\mathbf b]\subset\mathbb R^m$, and suppose  %that $W$ is in fact symmetric, i.e., $W(x,y)=W(y,x)$ for all $x,y\in[0,1]$. Also suppose 
Suppose that $W:\mathscr X^2\to\mathbb R_+$ is a.e.\ continuous.  
%More precisly, we suppose that there are $\mathbf x_i\in\mathscr X,i\in[Q]$, $Q\in\mathbb N$, such that $W$ is continuous on each hyperrectangle $[\mathbf a',\mathbf b']$ not containing any $\mathbf x_i$. 
%Similarly, suppose that the random function $s\mapsto\left((x,y)\mapsto B_{xy}(s)\right)$ is piecewise continuous, jointly in   $(x,y)$, a.s.
\end{assumption}

\subsection{Prelimit model}\label{section prelimit}
We consider a linear graphon Hawkes process $N$, defined through its conditional intensity density \eqref{linearcondintsection2},
with corresponding integral operator 
\begin{equation}
T_{\mathrm{hom}}:L^1(\mathscr X)\to L^1(\mathscr X):f(\cdot)\mapsto\|h\|_{L^1(\mathbb R_+)}\int_{\mathscr X}\mathbb E[B_{\cdot y}]W(\cdot,y)f(y)\ \mathrm dy.
\end{equation}
We define the intended prelimit $(\tilde N^d)_{d\in\mathbb N}$ by specifying conditional intensity densities.  %Since our space $\mathscr X$ is a \emph{compact} hyperrectangle by Assumption \ref{ass7}, there is no need to partition it into sites in order to satisfy the conditions of Theorem \ref{thm stability general X spectral radius}. %Note that when $\mathscr X$ is a (compact) hyperrectangle itself, choosing just one site suffices. 
Using the partition $\mathcal P^d$ from Assumption~\ref{ass7}, we partition the compact hyperrectangle $\mathscr X$ into further hyperrectangles $(\mathscr X_n^d)_{n\in[d]}$. % In the case where $\mathscr X$ is a hyperrectangle itself, the intended interpretation is that $\mathcal P^d$ consists of exactly $d$ hyperrectangles (although this intepretation is not necessary). In this case, $(\mathscr X_n^d)_{n\in\mathbb N}$ should be replaced by $(\mathscr X_n^d)_{n\in[d]}$, or all but $d$ sets can be replaced by empty sets. The analysis of this section simplifies under this regime, an all results carry over.
We define a prelimit model $\tilde N^d$ as the linear model having conditional intensity density 
\begin{equation}
\tilde \lambda^d_t(x)=\tilde\lambda^d_\infty(x)+\sum_{\substack{(s,y,\tilde B^d_{xy})\in \tilde N^d\\s<t}}\tilde B^d_{xy}(s)\tilde W^d(x,y)h(t-s),\label{cond int dens of prelimit}
\end{equation}
where the parameters are obtained by averaging over the elements $\mathscr X_n^d$ of $\mathcal P^d$:
\begin{align}
\tilde\lambda_{\infty}^d(x)&=\sum_{i=1}^d\frac{\mathbf 1_{\mathscr X_i^d}(x)}{\mathrm{Leb}^m(\mathscr X_i^d)}\int_{\mathscr X_i^d}\lambda_\infty(y') \ \mathrm dy';\label{lambdamulti}\\
\tilde B_{xy}^d(s)&=\sum_{i=1}^d\sum_{j=1}^d\frac{\mathbf1_{\mathscr X_i^d}(x)\mathbf1_{\mathscr X_j^d}(y)}{\mathrm{Leb}^m(\mathscr X_i^d)\mathrm{Leb}^m(\mathscr X_j^d)}\int_{\mathscr X_i^d\times \mathscr X_j^d}B_{x'y'}(s)\ \mathrm dx'\ \mathrm dy'\label{Bmulti};\\
\tilde W^d (x,y)&=\sum_{i=1}^d\sum_{j=1}^d\frac{\mathbf1_{\mathscr X_i^d}(x)\mathbf1_{\mathscr X_j^d}(y)}{\mathrm{Leb}^m(\mathscr X_i^d)\mathrm{Leb}^m(\mathscr X_j^d)}\int_{\mathscr X_i^d\times \mathscr X_j^d}W(x',y')\ \mathrm dx'\ \mathrm dy'.\label{Wmulti}
\end{align}  
Note that when we we average out the parameters of a graphon Hawkes process over the sets $\mathscr X_n^d$ as in \eqref{lambdamulti}--\eqref{Wmulti}, the location \emph{within} $\mathscr X_n^d$ of an event in $\mathscr X_n^d$ does not influence the dynamics of the process $\tilde N^d$. 
Since we have finitely many coordinates, the resulting process can be interpreted as a Hawkes process on a discrete network with nodes $n\in[d]$ corresponding to sites $\mathscr X_n^d$. 
This multivariate Hawkes process $N=(N^d_j)_{j\in[d]}$ on finitely many coordinates can be specified through its conditional intensity. 
Indeed, for each $i\in[d]$, we have a conditional intensity in coordinate $i$ given by 
\begin{equation}\label{conditional intensity multivariate model}
\Lambda_{t,i}^d=\lambda_{\infty,i}^d+\sum_{j=1}^d\sum_{\substack{s\in N_j^d\\s<t}}B_{ij}^d(s)W_{ij}^dh(t-s),
\end{equation}
where, cf.\ \eqref{lambdamulti}--\eqref{Wmulti},
\begin{align}
\lambda_{\infty,i}^d&=\int_{\mathscr X_i^d}\lambda_\infty(x) \ \mathrm dx;\label{cimm1}\\
B_{ij}^d(s)&=\frac1{\mathrm{Leb}^m(\mathscr X_j^d)}\int_{\mathscr X_i^d\times \mathscr X_j^d}B_{xy}(s)\ \mathrm dx\ \mathrm dy;\label{cimm2}\\
W_{ij}^d&=\frac1{\mathrm{Leb}^m(\mathscr X_i^d)\mathrm{Leb}^m(\mathscr X_j^d)}\int_{\mathscr X_i^d\times \mathscr X_j^d}W(x,y)\ \mathrm dx\ \mathrm dy.\label{cimm3}
\end{align} 
Note how \eqref{conditional intensity multivariate model}--\eqref{cimm3} describe the same model as \eqref{lambdamulti}--\eqref{Wmulti}, since integrating a constant density over site $i$ is identical to multiplying by $\mathrm{Leb}^m(\mathscr X_i^d)$. 
The connectivity structure of the model on $[d]$ is described by a complete weighted graph with weights $(W_{ij}^d)_{i,j\in[d]}$. 

We refer to the model with conditional intensity density given by \eqref{cond int dens of prelimit} as the \emph{annealed case}. 
In the \emph{quenched case}, the prelimit uses a simple graph, with edges sampled with probabilities equal to the weights $W_{ij}^d$. 
For this to make sense, we need $C_W\leq1$. 
As observed earlier, this can be accomplished by writing $W(x,y)B_{xy}=(W(x,y)/C_W)(C_WB_{xy})$, assuming $C_W<\infty$. 
More specifically, in the quenched case, we connect coordinates $i,j\in[d]$ in the prelimit model $\bar N^d$ if and only if $Z_{ij}^d=1$, where $Z_{ij}^d\sim\mathrm{Bernoulli}(W_{ij}^d)$. 
Then, for each $i\in[d]$, we have a conditional intensity in coordinate $i$ given by 
\begin{equation}\label{conditional intensity multivariate model quenched}
\Lambda_{t,i}^d=\lambda_{\infty,i}^d+\sum_{j=1}^d\sum_{\substack{s\in N_j^d\\s<t}}B_{ij}^d(s)Z_{ij}^dh(t-s).
\end{equation}
This is equivalent to a graphon Hawkes process $\bar N^d$ having conditional intensity density
\begin{equation}
\bar \lambda^d_t(x)=\tilde\lambda^d_\infty(x)+\sum_{\substack{(s,y,\tilde B^d_{xy})\in \bar N^d\\s<t}}\tilde B^d_{xy}(s)\tilde Z^d(x,y)h(t-s),\label{cond int dens of prelimit quenched}
\end{equation}
where
\begin{equation}
\label{Zmulti}\tilde Z^d(x,y)=\sum_{i=1}^d\sum_{j=1}^d\mathbf1_{\mathscr X_i^d}(x)\mathbf1_{\mathscr X_j^d}(y)Z_{ij}^d.
\end{equation}
We call $\tilde N^d$ on the weighted graph $\tilde W^d$ the \emph{annealed} model because its environment can be seen as the integrated version of that of the \emph{quenched} model $\bar N^d$ on a $\tilde W^d$-random graph.

%\begin{remark}
%The point process $\bar N^d$ resembles a mutually exciting point process on a stochatic block model (SBM).
%\end{remark}
 
%\begin{remark}\label{stabilityprelimit}
One has to be careful to prevent the multivariate Hawkes process induced by the graphon Hawkes process, as defined with the aid of equations \eqref{lambdamulti}--\eqref{Wmulti}, to be explosive.
It is easy to see that \eqref{alpha specrad} still holds for the model with averaged parameters, but \eqref{rho specrad} is more difficult to verify. 
Suppose that \eqref{rho specrad} holds for the original model, but that with high probability, $B$ is large where $W$ is small, and \textit{vice versa}. 
It may occur that after averaging, the resulting multivariate process is not stable any more. 
However, it can be argued that this problem does not arise for $d$ sufficiently large whenever $\mathrm{mesh}(\mathcal P^d)\to0$, as $d\to\infty$. 
This is shown in the next subsection.  
 %In any case, since we only consider the $\tilde N^d$ on bounded time intervals $[0,1]$, existence and uniqueness can still be guaranteed using the cluster representation, Definition \ref{clusterdef}: we do not need stability when we do not let $t\to\infty$.
%\end{remark}

\subsection{Technical results}\label{sec:technical lemmas}
In the proofs of the convergence results in the next subsections, we need some technical results. 
These results are collected in Lemmas~\ref{lemma offspring size} and~\ref{TVlemma4} below, which might be of independent interest. 
While new to the literature, they are primarily technical, hence we postpone their proofs to Appendix~\hyperref[app B]{B}.
Lemma~\ref{lemma offspring size} deals with the uniform boundedness of certain operator norms. 
Specifically, it demonstrates that if 
$\rho(T_{\mathrm{hom}})<1$, then for $d$ sufficiently large, and hence $\mathrm{mesh}(\mathcal P^d)$ sufficiently small, the approximating processes $\tilde N^d$ are stable, and in particular have finite expected total offspring, for every immigrant.

\begin{lemma}\label{lemma offspring size}
Grant Assumptions~\ref{ass1}, \ref{ass cw cb}, \ref{ass7} and \ref{ass6}. 
Let $N$ be a stable linear graphon Hawkes process satisfying the conditions of Theorem~\ref{thm stability general X spectral radius}, with multivariate approximating processes on weighted graphs $\tilde N^d$ as described by \eqref{lambdamulti}--\eqref{Wmulti}, and with corresponding integral operators $\tilde T_{\mathrm{hom}}^{(d)}$. 
Let $\delta:=\frac12(1-\rho(T_{\mathrm{hom}}))$. 
Then we can find some $D'\in\mathbb N$ such that for $d\geq D'$, $\tilde N^d$ satisfies the conditions of Theorem~\ref{thm stability general X spectral radius}, with $\rho(\tilde T_{\mathrm{hom}}^{(d)})\leq \rho(T_{\mathrm{hom}})+\delta<1$. 

Furthermore, we can find $D\in\mathbb N$ such that the expected cluster size $\mathbb E[Z_x^d]$ for $\tilde N^d$ of a particle in $x\in\mathscr X$ is uniformly bounded over $x\in\mathscr X=[\boldsymbol a,\boldsymbol b]$ and $d\geq D$. 
Denote this uniform bound by $\mathfrak K$.
\end{lemma}

The next result bounds the $L^1$-norm between a function and its piecewise constant approximation by the total variation of that function multiplied with the mesh of the corresponding partition. 
This result, which can be viewed as a suitable version of the Poincar\'e inequality, may be interesting in its own right, in that it has the potential to be applied more broadly. %To the best of our knowledge, this type of applications of the Poincaré inequality have not appeared in the literature. 

\begin{lemma}\label{TVlemma4}
Let $\mathscr X=\prod_{i=1}^m[a_i,b_i)\subset\mathbb R^m$ be some hyperrectangle. 
Assume that $\mathscr X$ shares its boundaries with elements from the partition $\mathcal P^d$ of $\mathbb R^m$ into $K\in\mathbb N$ bounded hyperrectangles $\mathscr X_i^d$, $i\in[K]$; if this is not the case, replace $\mathcal P^d$ by its coarsest refinement such that the boundaries of $\mathscr X$ are also boundaries of elements of $\mathcal P^d$.
Consider a piecewise constant approximation $$f^d|_{\mathscr X_n^d}\equiv\frac1{\mathrm{Leb}^m(\mathscr X_n^d)}\int_{\mathscr X_n^d}f(x)\ \mathrm dx$$ of $f\in L_{\mathrm{loc}}^1(\mathbb R^m)$, which is assumed to be of bounded variation over bounded sets.
%with $f^d$ being constant on each $\mathscr X_n^d$, with the value of $f^d$ on $\mathscr X_n^d$ being a measurable functional of $f|_{\mathscr X_n^d}$, satisfying $$f^d(\mathscr X_n^d)\in\left[\essinf_{x\in\mathscr X_n^d}f(x),\esssup_{x\in\mathscr X_n^d}f(x)\right];$$ for example, $f^d|_{\mathscr X_n^d}$ might be the infimum, the average, or the supremum over $\mathscr X_n^d$. 
Then it holds that $f^d\in L^1(\mathscr X,\mu)$, $f^d\to f$ a.e.\  as $\mathrm{mesh}(\mathcal P^d)\to0$, and finally,
\begin{equation}\label{L1 TV ineq A}
\|f^d-f\|_{L^1(\mathscr X,\mu)}\leq\frac12 \mathrm{Var}(f,\mathscr X)\mathrm{mesh}(\mathcal P^d).
\end{equation}
\end{lemma}

In the case of $\mathscr X\subset\mathbb R$, we can state a version of Lemma~\ref{TVlemma4} that is slightly more general, and admits an elementary proof. 
See Lemma~\ref{TVlemma3} in the appendix.

\begin{remark}
Instead of partitioning $\mathbb R^m$ into hyperrectangles, as prescribed by Assumption~\ref{ass7}, we could take a partition of $\mathbb R^m$ into \emph{convex} sets in such a way that the mesh of the partition tends to $0$, as $d\to\infty$. %In two dimensions, if we restrict ourselves to regular polyhedra, this is only possible if we work with triangles, squares, or hexagons. In higher dimensions, we could choose to triangulate our space, or we could work with what is called a \emph{honeycomb}. 
The particular choice of the tesselation does not alter our analysis. 
Indeed, in such a scenario, \cite{Acosta}, Theorem~3.2, still holds, hence the proof of Lemma~\ref{TVlemma4} can be adapted to different tesselations.
\end{remark}

\subsection{Annealed case}\label{sec:ann}
In this subsection, we prove convergence in the annealed case, where the prelimit $\tilde N^d$ is a multivariate Hawkes process on a complete weighted graph $\tilde W^d$, with parameters obtained as averages of the parameters of $N$ over the sets of a partition $\mathcal P^d$ of $\mathscr X=[\boldsymbol a,\boldsymbol b]$.
As a corollary, we show that any sufficiently regular sequence of mutually exciting point processes with $L^1$-convergent parameters converges to a graphon Hawkes process.
The quenched case is studied in the next subsection.

\begin{theorem}\label{annealed theorem}
Grant Assumptions~\ref{ass1}, \ref{ass cw cb}, \ref{ass7} and \ref{ass6}. 
Let $N$ be a linear graphon Hawkes process on $\mathscr X=[\boldsymbol a,\boldsymbol b]\subset\mathbb R^m$ and time interval $[0,1]$, starting on an empty history, satisfying the conditions  of Theorem \ref{thm stability general X spectral radius}, with multivariate approximating processes on weighted graphs $\tilde N^d$ as described by (\ref{cond int dens of prelimit})--(\ref{Wmulti}), with corresponding integral operators $\tilde T_{\mathrm{hom}}^{(d)}$. 
Using the metric $\mathfrak d$ from Definition \ref{metric on spatiotemporal}, it holds that $\tilde N^d\stackrel{\mathrm{ucp}}\longrightarrow N$, as $d\to\infty$. 
\end{theorem}
\begin{proof}
Under Assumptions~\ref{ass1}, \ref{ass cw cb}, \ref{ass7} and \ref{ass6}, as in Lemma~\ref{lemma offspring size}, for $\delta:=\frac12(1-\rho(T_{\mathrm{hom}}))$, let $D\in\mathbb N$ be such that: 
(i) for all $d\geq D$, $\rho(\tilde T_{\mathrm{hom}}^{(d)})\leq \rho(T_{\mathrm{hom}})+\delta<1$; and 
(ii) there exists $\mathfrak K<\infty$  bounding the expected cluster size for $N$ and the expected cluster size $\mathbb E[Z_x^d]$, uniformly over $x\in\mathscr X$, $d\geq D$.

The prelimit $\tilde N^d$ is a multivariate Hawkes process, i.e., a graphon Hawkes process with piecewise constant parameters on the sets of the partition $\mathcal P^d$. 
In particular, it is defined on the same space as $N$. 
Hence, it is possible to couple those processes by defining them w.r.t.\ the same Poisson random measure. 
We assume such coupled sample paths, and we bound the expected discrepancy between $N$ and $\tilde N^d$, in the metric $\mathfrak d$, for $d$ large. 
Since $\mathfrak d(N_{\mathscr A'},M_{\mathscr A'})\leq\mathfrak d(N_{\mathscr A},M_{\mathscr A})$ for $\mathscr A'\subset\mathscr A$ for this metric, it suffices to consider discrepancies on the whole space $[\boldsymbol a,\boldsymbol b]$. 
Fix $\delta,\epsilon>0$. 
We use the bound on the expected discrepancies to bound the probability of a discrepancy larger than $\delta$ by $\epsilon$; more formally, we show that there exists $d$ such that $\mathbb P\left(\mathfrak d(\tilde N^d, N)>\delta\right)<\epsilon$,
thus establishing ucp convergence. % from Definition \ref{ucp def}.

The discrepancy between two processes in the metric $\mathfrak d$ consists of several parts. 
For each simultaneous event, we compare the difference in location and the difference in $L^1(\mathscr X)$-norm between the marks corresponding to those locations. 
Next, for each non-simultaneous event the contribution to the $\mathfrak d$-distance equals unity, plus the distance of the spatial coordinate to the origin, plus the $L^1(\mathscr X)$-norm of the mark.

To bound the expected discrepancy in the $\mathfrak d$-metric, we consider differences in the parameters that contribute to the discrepancies between the two processes $N$ and $\tilde N^d$. 
A convenient way to do this is by employing the branching structure of the \emph{linear} graphon Hawkes process. Each contribution to the discrepancy in $\mathfrak d$-distance is of exactly one of the following types:
\begin{enumerate}[label=(\roman*)]
\item The difference $\|\lambda_\infty-\tilde\lambda^d_\infty\|_{L^1([\boldsymbol a,\boldsymbol b])}$ in baseline intensities creates a stream of nonsimultaneous events, each of which creates a cluster of events for either $N$ or $\tilde N^d$. 
\item Because of averaging, $B_{\cdot y}$ differs from $\tilde B_{\cdot y'}^d$,  and $W(\cdot ,y)$ differs from $\tilde W^d(\cdot ,y')$. 
For each simultaneous event with spatial coordinates $y$ and $y'$ of $N$ and $\tilde N^d$, respectively, this creates additional events. 
Again, each of those additional events generates clusters of events for either $N$ or $\tilde N^d$.
\item Even when we have a simultaneous event for both processes, the point from $\tilde N^d$ is located differently within $\mathscr X_i^d$ when compared to the point from $N$. 
Because all parameters of $\tilde N^d$ are constant on $\mathscr X_i^d$, this does not influence the future evolution of the process, although it affects $\mathfrak d$. 
Both the location and the mark may be different, causing discrepancies. 
\end{enumerate}
In the following, we formally bound (i)--(iii) separately.

Fix $\epsilon'>0$. 
We start by bounding the discrepancy caused by (i), i.e., the events generated by differences in baseline intensities, and offspring thereof. 
Suppose that $d_1\geq D$ is sufficiently large such that $\mathrm{mesh}(\mathcal P^d)\leq2\epsilon'/\|\lambda_\infty\|_{\mathrm{TV}(\mathscr X)}.$ 
Then, $\|\lambda_\infty-\tilde\lambda_\infty^d\|_{L^1(\mathscr X)}\leq\epsilon'$ by Lemma~\ref{TVlemma4}. 
Each of the events caused by the difference in baseline intensities creates a cluster of expected size bounded by $\mathfrak K$. 
Therefore, on the whole space $\mathscr X$, the difference in baseline intensities causes a stream of nonsimultaneous events with an expected count over our time frame $[0,1]$ bounded by $\epsilon'\|\lambda_\infty\|_{\mathrm{TV}(\mathscr X)}\mathfrak K$. 
The contribution to $\mathfrak d$ of each of those events can be bounded by $C_{\mathfrak d}:=1+|\boldsymbol a|\vee|\boldsymbol b|+C_B$. 
Hence, by choosing $$\epsilon'=\frac{\delta\epsilon}{3\|\lambda_\infty\|_{\mathrm{TV}(\mathscr X)}\mathfrak KC_{\mathfrak d}},$$ 
we can assure that the discrepancy in the metric $\mathfrak d$ caused by (i) can be bounded by $\delta\epsilon/3$.

Next, we bound the discrepancy in $\mathfrak d$-distance from (ii), i.e., discrepancies caused by differences between $B$ and $\tilde B^d$ and between $W$ and $\tilde W^d$, for each simultaneous event of $N$ and $\tilde N^d$. 
To this end, first note that the number of events for $N$ on the time interval $[0,1]$, and therefore the expected number of simultaneous events of $N$ and $\tilde N^d$ on $[0,1]$, can be bounded by $\alpha\mathfrak K$, where $\alpha:=\int_{\mathscr X}\lambda_\infty(x)\ \mathrm dx$, as in \eqref{alpha specrad}. 
For each of those events, there is a difference in mark and in location between $N$ and $\tilde N^d$. 
By the construction presented in Section~\ref{section existence etc}, we know that the locations of a simultaneous point for $N$ and $\tilde N^d$ lie in the same element $\mathscr X_n^d$ of the partition $\mathcal P^d$. 
Furthermore, since the parameters of $\tilde N^d$ are constant on the sets $\mathscr X_n^d$, it follows that the location of an event for $\tilde N^d$ \emph{within} $\mathscr X_n^d$ does not influence future dynamics, meaning that if we want to bound the expected number of events caused by differences between $B_{\cdot y}$ and $\tilde B_{\cdot y'}^d$ and between $W(\cdot ,y)$ and $\tilde W^d(\cdot ,y')$ for a simultaneous event for $N$ and $\tilde N^d$, we may in fact assume that $y=y'$. 
Now for each simultaneous event of $N$ and $\tilde N^d$, the difference between mark and graphon parameters and between locations causes an offspring with an expected count bounded by 
\begin{align*}
&\sup_{y\in\mathscr X}\int_{\mathscr X}\mathbb E\left|B_{xy}W(x,y)-\tilde B_{xy}^d\tilde W^d(x,y)\right|\ \mathrm dx\\
&\leq\sup_{y\in\mathscr X}\int_{\mathscr X}\mathbb E\left|B_{xy}-\tilde B_{xy}^d\right|W(x,y)\ \mathrm dx+
\sup_{y\in\mathscr X}\int_{\mathscr X}\mathbb E\left[\tilde B_{xy}^d\right]\left|W(x,y)-\tilde W^d(x,y)\right|\ \mathrm dx\\
&\leq\mathrm{mesh}(\mathcal P^d)\left(C_W\sup_{y\in\mathscr X}\mathbb E\|B_{\cdot y}\|_{\mathrm{TV}(\mathscr X)}+C_B\sup_{y\in\mathscr X}\|W(\cdot,y)\|_{\mathrm{TV}(\mathscr X)}\right)\big/2,
\end{align*}
using Assumption~\ref{ass cw cb} and Lemma~\ref{TVlemma4}. 
In turn, each of those events generates a cluster of further events with expected size bounded by $\mathfrak K$, with each event contributing to $\mathfrak d$ at most $C_{\mathfrak d}:=1+|\boldsymbol a|\vee|\boldsymbol b|+C_B$. 
Hence, selecting $d_2\geq D$ such that 
$$
\mathrm{mesh}(\mathcal P^d)<\frac{2\delta\epsilon}{3\alpha \mathfrak K^2\left(C_W\sup_{y\in\mathscr X}\mathbb E\|B_{\cdot y}\|_{\mathrm{TV}(\mathscr X)}+C_B\sup_{y\in\mathscr X}\|W(\cdot,y)\|_{\mathrm{TV}(\mathscr X)}\right)C_{\mathfrak d}},
$$
guarantees that the expected discrepancy between $N$ and $\tilde N^d$ during $[0,1]$ caused by (ii) above is less than $\delta\epsilon/3$.

Finally, we bound the expected discrepancy in $\mathfrak d$-distance caused by (iii), i.e., the difference in contributions to $\mathfrak d$ for simultaneous events of $N$ and $\tilde N^d$ through different locations and marks. 
Again, the expected number of simultaneous events of $N$ and $\tilde N^d$ on $[0,1]$ can be bounded by $\alpha\mathfrak K$. 
For each of those events, the locations are within the same element $\mathscr X_n^d$ of the partition $\mathcal P^d$, hence the difference in locations contributes at most $\mathrm{mesh}(\mathcal P^d)$ to the distance $\mathfrak d(N,\tilde N^d)$, while by Lemma~\ref{TVlemma4} the difference in marks causes an expected discrepancy bounded by $\mathrm{mesh}(\mathcal P^d)\sup_{y\in\mathscr X}\mathbb E\|B_{\cdot y}\|_{\mathrm{TV}(\mathscr X)}/2$. 
Hence, selecting $d_3\geq D$ such that 
$$
\mathrm{mesh}(\mathcal P^d)<\frac{2\delta\epsilon}{3\alpha\mathfrak K(2+\sup_{y\in\mathscr X}\mathbb E\|B_{\cdot y}\|_{\mathrm{TV}(\mathscr X)})},
$$
guarantees that the expected discrepancy in $\mathfrak d$-distance between $N$ and $\tilde N^d$ caused by differences in locations and mark sizes is less than $\delta\epsilon/3$.

Combining the bounds for the contributions (i)--(iii), it follows that $\mathbb E[\mathfrak d(\tilde N^d,N)]<\delta\epsilon$ for $d\geq\max\{d_1,d_2,d_3\}$, hence by Markov's inequality, $\mathbb P\left(\mathfrak d(\tilde N^d, N)>\delta\right)<\epsilon$. 
\end{proof}

%\begin{remark}
The essential ingredients in the proof of Theorem~\ref{annealed theorem} are the $L^1$-convergence of the parameters of the prelimit to those of the graphon Hawkes process, and that the prelimit has uniformly bounded cluster sizes. 
Hence, we can state the following converse to Theorem~\ref{annealed theorem}, which holds for prelimits obtained in different ways than through~\eqref{cond int dens of prelimit}--\eqref{Wmulti} as well.

\begin{corollary}\label{conversecor}
Let $(N^d)_{d\in\mathbb N}$ be a sequence of marked $d$-variate Hawkes processes. 
Embed each of the $N^d$ in the space of graphon Hawkes processes by taking a sequence of partitions $\mathcal P^d$ of $\mathscr X$ satisfying Assumption~\ref{ass7}, letting coordinate $i$ correspond to $\mathscr X_i^d$, and taking the parameters $\tilde\lambda_\infty^d,\tilde W^d,\tilde B^d$ of the graphon Hawkes process $\tilde N^d$ corresponding to $N^d$ piecewise constant on $\mathscr X_i^d$, $\mathscr X_i^d\times\mathscr X_j^d$, respectively, and such that \eqref{conditional intensity multivariate model}--\eqref{cimm3} hold.

Suppose that $\mathscr X$, $\mathcal P^d$, $\lambda_\infty$, $W$ and $B$ can be chosen such that: %there exists a stochastic process $B$, separable w.r.t.\ the collection of open sets on $\mathscr X^2$; a function $\lambda_\infty\in L^1_+(\mathscr X)$; a graphon $W:\mathscr X^2\to\mathbb R_+$; and a $\delta\in(0,1)$, such that
\begin{enumerate}[label=(\roman*)]
\item $\lambda_\infty,W,B$ satisfy Assumptions~\ref{ass1} and~\ref{ass cw cb}; %The graphon Hawkes process $N$ with parameters 
\item $\tilde \lambda_\infty^d\to\lambda_\infty$ in $L^1(\mathscr X)$, as $d\to\infty$;
\item $\tilde W^d(\cdot,y)\to W(\cdot,y)$ and $\mathbb E[\tilde B^d_{\cdot y}]\to \mathbb E[B_{\cdot y}]$ in $L^1(\mathscr X)$, as $d\to\infty$, uniformly over $y\in\mathscr X$;
\item The expected cluster size $\mathbb E[Z_x^d]$ is uniformly bounded in $x\in\mathscr X, d\in\mathbb N$.
\end{enumerate}
%Then we have $L^1$-convergence of the parameters, and by (iii) it follows that cluster sizes can be bounded by $\sum_{n\geq0}\rho_\delta^n=\frac1\delta$, which provides the desired bound on the expected cluster sizes. 
Then, using the metric $\mathfrak d$ from Definition \ref{metric on spatiotemporal}, it holds that $(N^d)_{d\in\mathbb N}\stackrel{\mathrm{ucp}}\longrightarrow N$, as $d\to\infty$. 
\end{corollary}
\begin{remark}
%Conditions~(ii)-(iii) of Corollary~\ref{conversecor} are satisfied under Assumptions~\ref{ass cw cb}, \ref{ass7} and \ref{ass6}.
Condition~(iv) of Corollary~\ref{conversecor} is satisfied if we assume that there is a $\delta\in(0,1)$ such that $\|h\|_{L^1(\mathbb R_+)}\sup_{y\in\mathscr X}\int_{\mathscr X}\mathbb E [\tilde B^d_{xy}]\tilde W^d(x,y)\ \mathrm dx\leq1-\delta$ uniformly over $d\in\mathbb N$.
This implies that the expected cluster size is bounded uniformly by $\sum_{n\geq0}(1-\delta)^n=1/\delta$; see also Remark~\ref{rmk5}.
\end{remark}

\begin{comment}
To summarize the contents of this section:

{\leftskip=1cm\relax
 \rightskip=1cm\relax
 \noindent\newline
\emph{Any sufficiently regular graphon Hawkes process is the limit of a sequence of mutually exciting processes. On the other hand, any sequence of mutually exciting point processes with $L^1$-convergent parameters and some uniformly bounded stability parameter converges to a graphon Hawkes process.}
 \par}
\end{comment}

\subsection{Quenched case}\label{sec:que}
In this subsection, we consider the quenched case, where the prelimit $\bar N^d$ is a multivariate Hawkes process on a looped, directed graph $\mathbb G^d$ with edges sampled from $(W_{ij}^d)_{i,j\in[d]}$. %To make everything precise, and fit convergence from graphs to graphons in a classical framework (see e.g., \cite{Lovasz}), we assume $\mathscr X=[0,1]$.%, although the proof of Theorem \ref{quenched theorem} holds as well for $\mathscr X=[\mathbf a,\mathbf b]$ with continuous graphons or with piecewise continuous graphons with ordered discontinuity points.
\begin{theorem}\label{quenched theorem}
Work in the setting of Theorem~\ref{annealed theorem}, with $\mathscr X=[\boldsymbol a,\boldsymbol b]\subset \mathbb R^m$ and time interval $[0,1]$, but different from the annealed process of Theorem~\ref{annealed theorem}, consider the conditional intensity density $\bar\lambda_t^d(\cdot)$ of the quenched process $\bar N^d$ given by 
\begin{equation}\label{cond int dens of quenched prelimit}
\bar \lambda^d_t(x)=\tilde\lambda^d_\infty(x)+\sum_{\substack{(s,y,\tilde B^d_{xy})\in \bar N^d\\s<t}}\tilde B^d_{xy}(s)\tilde Z^d(x,y)h(t-s),
\end{equation} with parameters as defined in \eqref{lambdamulti}, \eqref{Bmulti} and \eqref{Zmulti}. %, meaning that edges between coordinates $i,j\in[d]$ are sampled according to the weighted graph $(W_{ij}^d)_{i,j\in[d]}$.
Then, using the metric $\mathfrak d$ from Definition~\ref{metric on spatiotemporal}, it holds that $\bar N^d\stackrel{\mathrm{ucp}}\longrightarrow N$, as $d\to\infty$. 
Here, the uniform convergence \emph{in probability} is also w.r.t.\ the realization of the connectivity graph.
\end{theorem}
\begin{proof}
As observed before, we may assume without loss of generality that $C_W\leq1$.
Fix $\delta,\epsilon>0$. 
Let $\tilde N^d$ denote the annealed prelimit on a weighted complete directed graph of Theorem~\ref{annealed theorem}, so that $\bar N^d$ is the quenched prelimit on a looped directed graph sampled from this complete graph. 
The idea of this proof is to argue that for large $d$, with high probability, no set $\mathscr X_n^d$ from $\mathcal P^d$ contains more than one event from $\tilde N^d$, and conditionally on this high probability event, the probabilistic behavior (including stability) of $\bar N^d$ is the same as that of $\tilde N^d$. 
Indeed, suppose that there is at most one event in each set $\mathscr X_n^d$. 
Given an arrival in some $\mathscr X_j^d\ni y$, first sampling an edge $ Z_{ij}^d\sim\mathrm{Bernoulli}( W_{ij}^d)$ to $\mathscr X_i^d\ni x$, and in case of an edge $(i,j)$ generating the offspring according to $\tilde B_{xy}^d(s)h(\cdot-s)$, is equivalent to thinning the increase of intensity by $\tilde W_{ij}^d$, i.e., to generating offspring according to $\tilde B_{xy}^d(s)\tilde W_{ij}^dh(\cdot-s)$.

To prove ucp convergence, let $\delta\in(0,1)$. 
For such $\delta$, on the event $\{\mathfrak d(\tilde N^d,N)\leq\delta\}$, the processes $\tilde N^d$ and $N$ only have simultaneous events, by definition of $\mathfrak d$. 
Suppose that the partition $\mathcal P^d$ %of $\mathscr X$ %=[\boldsymbol a,\boldsymbol b]$ 
has no element $\mathscr X_n^d$ in which more than one event for $N$ occurs during $[0,1]$. 
Then, conditionally on $\{\mathfrak d(\tilde N^d,N)\leq\delta\}$, the same property holds with $N$ replaced by $\tilde N^d$, since a simultaneous event of $N$ and $\tilde N^d$ has a location within the same element $\mathscr X_n^d$ of the partition $\mathcal P^d$. 
We focus on taking $d$ sufficiently large such that this property holds for $\mathcal P^d$ and $N$.

As in Theorem~\ref{annealed theorem}, the expected cluster size of $N$ is bounded by $\mathfrak K$, hence $N$ has an expected number of events on $[0,1]$ that is bounded by $\alpha\mathfrak K<\infty$. 
Thus, $N(\omega)$, the realization of $N$ during time interval $[0,1]$, is finite a.s. 
This implies that $\ell(\omega):=\inf_{\boldsymbol x,\boldsymbol y\in N(\omega)}\|\boldsymbol x-\boldsymbol y\|>0$ a.s. 
For $d\geq1$ such that $\mathrm{mesh}(\mathcal P^d)<\ell(\omega)/2$, all elements of $N(\omega)$ are in different elements $\mathscr X_n^d$ from the partition $\mathcal P^d$. 
This property then carries over to $d'\geq d$, since $\mathrm{mesh}(\mathcal P^d)\to0$ monotonically, as $d\to\infty$, by Assumption~\ref{ass7}. 
For $d\in\mathbb N$, let
$$
A_d:=\left\{\omega\in\Omega:\mathrm{mesh}(\mathcal P^d)<\frac{\ell(\omega)}2\right\}. 
$$
Then, $\bigcup_{d\in\mathbb N}A_d=\Omega$ (up to a null set), and $A_d\uparrow\Omega$ (up to a null set). 
Hence, by continuity of probability measures, $\mathbb P(A_d)\uparrow 1$ as $d\to\infty$. 
Select $d_4\in\mathbb N$ such that $\mathbb P(A_{d_4})\geq1-\epsilon$; then for $d\geq d_4$ it holds that $\mathbb P(A_d)\geq1-\epsilon$. 
In the proof of Theorem~\ref{annealed theorem}, we showed that $\mathbb P(\mathfrak d(\tilde N^d,N)\leq\delta)\geq1-\epsilon$ for $d\geq\max\{d_1,d_2,d_3\}$. 
Hence, for $d\geq\max\{d_1,d_2,d_3,d_4\}$, $$\mathbb P(\{\mathfrak d(\tilde N^d,N)\leq\delta\}\cap A_d)\geq1-2\epsilon.$$ 
Since conditionally on $A_d$, the probabilistic behavior of $\bar N^d$ and $\tilde N^d$ is identical, we have
\begin{align*}
\mathbb P(\mathfrak d(\bar N^d,N)\leq\delta)&\geq\mathbb P(\{\mathfrak d(\bar N^d,N)\leq\delta\}\cap A_d)=\mathbb P(\{\mathfrak d(\tilde N^d,N)\leq\delta\}\cap A_d)\geq1-2\epsilon,
\end{align*} 
which establishes the ucp convergence $\bar N^d\stackrel{\mathrm{ucp}}\longrightarrow N$ as $d\to\infty$.%; compare the first paragraph of the proof of  Theorem \ref{annealed theorem}.
\end{proof}

We have shown that the graphon Hawkes process can be seen as a continuous-space limit of multivariate Hawkes processes. 
The next proposition, the proof of which is postponed to Appendix~\hyperref[app B]{B}, shows that the graphon corresponding to $N$ also occurs as the limit of the connectivity graphs corresponding $\tilde N^d$.

\begin{proposition}\label{graphon + process convergence simultaneously}
Work in the setting of Theorem~\ref{quenched theorem}, with $\mathscr X=[\boldsymbol a,\boldsymbol b]=\prod_{n=1}^d[a_n,b_n]$, and grant Assumption~\ref{ass2}.
Let $N$ be a graphon Hawkes process on $W$ with approximating multivariate Hawkes processes $(\bar N^d)_{d\in\mathbb N}$ with edges sampled from $(W_{ij}^d)_{i,j\in[d]}$, where the parameters are taken piecewise constant on the intervals of the partition $\mathcal P^d$. 
This partition can be chosen in such a way that $\mathbb P$-a.s., the connectivity graph $\mathbb G^d$ of $\bar N^d$ converges to $W$ in the cut norm (see \cite{Lovasz}), while $\bar N^d$ converges ucp to $N$, as $d\to\infty$.
\end{proposition}

\begin{remark}
The construction of this subsection essentially works with digraphons and looped, directed graphs $\mathbb G^d$. 
If we intend to use `standard' graphons and simple graphs, the analysis of this section can be adapted in order to sample edges only for $i<j$.
\end{remark}

\section{Large-time behavior}\label{large time behavior section}
We now specialize to a graphon Hawkes process with spatial coordinates in $$(\mathscr X,\mathscr A,\mu)=([0,1],\mathcal B[0,1],\mathrm{Leb}|_{[0,1]}).$$ 
This is for legibility only: as is easily verified, all arguments in this section and in Section~\ref{section6 fixed point theorems} only rely on $\mathscr X$ being a compact subset of an Euclidean space $\mathbb R^m$, and in particular do not require $\mathscr X$ to be one-dimensional. 

We investigate the large-time behavior of the graphon Hawkes process. 
First, we prove a functional law of large numbers (FLLN) in the stable case where $\rho(T_{\mathrm{hom}})<1$, and show that the corresponding prelimit diverges in the case $\rho(T_{\mathrm{hom}})>1$, thus providing a dichotomy between these two cases. %As for other, less complicated processes with a branching structure, it is difficult to make statements about the boundary case $\rho(T_{\mathrm{hom}})=1$. 
Having established the FLLN in the stable case, we next prove a functional central limit theorem (FCLT) for the number of events $(N_t(A))_{t\geq0}$ for $A\in\mathcal B[0,1]$. 
We establish those large-time results with the aid of limit theorems for finite-dimensional Hawkes processes.
 
To obtain explicit results, we work in this section with \emph{unmarked} processes, i.e., we set $B\equiv1$ in \eqref{linearcondintsection2}. 
Our approach is to approximate $N$ by multivariate Hawkes processes $\tilde N^d$ using Theorem~\ref{annealed theorem}, to which we apply explicit limit theorems for the unmarked case from \cite{Bacry}.
% hence the conditional intensity density is given by
%\begin{equation}\label{cidFCLT}
%\lambda_t(x)=\lambda_\infty(x)+\sum_{\substack{(s,y)\in N\\ s<t}}W(x,y)h(t-s).
%\end{equation}
In principle, it is also possible to consider functional limit theorems for marked Hawkes processes. %, which raises the question whether the assumption $B\equiv1$ is necessary. 
Such limit theorems, however, often assume univariate marked processes (\cite{Horst, Karabash}). 
If we model multivariate processes as marked univariate processes, we enforce a single common intensity function with i.i.d.\ coordinates. 
Other papers --- e.g., \cite{Fierro} --- work with multivariate processes, but obtain less explicit results.

\subsection{Functional law of large numbers}\label{LT stable case}
 %and Proposition \ref{thstationarity mathscr X}, 
Under the condition $\rho(T_{\mathrm{hom}})<1$, we expect that $N_T[0,1]=\mathcal O(T)$, as $T\to\infty$, by Theorem~\ref{thm stability general X spectral radius}. %Therefore, it makes sense to consider the behavior of $N_T(\cdot)/T$ as $T$ becomes large. 
Therefore, we consider $N_T(A)/T$ for large $T$ and for $A\in\mathcal B[0,1]$, with the objective to establish an FLLN for $\left(N_{Tv}(A)/T\right)_{v\in[0,1]}$. 
For the ergodic average of the number of events in any Borel set $A\in\mathcal B[0,1]$, we will argue that 
$$\left(\frac{N_{Tv}(A)}T\right)_{v\in[0,1]}\to(v\bar\lambda(A))_{v\in[0,1]},$$ where $\bar\lambda(A):=\int_A\bar\lambda(x)\ \mathrm dx$, with $\bar\lambda(x)$ the \emph{expected stationary arrival rate density}, meaning that it solves the Fredholm integral equation of the second kind $\bar\lambda=\lambda_\infty +T_{\mathrm{hom}}(\bar\lambda)$, i.e., 
\begin{equation}\label{fredholmeq intensity}
\bar\lambda(x)=\lambda_\infty(x)+\|h\|_{L^1(\mathbb R_+)}\int_0^1W(x,y)\bar\lambda(y)\ \mathrm dy.
\end{equation}
This is motivated by the following heuristic, supposed to hold for $T$ large and for $A\in\mathcal B[0,1]$:
\begin{align*}
\frac{\mathbb E[N_{Tv}(A)]}T&=\frac1T\int_A\int_0^{Tv}\mathbb E[\lambda_t(x)]\ \mathrm dt\ \mathrm dx\\
&=\frac1T\int_A\int_0^{Tv}\left[\lambda_\infty(x)+\int_0^1\int_0^tW(x,y)h(t-s)\mathbb E[\lambda_s(y)]\ \mathrm ds\ \mathrm dy\right]\ \mathrm dt\ \mathrm dx\\
&\approx\frac1T\int_A\int_0^{Tv}\left[\lambda_\infty(x)+\int_0^1W(x,y)\bar\lambda(y)\ \mathrm dy\ \|h\|_{L^1(\mathbb R_+)}\right]\ \mathrm dt\ \mathrm dx\\
&=\frac1T\int_A\int_0^{Tv}\bar\lambda(x)\ \mathrm dt\ \mathrm dx=v\bar\lambda(A),
\end{align*}
where in stationarity we would expect $\mathbb E[\lambda_s(y)]$ to be independent of $s$, in which case we denote this quantity by $\bar\lambda(y)$.
In the remainder of this subsection, we make this heuristic exact. 
%In our approach, we approximate $N$ by multivariate Hawkes processes $\tilde N^d$ using Theorem~\ref{annealed theorem}, in order to apply \cite{Bacry}, Theorem~1, to the prelimit $\tilde N^d$. 
%As mentioned before, to apply the explicit results of \cite{Bacry}, we are working with a system without marks, i.e., we set $B\equiv 1$.

\begin{theorem}\label{FLLN theorem}
Grant Assumptions~\ref{ass1}, \ref{ass cw cb}, \ref{ass7} and \ref{ass6}. 
Let $N=(N_t(A))_{t\in\mathbb R_+,A\in\mathcal B[0,1]}$ be an unmarked (i.e., $B\equiv1$) linear graphon Hawkes process on $\mathscr X=[0,1]$ and time interval $\mathbb R_+$, starting on an empty history, and satisfying the conditions of Theorem~\ref{thm stability general X spectral radius}.  
In particular, suppose $\rho(T_{\mathrm{hom}})<1$, where in the present setting
\begin{equation}\label{operator Thom no marks}
T_{\mathrm{hom}}:L^1[0,1]\to L^1[0,1]:f(\cdot)\mapsto\|h\|_{L^1(\mathbb R_+)}\int_0^1W(\cdot,y)f(y)\ \mathrm dy.
\end{equation} 
Let $\bar\lambda$ be defined in \eqref{fredholmeq intensity}.
Then,  
\begin{equation}
\label{FLLN for graphon Hawkes}
\sup_{v\in[0,1]}\left|\frac{N_{Tv}(A)}{T}-v\bar\lambda(A)\right|\stackrel{\mathbb P}\longrightarrow0,\quad\text{as }T\to\infty.
\end{equation}
\end{theorem}

\begin{proof}
The proof is structured as follows. 
We first approximate the graphon Hawkes process by $d$-dimensional Hawkes processes using Theorem~\ref{annealed theorem}; 
then we use an existing FLLN for this prelimit; 
and we demonstrate that the limit quantities for graphon Hawkes and multivariate Hawkes processes are close, for large $d$. 

Thus, we first approximate the graphon Hawkes process $N$ by processes $\tilde N^d$ on weighted graphs, described by \eqref{cond int dens of prelimit}--\eqref{Wmulti}, having constant parameters on each element $\mathscr X_n^d$ ($n\in[d]$) of the partition $\mathcal P^d$ of $[0,1]$ into $d$ intervals, with $\mathrm{mesh}(\mathcal P^d)\to0$ as $d\to\infty$. 
Denote the corresponding integral operators by $\tilde T_{\mathrm{hom}}^{(d)}$.
By Lemma~\ref{lemma offspring size} (which requires Assumptions~\ref{ass1}, \ref{ass cw cb}, \ref{ass7} and \ref{ass6}), we can find some $D\in\mathbb N$ such that for $d\geq D$, $\tilde N^d$ is stable with $\rho(\tilde T_{\mathrm{hom}}^{(d)})\leq \rho(T_{\mathrm{hom}})+\delta<1$, where $\delta:=\frac12(1-\rho(T_{\mathrm{hom}}))$, and such that the expected cluster size $\mathbb E[Z_x^d]$ for $\tilde N^d$ of a particle in $x\in[0,1]$ is bounded by $\mathfrak K<\infty$, uniformly over $x$ and $d\geq D$. 
This means that the FLLN \cite{Bacry}, Theorem~1, holds for $\tilde N^d$, for all $d\geq D$.

For $A\in\mathcal B[0,1]$,
let $(N_{t}(A))_{t\geq0}$ be the simple counting process keeping track of the number of points in $A$ over time. 
Let $\tilde K^d=(\|h\|_{L^1(\mathbb R_+)}W_{ij}^d)_{i,j\in[d]}$ (see \eqref{cimm3} and \cite{Bacry}, Assumption~(\textbf{A1})), and let $I_d$ be the $d\times d$ identity matrix.
Then it holds that 
\begin{align}\label{FLLN 1eq}
\sup_{v\in[0,1]}\left|\frac{N_{Tv}(A)}{T}-v\bar\lambda(A)\right|&\leq\sup_{v\in[0,1]}\left|\frac{N_{Tv}(A)}{T}-\frac{\tilde N^d_{Tv}(A)}{T}\right|\\
&\ +\sup_{v\in[0,1]}\left|\frac{\tilde N^d_{Tv}(A)}{T}-v(I_d-\tilde K^d)^{-1}\tilde\lambda_\infty^d(A)\right|\label{FLLN 2eq}\\
&\ +\sup_{v\in[0,1]}\left|v(I_d-\tilde K^d)^{-1}\tilde\lambda_\infty^d(A)-v\bar\lambda(A)\right|,\label{FLLN 3eq}
\end{align}
where % $\tilde N^d_{Tv}(A):=\sum_{n=1}^d\frac{|\mathscr X_n^d\cap A|}{|\mathscr X_n^d|}\tilde N^d_{Tv}(\mathscr X_n^d)$ 
$(I_d-\tilde K^d)^{-1}\tilde\lambda_\infty^d(A):=\int_A((I_d-\tilde K^d)^{-1}\tilde\lambda_\infty^d)(x)\ \mathrm dx$.
We aim to prove that the left-hand side of the inequality in the above display goes to $0$ in probability as $T\to\infty$, by first selecting $\bar d$ large enough such that the right-hand side of \eqref{FLLN 1eq} and \eqref{FLLN 3eq} are small for $d\geq\bar d$, independently of $T\geq1$, and next selecting $T$ large enough such that \eqref{FLLN 2eq} is small.

Fix $\epsilon,\kappa\in(0,1)$. 
We start by bounding \eqref{FLLN 1eq}. 
Suppose that $T\geq1$.  
Let $\mathfrak d(\cdot,\cdot)$ be the metric from Definition~\ref{metric on spatiotemporal}. 
If $\mathfrak d(N_{T},\tilde N^d_{T})<2T\epsilon'$, then there are fewer than $2T\epsilon'$ non-simultaneous events on $[0,T]$ for both processes, and this leads to the bound $$\sup_{v\in[0,1]}\left|\frac{N_{Tv}(A)}{T}-\frac{\tilde N^d_{Tv}(A)}{T}\right|<2\epsilon'.$$ 
Hence, we focus on establishing that $\mathfrak d(N_{T},\tilde N^d_{T})<2T\epsilon'$. 
To this end, let $M$ be a process corresponding to $N$, having the same immigrant streams of events $I_1<I_2<\cdots$, but where to each immigrant $I_k$ arriving at time $t$, at this time $t$ we immediately attach the immigrant's complete offspring $O_k\subset(t,\infty)$. 
More precisely, $M_T=\bigcup_{I_k\leq T}O_k$. 
Define $\tilde M^d$ similarly through $\tilde N^d$.
This means that the event times of $M_T,\tilde M^d_T$ are discrete subsets of $[0,\infty)$, marked by discrete subsets of $(t,\infty)$ for an immigrant arrival at time $t$. 
Informally, $M$, $\tilde M^d$ consist of the immigrant streams of $N$, $\tilde N^d$, with each immigrant replaced by a batch of size equal to the whole cluster corresponding to that immigrant.
%hence $M_T,\tilde M^d_T$ are not necessarily subsets of $[0,T]$, as $N_T,\tilde N^d_T$ are. 

Let $D$ be the process defined through $D_T=\bigcup_{I_k\leq T}(O_k\cap(T,\infty))$, where $(I_k)_{k\in\mathbb N}$, $(O_k)_{k\in\mathbb N}$ are the immigrant times and offspring sets corresponding to $N$; let $D^d$ be defined similarly through $N^d$. 
Then we have, a.s.,
%Since we do not have to `wait' until the whole cluster arrives, the discrepancy between the two processes occurs faster for $M$ than for $N$, hence 
\begin{equation}
\mathfrak d\left(N_{T},\tilde N^d_{T}\right)\leq\mathfrak d\left(N_{T},\tilde N^d_{T}\right)+\mathfrak d\left(D_{T},\tilde D^d_{T}\right)=\mathfrak d\left(M_{T},\tilde M^d_{T}\right).
\end{equation}
Furthermore, writing $M_{A}$ for the process only recording the immigrants arriving in $A\in\mathcal B(\mathbb R_+)$ --- meaning that $M_T\equiv M_{[0,T]}$ ---, we easily obtain the a.s.\ bound 
\begin{equation}
\mathfrak d\left(M_{T},\tilde M^d_{T}\right)\leq\sum_{i=1}^{\lceil T\rceil}\mathfrak d\left(M_{[i-1,i)},\tilde M^d_{[i-1,i)}\right).
\end{equation}
Next, note that clusters arrive independently over time, and that the clusters themselves are i.i.d., modulo the time shift corresponding to the immigrant arrival time. 
That is, $(M_{[i-1,i)},\tilde M^d_{[i-1,i)})_{i\in\mathbb N}$ represent i.i.d.\ random sets. 
Furthermore, by the proof of Theorem~\ref{annealed theorem}, we can find $d_1\geq D$ such that, for $d\geq d_1$, the expectation of $\mathfrak d(M_{[i-1,i)},\tilde M^d_{[i-1,i)})$ can be bounded by $\epsilon'$. 
Hence, for $d\geq d_1$, 
$\mathfrak d(N_{T},\tilde N^d_{T})$ is bounded in expectation by $\lceil T\rceil \epsilon'\leq2T\epsilon'$, thus, independently of $T\geq1$, $$\sup_{v\in[0,1]}\left|\frac{N_{Tv}(A)}{T}-\frac{\tilde N^d_{Tv}(A)}{T}\right|$$ is bounded in expectation by $2\epsilon'$. 
Choose $\epsilon'=\epsilon\kappa/12$. 
Markov's inequality implies that with probability of at least $1-\kappa/2$, \eqref{FLLN 1eq} is smaller than $\epsilon/3$. 

We now bound \eqref{FLLN 3eq}. 
To this end, note that the solution $\bar\lambda$ to a Fredholm integral equation of the second kind depends continuously on the initial conditions, whenever there exists a unique solution; see \cite{Integral equations book}, Corollary~9.3.12. 
This can be guaranteed under the condition $\rho(T_{\mathrm{hom}})<1$, since then $\bar\lambda$ as defined in \eqref{fredholmeq intensity} can be expressed by a Liouville-Neumann series: \begin{equation}\label{LiouvilleNeumann}\bar\lambda=\sum_{n\geq0}T_{\mathrm{hom}}^n\lambda_\infty,\end{equation} 
which is in $L^1[0,1]$ by Gelfand's formula.
The continuous dependence of the solution $\bar\lambda$ on the initial conditions means that we can find $\eta>0$, such that  
\begin{equation}\label{twoconditions}\|\lambda_\infty-\tilde\lambda_\infty^d\|_{L^1[0,1]}<\eta\quad\text{and}\quad\|h\|_{L^1(\mathbb R_+)}\|W-\tilde W^d\|_{L^1[0,1]^2}<\eta\end{equation} 
imply that $$\|(I_d-\tilde K^d)^{-1}\tilde\lambda_\infty^d-\bar\lambda\|_{L^1[0,1]}<\frac\epsilon3.$$ 
Note that $(I_d-\tilde K^d)^{-1}\tilde\lambda_\infty^d$ is the unique solution to $f=\tilde\lambda_\infty^d +\tilde T_{\mathrm{hom}}^d(f)$. 
By Lemma~\ref{TVlemma4}, there exists $d_2\geq D$ such that \eqref{twoconditions} holds for all $d\geq d_2$.  
Hence, for such $d$, \eqref{FLLN 3eq} is bounded by $\epsilon/3$.

Finally, given $d:=d_1\vee d_2$, we need to choose $T$ sufficiently large to make  \eqref{FLLN 2eq} small.
We achieve this by applying \cite{Bacry}, Theorem~1, to a multivariate Hawkes process on $2d$ coordinates corresponding to the sets $\mathscr X_n^d\cap A$ and $\mathscr X_n^d\cap A^\mathsf{c}$, $n\in[d]$, with equal parameters on the coordinates corresponding to $\mathscr X_n^d\cap A$ and $\mathscr X_n^d\cap A^\mathsf{c}$, computed according to \eqref{cond int dens of prelimit}--\eqref{Wmulti}. 
Note that \cite{Bacry}, Theorem~1, uses the Euclidean norm on $\mathbb R^d$.  
However, as all norms on a finite-dimensional vector space are equivalent, their result holds for the $\ell^1$-norm as well. 
It follows that \eqref{FLLN 2eq}, which is bounded by its $\ell^1$-norm over the $2d$ coordinates, converges to $0$ in probability as $T\to\infty$. 
In other words, we can find $T^*\geq1$ such that for $T\geq T^*$ it holds that, with probability of at least $1-\kappa/2$, the expression in \eqref{FLLN 2eq} is bounded by $\epsilon/3$.

Now we use our choice of $d$ to combine the bounds for \eqref{FLLN 1eq}--\eqref{FLLN 3eq}. 
Then, for all $T\geq T^*$, 
\[\mathbb P\left(\sup_{v\in[0,1]}\left|\frac{N_{Tv}(A)}{T}-v\bar\lambda(A)\right|<\epsilon\right)\geq1-\kappa.\qedhere 
\]
\end{proof}

%Having proved an FLLN in the stable case where $\rho(T_{\mathrm{hom}})<1$, 
We proceed to the case $\rho(T_{\mathrm{hom}})>1$, where we expect $N_T(A)/T\to\infty$ a.s., as $T\to\infty$, for each $A\in\mathcal B[0,1]$ of positive Lebesgue measure. 
We prove this claim formally below. 
First, however, we provide some heuristics. 
By Gelfand's formula, we have
$$\rho(T_{\mathrm{hom}})=\lim_{n\to\infty}\|T_{\mathrm{hom}}^n\|^{1/n}=\inf_{n\in\mathbb N}\|T_{\mathrm{hom}}^n\|^{1/n}>1.$$ 
This implies (see Proposition~\ref{app3prop} in Appendix~\hyperref[app C]{C}) the existence of some $g\in L^1_+[0,1]$ with $\|g\|_{L^1[0,1]}=1$ such that $\|T^n_{\mathrm{hom}}g\|_{L^1[0,1]}\to\infty$ as $n\to\infty$. 
Hence, if we can find some (Poisson) stream of particles with random spatial locations having probability density $g$, then each of those particles has infinite expected offspring, rendering the system unstable. 
This heuristic, however, while insightful, does not directly lead to a formal proof; see Proposition~\ref{heuristic does not work} in Appendix~\hyperref[app C]{C} for details. 
To prove our divergence claim, we follow a different route, for which we need some assumptions. 
\begin{assumption}\label{ass9}
Assume that $\alpha=\|\lambda_\infty\|_{L^1[0,1]}>0$. 
Furthermore, assume that the time from the birth of a parent to the birth of the child has finite expectation, i.e., $\int_0^\infty th(t)\ \mathrm dt<\infty$. 
Next, let 
\begin{equation}
W^{(k)}(x,y):=\int_0^1\cdots\int_0^1W(x,x_1)\cdots W(x_{k-1},y)\ \mathrm dx_1\cdots\mathrm dx_{k-1},
\end{equation}
and assume that there is some $k\in\mathbb N$ such that $W^{(k)}>0$ 
%is bounded away from zero, i.e., there is some $\kappa>0$ such that $W^{(k)}\geq\kappa$ 
on all of $[0,1]$, and such that $W^{(k)}$ is bounded away from zero, i.e., that there exists some $\kappa\in(0,1]$ such that $\inf_{x,y\in[0,1]}W^{(k)}(x,y)\geq\kappa$.
Finally, assume that $W$ is symmetric, i.e., $W(x,y)\equiv W(y,x)$.
\end{assumption}
The condition $W^{(k)}>0$ means that the system is \emph{mixing} after $k$ generations; if $W^{(k)}$ is bounded away from $0$, then this mixing is \emph{uniform}. 
We need a mixing assumption to argue that explosive behavior in one part of the space $\mathscr X$ spills over to the rest of the space.

\begin{theorem}\label{FLLN theorem 2}
Grant Assumptions~\ref{ass1}, \ref{ass cw cb} and \ref{ass9}. Let $N$ %=(N_t(A))_{t\in\mathbb R_+,A\in\mathcal B[0,1]}$ 
be an unmarked linear graphon Hawkes process on $\mathscr X=[0,1]$ and time interval $\mathbb R_+$, starting with an empty history. 
Suppose that we are in the regime $\rho(T_{\mathrm{hom}})>1$, with $T_{\mathrm{hom}}$ defined in \eqref{operator Thom no marks}. 
Then, for each $A\in\mathcal B[0,1]$ of positive Lebesgue measure, it holds that 
\begin{equation}
\frac{N_T(A)}T\to\infty\quad \text{a.s., as }T\to\infty.
\end{equation}
\end{theorem}

\begin{proof}
The proof is structured as follows. 
We first use the condition $\rho(T_{\mathrm{hom}})>1$ to find a function $\phi\in L^1[0,1]\cap L^\infty[0,1]$ such that $\|T^n_{\mathrm{hom}}\phi\|_{L^1[0,1]}\to\infty$. 
Since this function is \emph{bounded}, the same should hold for the unit profile $\mathbf1:x\mapsto1$. 
Then we use Assumption~\ref{ass9} to find a stream of immigrant particles that eventually lead to some uniformly distributed particle in $[0,1]$. 
Since the unit profile has infinite expected offspring, the same holds for our uniformly distributed particle. 
We conclude by applying the strong LLN.

Since $C_W<\infty$ by Assumption~\ref{ass cw cb}, we may apply Luzin's theorem to approximate the measurable function $W\in L^1[0,1]^2$ by continuous functions, and it follows by standard arguments that $T_{\mathrm{hom}}$ is a compact operator. 
Next, since $W$ is symmetric, it follows that when treated as an operator on $L^p[0,1]$, $p\in[1,\infty)$, $T_{\mathrm{hom}}$ is self-adjoint; see \cite{Conway}, Example VI.1.6. 
Hence, $T_{\mathrm{hom}}\in B(L^2[0,1])$ is self-adjoint, where $B(L^p[0,1])$ is the set of bounded linear operators on $L^p[0,1]$. 
For $T\in B(L^p[0,1])$, its spectrum is defined by $$\sigma^{L^p}(T):=\{\lambda\in\mathbb C:T-\lambda I\in B(L^p[0,1])\text{ is not invertible}\}.$$ 
Note that $L^2[0,1]\subset L^1[0,1]$, hence if $T_{\mathrm{hom}}-\lambda I$ is not invertible on $L^1[0,1]$, then it is not invertible on $L^2[0,1]$ either. 
Thus, $\sigma^{L^1}(T_{\mathrm{hom}})\subset\sigma^{L^2}(T_{\mathrm{hom}})$. 
Since the spectrum of a self-adjoint operator on a Hilbert space is a subset of the real line (\cite{Conway}, Proposition~VII.6.1), $\sigma(T_{\mathrm{hom}}):=\sigma^{L^1}(T_{\mathrm{hom}})\subset\mathbb R$. 

Next, it follows from the compactness of $T_{\mathrm{hom}}\in B(L^1[0,1])$ and \cite{Conway}, Theorem~VII.7.1, that $T_{\mathrm{hom}}$ has $\rho(T_{\mathrm{hom}})$ as an eigenvalue; denote the corresponding normalized eigenfunction as $\phi$. 
By the Krein-Rutman theorem (\cite{Kreinrutman}, Theorem~V.6.6) --- which applies by symmetry of $W$ and by the assumption that $W^{(k)}>0$ for some $k\in\mathbb N$ --- we can even select a normalized eigenfunction satisfying $\phi>0$ a.e. 
In other words, $\phi$ is a probability density on $[0,1]$. 
Furthermore, it holds that $\phi$ is essentially bounded, which follows from
\begin{align*}
\|\phi\|_{L^\infty[0,1]}&=\frac1{\rho(T_{\mathrm{hom}})}\|T_{\mathrm{hom}}\phi\|_{L^\infty[0,1]}=\frac1{\rho(T_{\mathrm{hom}})}\sup_{x\in[0,1]}\|h\|_{L^1(\mathbb R_+)}\int_0^1W(x,y)\phi(y)\ \mathrm dy\\
&\leq \frac1{\rho(T_{\mathrm{hom}})}C_W\|h\|_{L^1(\mathbb R_+)}\|\phi\|_{L^1[0,1]}.
\end{align*}
Hence, we can bound $\phi$ by the constant function $\|\phi\|_{L^\infty[0,1]}\times\mathbf1:x\mapsto\|\phi\|_{L^\infty[0,1]}$, which is in $L^1[0,1]$. 
For this unit profile $\mathbf1:x\mapsto 1$, we have $$\|T^n_{\mathrm{hom}}\mathbf1\|_{L^1[0,1]}=\frac{\left\|T^n_{\mathrm{hom}}\|\phi\|_{L^\infty[0,1]}\mathbf1\right\|_{L^1[0,1]}}{\|\phi\|_{L^\infty[0,1]}}\geq\frac{\|T^n_{\mathrm{hom}}\phi\|_{L^1[0,1]}}{\|\phi\|_{L^\infty[0,1]}}\to\infty,\quad\text{as }n\to\infty.$$ %hence the same holds for the constant unit profile $\mathbf1:x\mapsto1$, and in fact for any constant multiple of this profile. 

Our goal is to use Assumption~\ref{ass9} to construct a Poisson stream of uniformly distributed particles on $[0,1]$ --- i.e., distributed according to the unit profile  ---, each of which has infinite expected offspring size. 

Using Assumption~\ref{ass9}, there is a Poisson stream of intensity $$\theta:=\alpha\kappa(1-\exp(-\|h\|_{L^1(\mathbb R_+)}))^k>0$$ leading after $k$ generations (i.e., eventually) to some \emph{uniformly distributed} point in $[0,1]$. Indeed, we can bound the stream of events by a coupled stream where a $k$th-generation event in $x$ originating from an immigrant in $y$ is rejected with probability $1-\kappa/W^{(k)}(x,y)$.
Such a uniformly distributed event within $[0,1]$ has infinite expected offspring; here the expectation is w.r.t.\ both the location within $[0,1]$ and w.r.t.\ the offspring realization. 

Since the time between the birth of a parent and the birth of a child has the finite expectation $\int_0^\infty th(t)\ \mathrm dt<\infty$, we infer that $N_T[0,1]/T\to\infty$ a.s., as $T\to\infty$; this follows by applying the strong LLN to a compound Poisson process with arrival rate $\theta$ and infinite expected claim size. 
By invoking Assumption~\ref{ass9} again, it follows that, for any $A\in\mathcal B[0,1]$ of positive Lebesgue measure, we have $N_T(A)/T\to\infty$ a.s., as $T\to\infty$.
\end{proof}

\begin{remark}
Having stated limit theorems both in the subcritical ($\rho(T_{\mathrm{hom}})<1$) and supercritical ($\rho(T_{\mathrm{hom}})>1$) regime, we ask ourselves what happens in the critical ($\rho(T_{\mathrm{hom}})=1$) regime. 
By the Poisson cluster representation of the graphon Hawkes process, this is closely related to the critical behavior of Galton-Watson processes, i.e., one would expect that in the critical regime $N_T(A)$ is of order $T^2$. 
Indeed, asymptotics of this type for \emph{univariate} Hawkes processes can be found in \cite{Zhu2}, Theorem~21. %One would expect those asymptotics to carry over to higher dimensions.
\end{remark}

\subsection{Functional central limit theorem}
Having proved the FLLN given in \eqref{FLLN for graphon Hawkes}, a natural question is whether we can establish an FCLT for the graphon Hawkes process. 
We continue working in the unmarked setting.

As we have seen before, when we consider a graphon Hawkes process with piecewise constant parameters, treating the intervals of continuity as coordinates, we obtain a multivariate Hawkes process. 
For those processes, \cite{Bacry}, Corollary~1, provides an FCLT. 
The continuous-space extension of their result, which is our \textit{ansatz} for the FCLT, reads: for every $A\in\mathcal B[0,1]$,
\begin{equation}\label{FCLT explicit}
\sqrt T\left(\frac{N_{Tv}(A)}T-v\bar\lambda(A)\right)_{v\in[0,1]}\longrightarrow\left(\int_A\left((I-T_{\mathrm{hom}})^{-1}\bar\lambda^{1/2}\right)(x)X(v,x) \ \mathrm dx\right)_{v\in[0,1]},
\end{equation}
where $\bar\lambda\in L^1[0,1]$ solves (\ref{fredholmeq intensity}), $\bar\lambda^{1/2}$ is taken componentwise, and $X(t,x)=\int_{(0,t)}W(s,x)\ \mathrm ds$, for $W$ a standard $L^2$ Gaussian white noise on the strip $[0,\infty)\times[0,1]$.

Just as for the FLLN, we consider the number of points $N_{Tv}(A)$ in a single Borel set $A$. %, hence the analysis reduces to a univariate model. %, meaning that we ignore the spatial dependency structure. 
%An analysis of the system for multiple sets $A_1,\ldots,A_n$ simultaneously, only shows codependency similar to that of multivariate Hawkes, while 
%Without integrating, we would be looking for a limit theorem in density form, and (\ref{FCLT explicit}) requires a white noise limit process.
%It is more difficult to consider FLLNs and FCLTs without integrating over a set $A$, for in any single point $x\in[0,1]$, we have no event, a.s. Furthermore, an FCLT like (\ref{FCLT explicit}) requires a white noise limit process.
 % Such limit theorems are, to the best of our knowledge, not available in the literature. %Note that a white noise cannot be treated as a function in the usual sense, hence comparison to white noise in any reasonable metric gives no sensible results. 
%Note that even when we consider the behavior on a single set $A$, the process $N_t(A)$ is affected by the behavior in  other  parts of the space due to the spatial dynamics, 
Informally, by a result like \eqref{FCLT explicit}, we obtain a spatial white noise, temporal Brownian limit process, in the sense of Schwartz distributions.

In proving Theorem~\ref{FCLT theorem} below, we use a stronger stability condition than the requirement $\rho(T_{\mathrm{hom}})<1$ that we have been using before; see Remark~\ref{rmk5}. 

\begin{assumption}\label{ass10}
Let $D(y):=\int_0^1W(x,y)\ \mathrm dx$ be the \emph{outdegree} of $y\in[0,1]$. 
Suppose that 
\begin{equation}\label{eqass10}\|h\|_{L^1(\mathbb R_+)}\sup_{y\in[0,1]}D(y)<1.\end{equation}
\end{assumption}
Equation \eqref{eqass10} is equivalent to $\rho<1$, where $\rho$ is defined in \eqref{rho mathscr X} in Appendix~\hyperref[app A]{A}.

\begin{theorem}\label{FCLT theorem}
Grant Assumptions~\ref{ass1}, \ref{ass cw cb}, \ref{ass7}, \ref{ass6} and \ref{ass10}. 
Let $(N_t)_{t\geq0}$ be a stationary (see Proposition~\ref{existence mathscr X} and Theorem~\ref{thstationarity mathscr X}), unmarked linear graphon Hawkes process. % driven by the conditional intensity density (\ref{linearcondintsection2}) with $B\equiv1$. 
Then we have, for $A\in\mathcal B[0,1]$, and $B$ a standard Brownian motion,
\begin{equation}\label{FCLT mu sigma unspecified}
\sqrt T\left(\frac{N_{Tv}(A)}T-v\mu_A\right)_{v\in[0,1]}\longrightarrow\left(\sigma_A B(\cdot)\right)_{v\in[0,1]},
\end{equation}
as $T\to\infty$, weakly on $D[0,1]$ equipped with the Skorokhod topology.
Here, $\mu_A=\bar\lambda(A)$ and $\sigma_A\in\mathbb R_+$. 
Finally, for multivariate Hawkes processes, i.e., graphon Hawkes processes with piecewise constant parameters, $\sigma_A$ is given by $\int_A\left((I-T_{\mathrm{hom}})^{-1}\bar\lambda^{1/2}\right)(x)\ \mathrm dx$.
\end{theorem}
\begin{proof}
First, it follows from the FLLN given in \eqref{FLLN for graphon Hawkes}, which depends on Assumptions~\ref{ass1}, \ref{ass cw cb}, \ref{ass7} and \ref{ass6}, that if the FCLT \eqref{FCLT mu sigma unspecified} holds, % of the type $$\sqrt T\left(\frac{N_{Tv}(A)}T-v\mu_A\right)_{v\in[0,1]}\longrightarrow\left(\sigma_A B(\cdot)\right)_{v\in[0,1]}$$ 
then we necessarily have $\mu_A=\bar\lambda(A)$.

The proof of \eqref{FCLT mu sigma unspecified} employs Poisson embedding as in \cite{Zhu}, Theorem~1, and \cite{Stabilitypaper},
to verify the conditions of \cite{Billingsley}, Theorem~19.1. 
To this end, we need to check that 
\begin{equation}\label{finiteness cond zhu}
\sum_{n\geq1}\left(\mathbb E\left[\mathbb E\left[N_{(n,n+1]}(A)-\mu_A|\mathcal F_{0}^{(N)}\right]^2\right]\right)^{1/2}<\infty,
\end{equation}
where $N_{(t_1,t_2]}(A)$ counts the number of points arriving in $A$ during $(t_1,t_2]$, and where $\mathcal F_{0}^{(N)}$ is the $\sigma$-algebra recording the history of the stationary process $N$ before time $0$. %; this history is non-trivial, since we start with a stationary process. We prove (\ref{finiteness cond zhu}) by bounding by corresponding quantities for a coupled stationary univariate Hawkes process, to which we apply \cite[(2.2)]{Zhu}. 

The rest of the proof is structured as follows. 
First, we bound the left-hand side of \eqref{finiteness cond zhu} by the corresponding quantity with $A=[0,1]$. 
Then, we couple $N$ to a spatiotemporal process $M$ such that: (i) $M$ is essentially a univariate Hawkes process; (ii) $M$ is stationary; and (iii) the conditional variance of $N_{(n,n+1]}[0,1]$ can be bounded by that of $M_{(n,n+1]}[0,1]$. 
We finish the proof by applying \cite{Billingsley}, Theorem~19.1, to $M$.

We start bounding the left-hand side of \eqref{finiteness cond zhu} by the corresponding quantity with $A=[0,1]$:
\begin{align}
&\mathbb E\left[\mathbb E\left[N_{(n,n+1]}(A)-\mu_A|\mathcal F_{0}^{(N)}\right]^2\right]=\mathbb E\left[\mathbb E\left[N_{(n,n+1]}(A)|\mathcal F_{0}^{(N)}\right]^2\right]-\mu_A^2\nonumber\\
\nonumber&\leq\mathbb E\left[\mathbb E\left[N_{(n,n+1]}(A)|\mathcal F_{0}^{(N)}\right]^2\right]-\mu_A^2+\mathbb E\left[\mathbb E\left[N_{(n,n+1]}(A^\mathsf{c})|\mathcal F_{0}^{(N)}\right]^2\right]-\mu_{A^\mathsf{c}}^2\\
&\phantom{leq}+\nonumber\underbrace{2\left(\mathbb E\left[\mathbb E\left[N_{(n,n+1]}(A)|\mathcal F_{0}^{(N)}\right]\mathbb E\left[N_{(n,n+1]}(A^\mathsf{c})|\mathcal F_{0}^{(N)}\right]\right]-\mu_A\mu_{A^\mathsf{c}}\right)}_{\text{spatial covariance }\geq\ 0\text{ by nonnegativity of }W\text{ and }h}\\
&=\nonumber\mathbb E\left[\mathbb E\left[N_{(n,n+1]}(A)+N_{(n,n+1]}(A^\mathsf{c})-\mu_A-\mu_{A^\mathsf{c}}|\mathcal F_{0}^{(N)}\right]^2\right]\\
&=\mathbb E\left[\mathbb E\left[N_{(n,n+1]}[0,1]-\mu_{[0,1]}|\mathcal F_{0}^{(N)}\right]^2\right].\label{boundingbyM}%=\mathbb E\left[\mathbb E\left[N_{(n,n+1]}([0,1])|\mathcal F_0^{-\infty}\right]^2\right]-\mu_{[0,1]}^2.
\end{align}

Next, we bound \eqref{boundingbyM} by the corresponding conditional variance expression for a spatiotemporal point process $M=(M_t)_{t\in\mathbb R}$ such that $M[0,1]$ is a stable univariate Hawkes process. 
Consider the cluster representation for the graphon Hawkes process, Definition~\ref{clusterdef}. 
We define $M$ by coupling it to $N$ iteratively as follows. 
For each $k\in\mathbb N_0$, we set $M_k^{(1)}$ equal to the $k$th-generation of $N$, but replace the spatial coordinates by i.i.d.\ $\mathrm{Uni}(0,1)$ marks. 
Next, for each $k\in\mathbb N$ and each $(k-1)$th-generation event $(s,y)$ of $N$, $M_k^{(2)}$ is given by initial streams of events generated by an inhomogeneous Poisson process with intensity $$[s,\infty)\ni t\mapsto h(t-s)\left(\sup_{y\in[0,1]}D(y)-\int_0^1 W(x,y)\ \mathrm dx\right),$$ marked by i.i.d.\ $\mathrm{Uni}(0,1)$ spatial coordinates; use Assumption~\ref{ass10}. 
Next, for each event $(s',y')$ in $M_k^{(2)}$, we let $M_k^{(3)}$ be a univariate Hawkes cluster of events with intensity  $$[s',\infty)\ni t\mapsto h(t-s')\sup_{y\in[0,1]}D(y),$$ marked by i.i.d.\ $\mathrm{Uni}(0,1)$ spatial coordinates. %We let the $M_k^{(i)}$'s, $i=1,2,3$ be conditionally independent, and 
We set $$M^{(1)}:=\bigcup_{k\geq0}M_k^{(1)};\quad M^{(i)}:=\bigcup_{k\geq1}M_k^{(i)}\text{ for }i=2,3;\quad M:=M^{(1)}\cup M^{(2)}\cup M^{(3)}.$$ 

Note that it is possible to select $M_k^{(2)}$ and $M_k^{(3)}$ in such a way that $M[0,1]$ is stationary under $\mathbb P$, since we couple the history before time $0$ of the stationary process $N$ to $M^{(1)}$. 
Since $N$ is stationary, so is the cluster center process (see \cite{DaleyVereJones}) corresponding to $N$: those centers arrive according to a stationary homogeneous Poisson point process of rate $\|\lambda_\infty\|_{L^1[0,1]}$ with i.i.d.\ $\mathrm{Uni}(0,1)$ marks; the same holds for $M$. 
Compared to $N$, the clusters of $M$ are enlarged by $M^{(2)}$ and $M^{(3)}$. 
Here, $M$ is a univariate Hawkes process with baseline intensity $\|\lambda_\infty\|_{L^1[0,1]}$, excitation function $h(\cdot)\sup_{y\in[0,1]}D(y)$ and i.i.d.\ $\mathrm{Uni}(0,1)$ spatial coordinates; by Assumption~\ref{ass10} it satisfies the stability conditions of Theorem~\ref{thm stability general X spectral radius} (see Remark~\ref{rmk5}). 
Note that any stationary distribution for the univariate Hawkes process consists of a stationary cluster center process, to which mutually independent component processes are attached. 
This is exactly how it is done for $M$. 
%We start with a stationary cluster arrival rate, and the distribution of the clusters at process level is fixed by the dynamics of the univariate Hawkes process.
%To see this, start with the stationary version of the univariate Hawkes process $(M_t[0,1])_{t\in\mathbb R_+}$. 
%\textbf{(RIGHT??)}

To compare the conditional variances of $N$ and $M$, note that those conditional variances are w.r.t.\ different $\sigma$-algebras. 
To overcome this problem, let $\mathcal F_0^{(N,M)}$ be the $\sigma$-algebra generated by the histories of $N,M$ before time $0$. 
Since $M_k^{(2)}$ and $M_k^{(3)}$ are drawn conditionally independently of $N$, it holds that  
\begin{align}\mathbb E\left[\mathbb E\left[N_{(n,n+1]}[0,1]-\mu_{[0,1]}|\mathcal F_0^{(N)}\right]^2\right]&=\mathbb E\left[\mathbb E\left[N_{(n,n+1]}[0,1]-\mu_{[0,1]}|\mathcal F_0^{(N,M)}\right]^2\right]
\label{sigmaalgebra1}.\end{align}
Next, since the effect of the arrival of a $k$th-generation event on the evolution of a Hawkes process does not depend on $k\in\mathbb N_0$,
\begin{align}
\mathbb E\left[\mathbb E\left[M_{(n,n+1]}[0,1]-\mu^{(M)}_{[0,1]}|\mathcal F_0^{(M)}\right]^2\right]
&=
\mathbb E\left[\mathbb E\left[M_{(n,n+1]}[0,1]-\mu^{(M)}_{[0,1]}|\mathcal F_0^{(N,M)}\right]^2\right].\label{sigmaalgebra2}
\end{align} 

Note that $N$ is not independent of $(M_k^{(2)},M_k^{(3)})$, as there is a positive covariation between $N$ and $M_k^{(2)},M_k^{(3)}$. 
Indeed, higher $k$th-generation offspring for $N$ leads to an excited conditional intensity for $M_k^{(2)}$, which in turn induces an excited conditional intensity for $M_k^{(3)}$. 
Therefore, the conditional variance of $N_{(n,n+1]}[0,1]$ given $\mathcal F_0^{(N,M)}$ can be bounded by that of $M_{(n,n+1]}[0,1]$.
Upon combining this with \eqref{sigmaalgebra1}--\eqref{sigmaalgebra2}, it follows from the proof of \cite{Zhu}, Theorem~1, that 
\begin{align}
&\phantom{\leq}\ \nonumber\sum_{n\geq1}\left(\mathbb E\left[\mathbb E\left[N_{(n,n+1]}[0,1]-\mu_{[0,1]}|\mathcal F_0^{(N)}\right]^2\right]\right)^{1/2}\\
&\leq \sum_{n\geq1}\left(\mathbb E\left[\mathbb E\left[M_{(n,n+1]}[0,1]-\mu^{(M)}_{[0,1]}|\mathcal F_0^{(M)}\right]^2\right]\right)^{1/2}<\infty. \label{bounding by M}
\end{align}
Combining \eqref{boundingbyM} and \eqref{bounding by M}, \cite{Zhu}, Theorem~3, gives that for some $\sigma_A\in\mathbb R_+$,
\begin{equation}
\left(\frac{N_{\lfloor Tv\rfloor}(A)-\lfloor Tv\rfloor\mu_A}{\sqrt T}\right)_{v\in[0,1]}\longrightarrow\sigma_AB(\cdot),
\end{equation}
as $T\to\infty$, weakly on $D[0,1]$ equipped with the Skorokhod topology. 
Again by bounding by $M$ as in \eqref{boundingbyM}--\eqref{bounding by M}, we deduce that $\mathbb E[N_{[0,1]}(A)^2]<\infty$, hence we may proceed as in \cite{Zhu}, eqn.\ (2.19), to conclude that \eqref{FCLT mu sigma unspecified} holds.

We still need to identify $\sigma_A$. 
To this end, note that by \cite{Bacry}, Corollary~1, \eqref{FCLT mu sigma unspecified} holds for graphon Hawkes processes with piecewise constant parameters on the sets of a partition $\mathcal P$ of $[0,1]$ such that $A\in\mathcal P$.
\end{proof}

\begin{comment}
\begin{remark}
An FCLT as given in (\ref{FCLT explicit}) can only be stated in a meaningful way when we integrate $N$ over some measurable set $A\in\mathcal B[0,1]$. However, it allows for the heuristic interpretation
\begin{equation}\label{FCLT infinitesimal}
\sqrt T\left(\frac{N_{Tv}(\mathrm dx)}T-v\bar\lambda(x)\ \mathrm dx\right)_{v\in[0,1]}\longrightarrow\left(\left((I-T_{\mathrm{hom}})^{-1}\bar\lambda^{1/2}\right)(x)X(v,x)\ \mathrm dx\right)_{v\in[0,1]}.
\end{equation}
In words, for large times, the centralized and rescaled process resembles a Gaussian white noise process over space, and a Brownian motion over time.
\end{remark}
\end{comment}

\section{Fixed-point theorems in the transform domain and an application}\label{section6 fixed point theorems}
Given the definition of a finite-dimensional linear Hawkes process as a Poisson cluster process, it is possible  to derive fixed-point equations for Hawkes-fed birth-death processes in the transform domain, which can be used to approximate transforms numerically; see \cite{multivariateKLM, Infinite server queues} for more details. 
In this section, we utilize Definition~\ref{clusterdef} to establish fixed-point theorems in the infinite-dimensional setting. 
As argued at the start of Section~\ref{large time behavior section},  we can restrict ourselves to $\mathscr X=[0,1]$ without losing generality. 
As before, this is for legibility only: all arguments only rely on $\mathscr X$ being a compact subset of an Euclidean space $\mathbb R^m$, and in particular do not require $\mathscr X$ to be one-dimensional. 
We finish this section by applying the fixed-point theorems we establish to prove that, starting with a $d$-dimensional Hawkes birth-death process $\tilde Q^d(t)$, the limits $d\to\infty$ and $t\to\infty$ commute.

\subsection{Transform characterization}\label{section tranform chars}
% FOR LATER: ANALYZING JOINT TRANSFORMS SHOULD ALSO BE POSSIBLE, WHERE LAMBDA IS DOMINATED BY THE LEBESGUE MEASURE WITH DENSITY lambda, HENCE WE GET AT TIME T
% E[exp(-\int_0^1 g(x)\lambda_t(x) dx)]
Besides a graphon Hawkes pure-birth process $N$, i.e., a counting process, we are interested in a model allowing for deaths, i.e., departures from the system. 
Suppose that each particle in spatial coordinate $x\in\mathscr X$ has a stochastic lifetime distributed as $J_x$, having a distribution only dependent upon the spatial coordinate, independent of the lifetimes of other particles.
%where lifetimes of different events are independent. 

\begin{definition}
Let $(\mathscr X,\mathcal B(\mathscr X),\mu)$ be a topological $\sigma$-finite measure space. 
Let $N=(N_t(x))_{t\in[0,\infty)\times\mathscr X}$ be a spatiotemporal point process, where to each event $(t,x)$ a $J_x$-distributed lifetime $w_{t,x}$ is attached, independent of everything else. 
Define the spatiotemporal \emph{birth-death process} $Q$ through the formula $$Q_t(A)=\left|\{(t',x)\in N:t'\leq t<t'+w_{t',x}\}\right|,\quad A\in \mathcal B(\mathscr X).$$
\end{definition}

\begin{assumption}\label{separable lifetimes}
Consider a topological $\sigma$-finite measure space $(\mathscr X,\mathcal B(\mathscr X),\mu)$.
Let $N=(N_t(x))_{t\in[0,\infty)\times\mathscr X}$ be a \emph{linear} spatiotemporal point process, where to each event $(t,x)$ a $J_x$-distributed lifetime $w_{t,x}$ is attached, independent of everything else. 
Assume that the lifetimes $(J_x)_{x\in\mathscr X}$ are i.i.d., distributed as $J$, which is a stochastic process $\mathscr X\to\mathbb R_+$, separable w.r.t.\ the class $\mathcal U$ of open subsets of $\mathscr X$.
\end{assumption}

The definition of \emph{separability} of a stochastic process can be found in \cite{Neveu}, \S III.4.

We aim to characterize the distribution of the linear graphon Hawkes process \emph{with departures} through a suitable transform. %, generalizing results from \cite{multivariateKLM, Infinite server queues}. 
Since we work with a continuum of spatial coordinates, we consider a Laplace functional instead of a Laplace transform.

\begin{definition}\label{Laplace functional}
Let $Z$ be a spatiotemporal point or birth-death process, where each event consists of a spatial coordinate $x\in\mathscr X$ and a temporal coordinate $t\in\mathbb R_+$. 
Let $\mathrm{BM}_+(\mathscr X)$ be the space of bounded measurable functions $f:\mathscr X\to\mathbb R_+$. 
For $t\in\mathbb R_+$, we define its \emph{Laplace functional} $\mathcal L_N(\cdot,t):\mathrm{BM}_+(\mathscr X)\to\mathbb R_+$ by 
\begin{equation}\label{L_n(f)}
\mathcal L_Z(f,t):=\mathbb E\left[\exp(-f(Z,t))\right]:=\mathbb E\left[\exp\left(-\sum_{(x_i,t_i)\in Z(t)}f(x_i)\right)\right],
\end{equation} 
%or, for short, $\mathbb E\left[\exp\left(-\sum_{x_i\in N(t)}f(x_i)\right)\right]$,
 with the understanding that $Z(t)$ registers the spatial coordinates of events that have occurred (but have not yet expired) by time $t$.
Here, the expectation is w.r.t.\ the natural filtration generated by $Z$ at time $0$. 
We denote the space of such transforms by $\mathbb L$.
\end{definition}

Note that a characteristic functional defines a distribution uniquely, by a functional version of the Minlos-Bochner theorem. 
See \cite{Levycontinuity}, Theorem~2.1, and \cite{Prohorov}, Theorem 2.

\begin{remark}
By taking $f\equiv f(z):\mathscr X\to\mathbb R_+:x\to z\mathbf1_A(x)$ for $z\in\mathbb R$ and $A\subset\mathscr X$, or, more generally,
$f\equiv f(\mathbf z):\mathscr X\to\mathbb R:x\to\sum_{i=1}^Kz_i\mathbf1_{A_i}(x)$ for $\mathbf z\in\mathbb R^K$,  $K\in\mathbb N$, and $A_1,\ldots, A_K\subset \mathscr X$,
one obtains usual Laplace transforms, which can be used to determine moments of $N(A)$ or $N(A_1)\cdots N(A_K)$ by differentiation. 
By ranging over different $K$ and $A_1,\ldots,A_K$, the Laplace functional allows one to construct \textit{moment measures}.
\end{remark}

Our goal is to use the first step of the cluster representation, (i) in Definition~\ref{clusterdef}, to relate the transform $\mathcal L_Q$, where $Q$ is a linear graphon Hawkes birth-death process, to the cluster processes constituted by each of the (immigrant) arrivals. 
Suppose that the lifetime random variables $(J_x)_{x\in[0,1]}$ satisfy Assumption~\ref{separable lifetimes}. 
Denote the cluster process resulting from an arrival in coordinate $x\in[0,1]$ by $S_x$.  
This keeps track of all surviving offspring of the initial particle in coordinate $x$, including that particle itself, if it has not yet expired. 
We denote the Laplace functional of $S_x$ by $\eta_x=\mathcal L_{S_x}$, i.e., $\eta_x(f,u)=\mathbb E\left[-\exp(f(S_x,u))\right]$. 
With this characterization of $\mathcal L_Q$ in terms of $(\eta_x)_{x\in[0,1]}$, we use the iterative step in the cluster representation, (ii)--(iii) in Definition~\ref{clusterdef}, to derive a fixed-point equation for $\eta_x$, where we express $\eta_x$ as a function of $(\eta_x)_{x\in[0,1]}$, the collection of Laplace functionals of cluster processes $(S_x)_{x\in[0,1]}$. 
This asks for a transform of $\mathscr X$-dimensional, $\mathscr X$-valued point processes.

\begin{definition}\label{Laplace functional infdim}
Let $Z=(Z_x)_{x\in \mathscr X}$ be an $\mathscr X$-dimensional spatiotemporal point or birth-death process, where for $x\in\mathscr X$, $Z_x$ is a spatiotemporal point or birth-death process in the sense of Definition~\ref{Laplace functional}. 
Let $\mathscr Z$ be the space of such processes. 
We define the Laplace functional of $Z\in\mathscr Z$ as the mapping $$\mathcal L_Z(\cdot,t):\mathrm{BM}_+(\mathscr X)\to\mathbb R_+^{\mathscr X}:f\mapsto\left(\mathcal L_{Z_x}(f,t)\right)_{x\in\mathscr X}.$$ 
We denote the space of such transforms by $\mathbb L^{\mathscr X}$.

Fix a time horizon $t\in\mathbb R_+$ and some $p\in[1,\infty]$. 
On $\mathbb L^{\mathscr X}$, we consider the metric
\begin{equation}\label{metric LX}
d_{\mathbb L^{\mathscr X},p}:\mathbb L^{\mathscr X}\times\mathbb L^{\mathscr X}\to\mathbb R_+:
(\mathcal L,\mathcal L')\mapsto\sup_{\substack{u\in[0,t]\\ f\in\mathrm{BM}_+(\mathscr X)}}\|\mathcal L_x(f,u)-\mathcal L_x'(f,u)\|_{L^p(\mathscr X)}.
\end{equation}
\end{definition}

%We now specify to $\mathscr X=[0,1]$, for legibility; see the introduction to this section.
Our first result relates the transform of $Q$ to the transform of the cluster processes $(S_x)_{x\in[0,1]}$. 
It implies that in order to calculate the transform of $Q$, it suffices to calculate $(\eta_x)_{x\in[0,1]}$. 
The proof of this result can be found in Appendix~\hyperref[app D]{D}.

\begin{theorem}\label{transform Q S_x}
Grant Assumptions~\ref{ass1} and~\ref{separable lifetimes}. 
Let $\mathscr X=[0,1]$.
The transform of $Q$ satisfies 
\begin{equation}\label{queue into cluster laplace}
\mathcal L_Q(f,t)=\exp\left(\int_0^1\int_0^t\left(\eta_x(f,u)-1\right)\lambda_\infty(x)\ \mathrm du\ \mathrm dx\right).
\end{equation}
\end{theorem}

The next step is to derive a fixed-point equation for the cluster transforms $\eta_x$, by operationalizing the iterative step from the cluster representation. 
This iterative step can be transformed into a distributional equality for $S_x$. 
Given an arrival in $x$, translate the time frame such that this arrival occurs at $t=0$. 
Conditionally upon the outcome of $B$, let $K_x(\cdot)$ be an inhomogeneous Poisson process with intensity $\|B_{\cdot x}(0)W(\cdot,x)\|_{L^1[0,1]}h(\cdot)$, counting the children of the initial event in $x$. 
Let the locations $x_i$ of the children be drawn i.i.d.\ according to the law having Lebesgue density $$\mathbb P(x_i\in\mathrm dz)=\frac{B_{zx}(0)W(z,x)\ \mathrm dz}{\|B_{\cdot x}(0)W(\cdot,x)\|_{L^1[0,1]}}.$$ 
Then, with $\delta_x(y)=\mathbf1\{y=x\}$ %$\delta_x:[0,1]\to\{0,1\}:y\mapsto\mathbf1\{y=x\}$ 
the Kronecker delta, the following distributional equality holds:
\begin{equation}\label{distributionaleq}
S_x(u)\stackrel{\mathcal D}=\delta_x\cdot\mathbf1\{J_x>u\}+\sum_{i=1}^{K_x(u)}S_{x_i}(u-t_i).
\end{equation}
%As a remark, $S_x$ is a process with events in $[0,1]$. Here $\delta_x$ indicates that we have an event for $y=x$.

This distributional equality can be used to prove a fixed-point result. 
Denote the survival function and c.d.f.\ corresponding to the lifetime random variable $J_x$ by 
$$\mathscr J_x(u)=\mathbb P(J_x>u)\quad\text{and}\quad\bar{\mathscr J}_x(u)=\mathbb P(J_x\leq u)=1-\mathscr J_x(u).$$
\begin{definition}\label{defPhi}
Let $\beta_x$ be the Laplace functional of the stochastic process $B_{\cdot x}(0)$. 
For $\mathscr X=[0,1]$, we define the operator $\Phi:\mathbb L^{[0,1]} \to\mathbb L^{[0,1]}$ as follows. 
Let 
\begin{equation}\label{def10gamma}\gamma_x(f,u):=\bar{\mathscr J}_x(u)+\mathscr J_x(u)e^{-f(x)}.
\end{equation} 
For $\xi\in \mathbb L^{[0,1]}$, $f\in\mathrm{BM}_+[0,1]$ and $u\in\mathbb R_+$, let
\begin{align}
\Phi_x(\xi)(f,u):=\gamma_x(f,u)\beta_x\left(\left\{y\mapsto\int_0^u\left(1-\xi_y(f,u-s)\right)W(y,x)h(s)\ \mathrm ds\right\}\right).
\end{align}
\end{definition}
Before proving that $\eta$ is a fixed point of $\Phi$, we prove that $\Phi$ is a well-defined mapping. 
The proofs of the next four results can be found in Appendix~\hyperref[app D]{D}.

\begin{lemma}\label{lemmawelldefined}
Grant Assumptions~\ref{ass1} and~\ref{separable lifetimes}.
The mapping $\Phi:\mathbb L^{[0,1]} \to\mathbb L^{[0,1]}$ is well-defined, i.e., $\Phi(\xi)\in\mathbb L^{[0,1]}$ for $\xi\in\mathbb L^{[0,1]}$.
\end{lemma}

\begin{theorem}\label{fixedpointtheorem}
Grant Assumptions~\ref{ass1} and~\ref{separable lifetimes}.
The $[0,1]$-dimensional spatiotemporal cluster process $\eta$ is a fixed point of $\Phi$: we have $\eta=\Phi(\eta)$, i.e., for each $x\in[0,1]$, $\eta_x=\Phi_x(\eta)$.
\end{theorem}

\begin{lemma}\label{continuity}
Grant Assumptions~\ref{ass1}, \ref{ass cw cb} and~\ref{separable lifetimes}.
Fix a time horizon $t\in\mathbb R_+$ and $p\in[1,\infty]$. 
The mapping $\Phi$ %from Definition \ref{defPhi} 
is continuous w.r.t.\ the topology induced by the metric $d_{\mathbb L^{[0,1]},p}$.
\end{lemma}

Next we show that $\Phi$ is a \emph{contraction}. 
To state the result, for $\xi^{(0)}\in\mathbb L^{[0,1]}$, we denote its iterates under $\Phi$ by $(\xi^{(n)})_{n\in\mathbb N_0}$, i.e., $\xi^{(n)}=\Phi(\xi^{(n-1)})$ for $n\in\mathbb N$. 
\begin{lemma}\label{contractionlemma}
Grant Assumptions~\ref{ass1}, \ref{ass cw cb} and~\ref{separable lifetimes}.
%Suppose $p=1$ or $p=\infty$.
Let $\xi^{(0)},\zeta^{(0)}\in\mathbb L^{[0,1]}$. 
Then there exists $C\in\mathbb R_+$ such that, for any $u\in[0,t]$, $f\in\mathrm{BM}_+[0,1]$ and $n\in\mathbb N$, it holds that 
\begin{equation}\label{contractioneq}
\left\|\xi_x^{(n)}(f,u)-\zeta_x^{(n)}(f,u)\right\|_{L^p[0,1]}\leq\frac{C^nu^n}{n!}.
\end{equation}
\end{lemma}
%\begin{proof}
%See Appendix \hyperref[app A4]{A4}.
%\end{proof}

We are now ready to state the main result of this subsection, saying that iterates under $\Phi$ of any transform in $\mathbb L^{[0,1]}$ converge to  $\eta$.

\begin{theorem}\label{convergence_iterates}
Grant Assumptions~\ref{ass1}, \ref{ass cw cb} and~\ref{separable lifetimes}.
Fix $t\in\mathbb R_+$. 
For any $\xi^{(0)}\in\mathbb L^{[0,1]}$, the sequence $(\xi^{(n)})_{n\in\mathbb N_0}$ converges pointwise to the unique fixed point $\eta$ in $\mathbb L^{[0,1]}$ of $\Phi$. 
More specifically, for any $u\leq t$ and $f\in\mathrm{BM}_+[0,1]$,
\begin{equation}
\xi^{(n)}(f,u)\to \eta(f,u).
\end{equation}
\end{theorem}
\begin{proof}
By Lemma~\ref{contractionlemma}, iterates $\xi^{(n)},\zeta^{(n)}$ of two transforms $\xi^{(0)},\zeta^{(0)}\in\mathbb L^{[0,1]}$ converge pointwise and have the same limit $\mathscr L$. 
By Lemma~\ref{lemmawelldefined}, we know that those iterates are in $\mathbb L^{[0,1]}$ for any $n\in\mathbb N_0$. 
Next, for each $x\in[0,1]$, we have convergence of $\xi_x^{(n)}$ to $\mathscr L_x$, pointwise, as $n\to\infty$. 
This convergence holds in particular for all test functions $f$ in the Schwartz space on $[0,1]$. 
By treating a $[0,1]$-marked point process as a functional on $[0,1]$, we may apply a functional form of Lévy's continuity theorem, \cite{Levycontinuity}, Theorem~2.3, to conclude that there exists a graphon point process $Z_x$ such that $\mathscr L_x=\mathcal L_{Z_x}$, i.e., $\mathscr L_x\in \mathbb L$. 
In particular, for $Z=(Z_x)_{x\in[0,1]}\in\mathscr Z$, we have $\mathscr L=\mathcal L_{Z}$, meaning that $\mathscr L\in\mathbb L^{[0,1]}$.
Since $\Phi$ is continuous by Lemma~\ref{continuity}, we have
$$\mathscr L=\lim_{n\to\infty}\xi^{(n+1)}=\lim_{n\to\infty}\Phi\left(\xi^{(n)}\right)=\Phi\left(\lim_{n\to\infty}\xi^{(n)}\right)=\Phi(\mathscr L),$$ i.e., $\mathscr L$ is a fixed point of $\Phi$. 
From Theorem~\ref{fixedpointtheorem}, $\eta$ is also a fixed point of $\Phi$. 
As iterates of $\eta$ converge to $\mathscr L$, it follows that $\eta=\mathscr L$. 
\end{proof}

\subsection{Commuting large-time and high-dimension limits}\label{section interchanging limits}
In this subsection, we study large-time behavior of the graphon Hawkes birth-death process. 
Suppose that we start with multivariate Hawkes population processes $\tilde Q^d(\cdot)$ converging to a graphon Hawkes population process $Q(\cdot)$, where we first let $t\to\infty$, and then $d\to\infty$. 
We would like to know whether the resulting limit 
%(as $t\to\infty$, i.e., stationary) 
graphon process is the same as the stationary version of $Q$, i.e., whether we can interchange limits as illustrated by Figure~\ref{interchanging limits}. 
To answer this question, we view the prelimit as a piecewise constant graphon Hawkes process, so that we can employ the characterizations of the Laplace functional from the previous subsection to describe its distribution. %we use the characterization of the Laplace functionals of the population process $Q(\cdot)$ from the last subsection to study large-time behavior for the graphon Hawkes process.

\begin{figure}[!htbp]
\centering
\begin{tikzcd}
\tilde Q^d(t) \arrow{r}{d\to\infty} \arrow[swap]{d}{\substack{t\\\downarrow\\\infty}} & Q(t) \arrow{d}{\substack{t\\\downarrow\\\infty}} \\
\tilde Q^d(\infty) \arrow{r}{d\to\infty} & Q(\infty)
\end{tikzcd}
\caption{The limits $d\to\infty$ and $t\to\infty$, applied to $\tilde Q^d(t)$, commute.}\label{interchanging limits}
\end{figure}

\begin{assumption}\label{ass8}
Suppose that the baseline intensity $\lambda_\infty$ is bounded below by some $c_\infty$ on the set $\{x:\lambda_\infty(x)>0\}=\{x:\lambda_\infty(x)\geq c_\infty\}$. 
Suppose that $\int_0^\infty th(t)\ \mathrm dt<\infty$. 
Finally, assume that there is a single lifetime distribution $J$, independent of the particle location $x\in[0,1]$, satisfying $\mathbb E[J]<\infty$. 
Write $\mathscr J,\bar{\mathscr J}$ for the corresponding survival function and c.d.f., respectively.
\end{assumption}

We impose the (mild) assumption on the lower bound on $\lambda_\infty$ in Theorem~\ref{th interchanging limits} below for technical reasons. 
Furthermore, we require the assumption of finite expected time between parent and child, $\int_0^\infty th(t)\ \mathrm dt<\infty$, because we use boundedness of the cluster durations. 
Finally, we use the assumption on the lifetime distributions in order to couple lifetimes of prelimit and graphon Hawkes processes.

\begin{theorem}\label{th interchanging limits}
Grant Assumptions~\ref{ass1}, \ref{ass cw cb}, \ref{ass7}, \ref{ass6}, \ref{separable lifetimes} and \ref{ass8}, and suppose that the stability conditions of Theorem~\ref{thm stability general X spectral radius} hold.
Take a linear graphon Hawkes birth-death process $Q(t)$ and consider its multivariate projections $\tilde Q^d(t)$, using partitions $\mathcal P^d$ of $[0,1]$ with $\mathrm{mesh}(\mathcal P^d)\to0$, as $d\to\infty$.
Assume that the lifetime distributions for $\tilde Q^d$ are the same as those of $Q$, with coupled lifetimes for simultaneous events.
Then the stationary version $\tilde Q^d(\infty)$ of $\tilde Q^d(t)$ converges weakly in the space of all point measures on $[0,\infty)\times[0,1]$ to $Q(\infty)$, as $d\to\infty$. %This is illustrated by Figure \ref{interchanging limits}.
\end{theorem}
 
\begin{proof} 
% As before, the parameters of the prelimit in coordinate $n\in[d]$ are obtained as the averages of the parameters of the graphon Hawkes process restricted to $\mathscr X_n^d$, i.e., to the corresponding interval of the partition $\mathcal P^d$. We assume the same for 
We prove weak convergence of $\tilde Q^d(\infty)$ to $Q(\infty)$ by proving pointwise convergence of the corresponding Laplace functionals $\tilde{\mathcal L}_{Q}^d(f,\infty),\mathcal L_Q(f,\infty)$, using \cite{DaleyVereJones}, Proposition~11.1.VIII. 
We do this by invoking Theorems~\ref{transform Q S_x} and~\ref{fixedpointtheorem}, after which we bound the difference $|\tilde{\mathcal L}_Q^d(f,\infty)-\mathcal L_Q(f,\infty)|$ by two integral terms, which we treat separately.

Using Theorems~\ref{transform Q S_x} and~\ref{fixedpointtheorem}, we characterize the transform of the prelimit for $t\to\infty$ as 
\begin{equation}\label{prelimit Laplace characterization}
\tilde{\mathcal L}^d_Q(f,\infty)=\exp\left(\sum_{k=1}^d\lambda_{\infty,k}^d\int_{\mathscr X_k^d}\int_0^\infty(\tilde\eta_x^d(f,u)-1)\ \mathrm du\ \mathrm dx\right),
\end{equation}
where for $x\in \mathscr X_m^d$ we have, by recognizing the Laplace functional $\beta_{km}^d$ of the marks $B_{km}^d$,
\begin{align}\label{prelimit fixedpoint}
\nonumber\tilde\eta_x^d(f,u)&=\gamma_x(f,u)\mathbb E\left[\exp\left(\sum_{k=1}^dB_{km}^d(0)W_{km}^d\int_{\mathscr X_k^d}\int_0^u(\tilde\eta_y^d(f,u-s)-1)h(s)\ \mathrm ds\ \mathrm dy\right)\right]\nonumber\\
&=\gamma_x(f,u)\prod_{k=1}^d\beta_{km}^d\left(W_{km}^d\int_{\mathscr X_k^d}\int_0^u(\tilde\eta_y^d(f,u-s)-1)h(s)\ \mathrm ds\ \mathrm dy\right).
\end{align}
Here, $\gamma_x(f,u)=\bar{\mathscr J}(u)+\mathscr J(u)e^{-f(x)}$, cf.\ \eqref{def10gamma}.
From the convergence of the processes $\tilde N^d$  to $N$, as $d\to\infty$, for finite time horizons, 
%and for lifetime random variables $(J_m^d)_{m\in[d]}$ such that $(\sum_{m=1}^dJ_m^d\mathbf1\{x\in \mathscr X_m^d\})_{x\in[0,1]}$ converges weakly to $(J_x)_{x\in[0,1]}$ in the Skorokhod topology, 
we conclude the convergence $\tilde\eta^d_x(f,u)\to\eta_x(f,u)$ pointwise, as $d\to\infty$, since the transforms up to time $u$ only depend on smaller times $0\leq s\leq u$; or, alternatively, since the convergence for single clusters is implied by the convergence of the whole process (on bounded time intervals). 
%Now note that under Assumptions \ref{ass7} and \ref{ass8}, we can find a $d'$ such that $d\geq d'$ implies that $(\sum_{m=1}^dJ_m^d\mathbf1\{x\in \mathscr X_m^d\})_{x\in[0,1]}$ is within a distance of $\epsilon$ to $(J_x)_{x\in[0,1]}$, uniformly over $x\in[0,1]$ and over $\mathbb P$-a.a. realizations. 

We show that $|\tilde{\mathcal L}_Q^d(f,\infty)-\mathcal L_Q(f,\infty)|$ tends to $0$ as $d\to\infty$, for each $f$, since pointwise convergence of Laplace functionals is equivalent to weak convergence of random measures, see \cite{DaleyVereJones}, Proposition~11.1.VIII. 
It suffices to restrict ourselves to positive, continuous functions $f\in C_+[0,1]$ instead of $f\in\mathrm{BM}_+[0,1]$, see, e.g., \cite{nonlifeinsurancemath}, Theorem~9.1.4. 
Such a function $f$ is uniformly continuous on $[0,1]$ and uniformly bounded by some $C_f\in\mathbb R_+$.

Note that by the mean value theorem and the triangle inequality
\begin{align}\nonumber
&|\tilde{\mathcal L}_Q^d(f,\infty)-\mathcal L_Q(f,\infty)|\\\nonumber&=\Bigg|\exp\left(\int_0^1\sum_{k=1}^d\mathbf1\{x\in \mathscr X_k^d\}\lambda_{\infty,k}^d\int_0^\infty(\tilde\eta_x^d(f,u)-1)\ \mathrm du\ \mathrm dx\right)\\\nonumber
&\phantom{....}-\exp\left(\int_0^1\int_0^\infty(\eta_x(f,u)-1)\lambda_\infty(x)\ \mathrm du\ \mathrm dx\right)\Bigg|\\\nonumber
&\leq\left|\int_0^1\int_0^\infty\left[\sum_{k=1}^d\mathbf1\{x\in \mathscr X_k^d\}\lambda_{\infty,k}^d(\tilde \eta_x^d(f,u)-1)-\lambda_\infty(x)(\eta_x(f,u)-1)\right]\ \mathrm du\ \mathrm dx\right|\\\label{to be bounded1}
&\leq\int_0^1\left|\sum_{k=1}^d\mathbf1\{x\in \mathscr X_k^d\}\lambda_{\infty,k}^d-\lambda_\infty(x)\right|\int_0^\infty\left|\tilde \eta_x^d(f,u)-1\right|\ \mathrm du\ \mathrm dx\\\label{to be bounded2}
&\phantom{=.}+\int_0^1\int_0^\infty\lambda_\infty(x)\left|\eta_x(f,u)-\tilde\eta_x^d(f,u)\right|\ \mathrm du\ \mathrm dx.
\end{align} 
We bound the integrals appearing in \eqref{to be bounded1} and \eqref{to be bounded2} separately.

For \eqref{to be bounded1}, under the stability conditions, the collection of cluster durations $(D_x)_{x\in[0,1]}$ of the clusters $(S_x)_{x\in[0,1]}$ is a uniformly tight collection of random variables, since the expected cluster sizes are uniformly bounded by $\mathfrak K$, by Lemma~\ref{lemma offspring size}, and since the excitation function is not dependent on the spatial coordinate. 
Here, we use our assumption $\int_0^\infty th(t)\ \mathrm dt<\infty$ in order to have times between the birth of a parent and the birth of a child --- conditionally upon the creation of a child ---
%\footnote{I.e., conditionally on the creation of a child, this is the time between the birth of the child and the birth of the parent.} 
of bounded expectation.
Under the conditions of Theorem~\ref{thm stability general X spectral radius}, we have a well-defined, stable system, assuring that $\tilde{\mathcal L}^d_Q(f,\infty)>0$ for $d$ sufficiently large, by Lemma~\ref{lemma offspring size}. 
From the formula \eqref{queue into cluster laplace} for $\tilde{\mathcal L}^d_Q(f,\infty)$, it follows that $\int_0^\infty\left|\tilde \eta_x^d(f,u)-1\right|\ \mathrm du$ is convergent a.e.\ on $\{x:\lambda_\infty(x)>0\}$, since by Assumption~\ref{ass8} the baseline intensity $\lambda_\infty$ is bounded below by some $c_\infty$ on the set $\{x:\lambda_\infty(x)>0\}$;
here we use that $f$ is a \emph{positive} function. 
By this reasoning, we can even find a bound on $\int_0^\infty\left|\tilde\eta_x^d(f,u)-1\right|\ \mathrm du$ uniformly over $x\in\{y:\lambda_\infty(y)>0\}$.
It then follows by the $L^1$ convergence of $\tilde\lambda_\infty^d$ to $\lambda_\infty$ (Lemma~\ref{TVlemma4}) that \eqref{to be bounded1} tends to $0$, as $d\to\infty$.

%For (\ref{to be bounded2}), we basically want to prove convergence of the cluster processes transforms. We have seen in Theorems \ref{convergenceJ1}, \ref{convergenceJ1sampled graphs} that the multivariate Hawkes processes with either sampled edges of using weighted graphs converges weakly w.r.t.\ the Skorokhod topology on $[0,\infty)$ to the graphon Hawkes processes having the appropriate parameters.  From this, convergence of the cluster processes can be inferred. 

Next, to bound \eqref{to be bounded2} by $\epsilon$, it suffices to bound %the inner integral 
$\int_0^\infty\left|\eta_x(f,u)-\tilde\eta_x^d(f,u)\right|\ \mathrm du$ by $\epsilon/\alpha$, uniformly over $x\in[0,1]$. %Let $C_f=\|f\|_{L^\infty[0,1]}$; recall that $f\in\mathrm{BM}_+[0,1]$. 
By the mean value theorem,
\begin{align*}
\int_0^\infty\left|\eta_x(f,u)-\tilde\eta_x^d(f,u)\right|\ \mathrm du&=
\int_0^\infty\left|\mathbb E\left[\exp(-f(S_x,u))\right]-\mathbb E[\exp(-f(\tilde S_x^d,u))]\right|\ \mathrm du\\
&\leq\int_0^\infty\mathbb E\left|\exp(-f(S_x,u))-\exp(-f(\tilde S_x^d,u))\right|\ \mathrm du\\
&\leq \int_0^\infty\mathbb E\left|f(S_x,u)-f(\tilde S_x^d,u)\right|\ \mathrm du.
\end{align*}

Using Lemma~\ref{lemma offspring size}, we select $D\in\mathbb N$, $\mathfrak K\in\mathbb R_+$ such that the expected cluster sizes of $N$, $\tilde N^d$ are uniformly bounded by $\mathfrak K$, uniformly over $d\geq D$, $x\in[0,1]$. 
Then the expected discrepancy in $\mathfrak d$-distance between a cluster $S_x$ and its coupled cluster $\tilde S_x^d$ of $\tilde N^d$ can be bounded as in (ii) from the proof of Theorem~\ref{annealed theorem}: we can find a $d_1\geq D$ such that for $d\geq d_1$ the expected discrepancy in $\mathfrak d$-distance between the two coupled clusters is bounded by $$\epsilon':=\frac{\epsilon}{4\alpha C_f\mathfrak K\mathbb E[J]},$$ hence with probability of at least $1-\epsilon'$, the two clusters $S_x$ and $\tilde S_x^{d}$ only consist of simultaneous events. 
Now we use the uniform continuity of $f$ to find $\delta>0$ such that $$|f(x)-f(y)|<\epsilon'':=\frac{\epsilon}{2\alpha\mathfrak K\mathbb E[J]}$$ for $|x-y|<\delta$. 
Next, we select $d_2\geq d_1$ such that $\mathrm{mesh}(\mathcal P^{d_2})<\delta$.

Given that we only have simultaneous events within the coupled clusters $S_x$ and $\tilde S_x^{d}$, which is an event with probability of at least $1-\epsilon'$ for $d\geq d_2$, $\int_0^\infty\mathbb E|f(S_x,u)-f(\tilde S_x^d,u)|\ \mathrm du$ can be bounded by the product of: the expected number of events within the cluster $S_x$, the expected duration of an event, and $\epsilon''$ ---  since the simultaneous events are within the same element $\mathscr X_n^d$ of the partition $\mathcal P^d$. 
Here, we use the uniform continuity argument from the previous paragraph and the fact that simultaneous events are in the same $\mathscr X_n^d$, which follows by the construction presented in Section~\ref{section existence etc}.
Hence, given this event with probability of at least $1-\epsilon'$, $\int_0^\infty\mathbb E|f(S_x,u)-f(\tilde S_x^d,u)|\ \mathrm du$ can be bounded by $\mathfrak K\mathbb E[J]\epsilon''<\epsilon/(2\alpha)$. 

Next, with probability of at most $\epsilon'$, the clusters $S_x$ and $\tilde S_x^{d}$ are not equal, in which case, similarly as before, for $d\geq D$, $\int_0^\infty\mathbb E|f(S_x,u)-f(\tilde S_x^d,u)|\ \mathrm du$ can be bounded by $2C_f\mathfrak K\mathbb E[J]$. 
To conclude, for $d\geq d_2$, $\int_0^\infty\left|\eta_x(f,u)-\tilde\eta_x^d(f,u)\right|\ \mathrm du$ can be bounded by $\epsilon/\alpha$.

This proves that \eqref{to be bounded2} is bounded by $\epsilon$, for $d\geq d_2$. 
We already showed that \eqref{to be bounded1} tends to $0$, as $d\to\infty$, hence pointwise convergence of $\tilde{\mathcal L}_Q^d(f,\infty)$ to $\mathcal L_Q(f,\infty)$ follows, for $f\in C_+[0,1]$. 
Invoking again \cite{nonlifeinsurancemath}, Theorem~9.1.4, it follows that $\tilde Q^d(\infty)\to Q(\infty)$ weakly in the space of all point measures on $[0,\infty)\times[0,1]$, as $d\to\infty$. 
\end{proof}

\section{Concluding remarks}
In this paper, we have introduced the \textit{graphon Hawkes process}. 
Working on a corresponding uncountable space introduces the need for functional-analytic techniques that are unnecessary in the discrete space setting. 
%We use coupling techniques on the cluster representation that are, to the best of our knowledge, new to the literature, and may be applied more broadly. 
After establishing its existence, uniqueness and stability, this graphon Hawkes process is shown to occur as the limit of suitably chosen multivariate Hawkes processes. %The other way around, we have shown that any sequence of multivariate Hawkes processes with convergent parameters converges to a graphon Hawkes process. 
Along the way, we have established technical Lemmas~\ref{lemma offspring size} and \ref{TVlemma4}, which may be interesting in their own right. 
Next, we have provided a dichotomy in large-time behavior depending on whether $\rho(T_{\mathrm{hom}})<1$ or $\rho(T_{\mathrm{hom}})>1$, and we have established an FLLN and an FCLT in the stable case. 
By reducing the problem to a finite-dimensional setting, our proofs of the FLLN and FCLT become more tractable, facilitating the derivation of functional limit theorems for infinite-dimensional systems. 
This method, given its effectiveness, may also benefit other infinite-dimensional models. 
Finally, we have characterized the distribution of the graphon Hawkes process by fixed-point equations in the transform domain, and we have used this characterization to establish that, for convergent multivariate Hawkes population processes $\tilde Q^d(t)$ in the stable regime, the limits $d\to\infty$ and $t\to\infty$ commute, an application that we believe is both interesting and insightful.

Several follow-up questions and extensions are conceivable.
\begin{itemize}
%\item In case of an unbounded spatial set $\mathscr X$, one could investigate whether existence and uniqueness can be proved under the continuous-space counterparts of the stability conditions found in \cite[Assumption 4]{Delattre16}; see Remark \ref{remark stability delattre}.
\item As we have shown, it is possible to define the graphon Hawkes process for nonlinear conditional intensity density functions. 
We wonder whether the convergence results of Theorems~\ref{annealed theorem} and \ref{quenched theorem} carry over to the nonlinear case. 
This would, however, require an intrinsically different approach than the one followed in this work, since one loses the branching representation and the additive structure of the conditional intensity density.
%\item Theorems \ref{annealed theorem} and \ref{quenched theorem} asssume that we start the systems on an empty history. We'd expect the same results to hold if we start the systems with the same, but possibly nonempty histories, assuming sufficient regularity on those histories.
If we consider nonlinear models with $f_x(z)=f(\lambda_\infty(x)+z)$ for a \emph{single} Lipschitz continuous function $f$, then by a straightforward modification of the proof of Theorem~\ref{annealed theorem} --- first bounding by $c_x\equiv c$, the Lipschitz constant of $f$ ---, we are able to prove convergence of the multivariate processes, obtained by averaging parameters, to the corresponding graphon Hawkes processes. %; here for stability, we should assume that $\rho(T_{\mathrm{hom}})<c^{-1}$.
%If one would use a general form of the nonlinear functions $f_x$, see (\ref{eq cond int dens def}), this causes problems when we want to take averages of parameters. More specifically, we would want to take the average of both the functions $f_x$ over $\mathscr X_n^d$, but also the average of $B$ and $W$. %However, as $B$ and $W$ are arguments of $f_x$, we would then take averages twice. 
We note that nonlinearity of the form $c_x\equiv c$ is often assumed in related literature; see e.g., \cite{Agathe-Nerine}.
%\item [Delete?] Throughout this work, we have assumed that the mark stochastic process $B$ and the graphon $W$ are uniformly bounded, see Assumption \ref{ass cw cb}. For similar reasons, we have assumed that $\mathscr X$ is bounded. From Section \ref{section convergence results} onwards, our proofs rely crucially on those boundedness assumptions. Without those assumptions, our proofs fail, since essentially excitation might `come from infinity' in such models, cf.\ \cite{Delattre16}. Convergence results for such models would state that any (sufficiently regular) graphon Hawkes process on an unbounded set $\mathscr X$ can be obtained as the limit of Hawkes processes on countable networks, as introduced in \cite{Delattre16}. This could provides an interesting path for further research, but is outside the scope of the current work.
\item The functional limit theorems in the stable regime exploit the convergence results of Section~\ref{section convergence results} and related results for multivariate Hawkes processes. 
We would be interested in direct proofs for our FLLN and FCLT as well, because those proofs might be interesting in their own rights. %, and because they could allow for extensions, for example to marked processes. %Furthermore, it may be possible to prove Conjecture \ref{conjecture1} this way.
%The expression for $\sigma_A$ given in Theorem \ref{FCLT theorem} for multivariate Hawkes processes holds for graphon Hawkes processes with non-piecewise constant parameters as well, i.e., infinite-dimensional graphon Hawkes processes.
\item In this paper, we have established both an FLLN and an FCLT. 
Given those limit theorems, one may be interested in establishing %(sample path) 
\emph{large deviation principles} for the graphon Hawkes process. 
Since we are working with an infinite-dimensional model, we expect that this requires more sophisticated, functional-analytic tools than those required for finite-dimensional Hawkes processes (\cite{ldpKLM}). 
We intend to pursue this in future work.
%\item In Sections \ref{large time behavior section}-\ref{section6 fixed point theorems}, we work with $\mathscr X=[0,1]$, but, as noted before, this is for legibility only: the results of those sections carry over without any complications to $\mathscr X=[\boldsymbol a,\boldsymbol b]\subset\mathbb R^m$.

\end{itemize}

\section*{Appendix A: Relegated proofs of and supplement to Sections~\ref{sectiondefs} and~\ref{section existence etc}}\label{app A}

{\sc Proof of Lemma~\ref{sampling lemma}}.
Denote the random spatial coordinate by $X$. 
Under procedure~(i),
\begin{align*}
\mathbb P(X\in\mathrm dx)&=\frac{\|\lambda^1(\cdot)\|_{L^1(\mathscr X)}}{\|\lambda^1(\cdot)+\lambda^2(\cdot)\|_{L^1(\mathscr X)}}\frac{\lambda^1(x)\mathrm dx}{\|\lambda^1(\cdot)\|_{L^1(\mathscr X)}}+\frac{\|\lambda^2(\cdot)\|_{L^1(\mathscr X)}}{\|\lambda^1(\cdot)+\lambda^2(\cdot)\|_{L^1(\mathscr X)}}\frac{\lambda^2(x)\mathrm dx}{\|\lambda^2(\cdot)\|_{L^1(\mathscr X)}}\\
&=\frac{\lambda^1(x)\mathrm dx+\lambda^2(x)\mathrm dx}{\|\lambda^1(\cdot)+\lambda^2(\cdot)\|_{L^1(\mathscr X)}}=\frac{\lambda(x)\mathrm dx}{\|\lambda(\cdot)\|_{L^1(\mathscr X)}},
\end{align*} 
which is the law used under procedure~(ii).
$\hfill\Box$\newline

{\sc Proof of Theorem~\ref{thm stability general X spectral radius}}.
First, the proof constructs a Picard scheme to establish existence. %, as in the proof of \cite{Massoulie}, Theorem~2. 
Next, for the uniqueness part, we only prove the key modifications required in our general setting compared to \cite{Massoulie}.

\emph{Existence}. We use Picard's iteration to construct mappings $\{\lambda^n(t,z)\}$ and spatiotemporal point processes $N^n$, for $n\in\mathbb N_0$. 
Here, $z=(z_1,z_2,z_3)=(B,i,U)$ denotes the mark random variable: the first component denotes the random mark process $B$; % which is defined on the probability space $(L_B,\mathcal L_B,\mathcal Q_B)$; 
the second component denotes the site index $z_2=i$, defined on $(\mathbb N,\mathcal P(\mathbb N),\sigma)$, where $\sigma$ denotes the counting measure; and the final component denotes the sequences of random variables $(U_i)_{i\in\mathbb N}$ determining the spatial coordinate; see the construction in Section~\ref{section existence etc}. 
Let $L\ni z$ denote the mark space. 
For the Picard scheme, we set $\lambda^0(t,z,x)\equiv0$, $\Lambda^0(t,z)\equiv 0$, we let every $N^n$ coincide with $N$ on $\mathbb R_-\times L$, and for $n\in\mathbb N_0$, $t>0$ and $z=(B,i,U)\in L$,
\begin{align}
N^n(\mathrm dt\times\mathrm dz)&=\bar N^i\left(\mathrm dt\times\mathrm d(B\times U)\times\left[0,\Lambda^n(t,i)\right]\right);\\
\Lambda^{n+1}(t,i)&=\int_{\mathscr X_i}\lambda^{n+1}(t,x)\ \mathrm dx\\&=\int_{\mathscr X_i}f_x\left(\lambda_\infty(x)+\sum_{\substack{(s,B,y)\in N^n\\s<t}}B_{xy}(s)W(x,y)h(t-s)\right)\ \mathrm d\mu(x),\nonumber
\end{align}
with the understanding that the locations $y$ are determined by the site index and by the random marks $U$, as outlined in Section~\ref{section existence etc}.
Then it holds that
\begin{align}
&\sup_{t\geq0,i\in\mathbb N}\mathbb E\int_{\mathscr X_i}\left|\lambda^{n+1}(t,x)-\lambda^{n}(t,x)\right|\ \mathrm d\mu(x)\nonumber\\
&\leq\sup_{t\geq0,i\in\mathbb N}\mathbb E\int_{\mathscr X_i}\int_{(0,t)}\int_{\mathscr X}c_xB_{xy}(s)W(x,y)h(t-s)\ |N_s^n-N_s^{n-1}|(\mathrm ds\times\mathrm dy)\ \mathrm d\mu(x)\nonumber\\
&\leq\sup_{s\geq0,i\in\mathbb N}\int_{\mathscr X_i}\|h\|_{L^1(\mathbb R_+)}c_x\int_{\mathscr X}\mathbb E[B_{xy}]W(x,y)\mathbb E\left|\lambda^n(s,y)-\lambda^{n-1}(s,y)\right|\ \mathrm dy\ \mathrm d\mu(x)\nonumber\\
&=\sup_{s\geq0,i\in\mathbb N}\int_{\mathscr X_i}T_{\mathrm{hom}}\left(\mathbb E\left|\lambda^n(s,y)-\lambda^{n-1}(s,y)\right|\right)(x)\ \mathrm d\mu(x)\nonumber\\
&\leq\cdots\leq\sup_{t\geq0,i\in\mathbb N}\int_{\mathscr X_i}T_{\mathrm{hom}}^n\left(\mathbb E\lambda^1(t,\cdot)\right)(x)\ \mathrm d\mu(x)\nonumber\\
&=\sup_{t\geq0,i\in\mathbb N}\int_{\mathscr X_i}T_{\mathrm{hom}}^n(f_\cdot(\lambda_\infty(\cdot)+\eta(t,\cdot)))(x)\ \mathrm d\mu(x)\nonumber\\
&=\sup_{t\geq0,i\in\mathbb N}\|T_{\mathrm{hom}}^n(f_\cdot(\lambda_\infty(\cdot)+\eta(t,\cdot)))\|_{L^1(\mathscr X_i,\mu)}\nonumber\\
&\leq\|T_{\mathrm{hom}}^n\|\sup_{t\geq0,i\in\mathbb N}\|f_\cdot(\lambda_\infty(\cdot)+\eta(t,\cdot))\|_{L^1(\mathscr X_i,\mu)}\to0\label{exponential convergence Picard scheme}
%&\left(\alpha+\sup_{t\geq0,i\in\mathbb N}\|\eta(t,\cdot)\|_{L^1(\mathscr X_i,\mu)}\right)\to0,
\end{align} 
exponentially, as $n\to\infty$, using Gelfand's formula and our assumption that $\rho(T_{\mathrm{hom}})<1$.
More specifically, by Gelfand's formula we can find some $r\in(\rho(T_{\mathrm{hom}}),1)$ and some $K\in\mathbb N$ such that for all $n\geq K$ it holds that $\|T_{\mathrm{hom}}^n\|\leq r^n$. %Write 
%\begin{equation}
%\label{def C proof existence}C:=\sup_{t\geq0,i\in\mathbb N}\|f_\cdot(\lambda_\infty(\cdot)+\eta(t,\cdot))\|_{L^1(\mathscr X_i,\mu)}<\infty.
%\end{equation}
For such large $n\geq K$, it thus holds that $$\sup_{t\geq0,i\in\mathbb N}\mathbb E\int_{\mathscr X_i}\left|\lambda^{n+1}(t,x)-\lambda^{n}(t,x)\right|\ \mathrm d\mu(x)\leq Cr^n.$$ 
Letting $r'\in(r,1)$, Markov's inequality gives that $$\sup_{t\geq0,i\in\mathbb N}\mathbb P\left(\int_{\mathscr X_i}\left|\lambda^{n+1}(t,x)-\lambda^{n}(t,x)\right|\ \mathrm d\mu(x)\geq (r')^n\right)\leq C\left(\frac r{r'}\right)^n.$$ 
By an application of the Borel-Cantelli lemma, it holds that, with probability $1$, only finitely many of those events $$\left\{\int_{\mathscr X_i}\left|\lambda^{n+1}(t,x)-\lambda^{n}(t,x)\right|\ \mathrm d\mu(x)\geq (r')^n\right\}$$ occur, hence for any $t\geq0,i\in\mathbb N$, it holds that, with probability $1$, $(\lambda^n(t,\cdot))_{n\in\mathbb N_0}$ is a Cauchy sequence in $L^1(\mathscr X_i,\mu)$; by completeness, this sequence has some ($t$-uniform) $L^1(\mathscr X_i,\mu)$-limit $\lambda(t,\cdot)$.

Note that our limit process $\lambda(t,x)$ is positive a.s. 
Furthermore, it satisfies
\begin{align}
\sup_{t\geq0,i\in\mathbb N}\mathbb E\int_{\mathscr X_i}\lambda(t,x)\ \mathrm d\mu(x)&\leq\sup_{t\geq0,i\in\mathbb N}\mathbb E\left[\sum_{n\geq0}\int_{\mathscr X_i}|\lambda^{n+1}(t,x)-\lambda^n(t,x)|\ \mathrm d\mu(x)\right] \nonumber\\
&\leq (1-r)^{-1}\sup_{t\geq0,i\in\mathbb N}\mathbb E\left[\int_{\mathscr X_i}|\lambda^K(t,x)-\lambda^0(t,x)|\ \mathrm d\mu(x)\right]\nonumber\\
&= (1-r)^{-1}\sup_{t\geq0,i\in\mathbb N}\mathbb E\int_{\mathscr X_i}\lambda^K(t,x)\ \mathrm d\mu(x)\nonumber\\
&\leq(1-r)^{-1}\|T_{\mathrm{hom}}\|^{K-1}\sup_{t\geq0,i\in\mathbb N}\mathbb E\int_{\mathscr X_i}\lambda^1(t,x)\ \mathrm d\mu(x)\nonumber\\
%&\leq(1-r)^{-1}(C_BC_WC_{\mathrm{Lip}}\|h\|_{L^1(\mathbb R_+)})^{K-1}\sup_{t\geq0,i\in\mathbb N}\mathbb E\int_{\mathscr X_i}\lambda^1(t,x)\ \mathrm d\mu(x)\nonumber\\
&\leq (1-r)^{-1}\|T_{\mathrm{hom}}\|^{K-1} \cdot C<\infty.\label{bounded mean intensity}
\end{align}
%where $C$ is defined in (\ref{def C proof existence}). 
This means that the conditional intensity for the arrival of some point anywhere in $\mathscr X_i$ is a.s.\ bounded, uniformly over $t\in\mathbb R_+,i\in\mathbb N$.

Next, take some set $D\in\mathcal L_1\otimes\mathcal P(\mathbb N)\otimes\mathcal L_2$ with $(\mathcal Q_1\times\sigma\times Q_2)(D)<\infty$. % We consider the probability of some nonsimultaneous point \emph{or} a point with a different mark $B$, ignoring the uniform marks in $L_1$, i.e., ignoring the spatial coordinate: 
Then, 
\begin{align*}
&\sum_{n\geq0}\mathbb P\left(|N^{n+1}-N^n|((0,T]\times D)\neq0\right)\leq\sum_{n\geq0}\mathbb E|N^{n+1}-N^n|((0,T]\times D)\\
&\leq\mathbb Q_2(D)\sum_{n\geq0}\int_{(0,T]}\sup_{i\in\mathbb N}\mathbb E\int_{\mathscr X_i}|\lambda^{n+1}(t,x)-\lambda^n(t,x)|\ \mathrm d\mu(x)\ \mathrm dt\\
&=\mathbb Q_2(D)\int_{(0,T]}\sum_{n\geq0}\sup_{i\in\mathbb N}\mathbb E\int_{\mathscr X_i}|\lambda^{n+1}(t,x)-\lambda^n(t,x)|\ \mathrm d\mu(x)\ \mathrm dt\\
&\leq T\mathbb Q_2(D)\sum_{n\geq0}\sup_{t\geq0,i\in\mathbb N}\mathbb E\int_{\mathscr X_i}|\lambda^{n+1}(t,x)-\lambda^n(t,x)|\ \mathrm d\mu(x)<\infty,
\end{align*}
by the exponential convergence in (\ref{exponential convergence Picard scheme}). 
The interchange of the summation and integration is justified by monotone convergence. 
By Borel-Cantelli, it follows that, with probability $1$, on $(0,T]\times D$, the processes $N^n$ are eventually constant, and in this sense they converge to some limiting point process $N$, as $n\to\infty$. 
From our Picard scheme, it follows that on $(0,T]\times D$, with probability $1$, the mappings $\lambda^n$ are eventually constant as well. 
This means that the spatial coordinate as determined from $\lambda^n$, for sufficiently large $n$, corresponds to the one determined by $\lambda$. 
Furthermore, by Lemma~\ref{sampling lemma}, we can define spatial coordinates in a meaningful way even when we only have $\lambda^n$.

The limiting process $N$ satisfies 
\begin{equation}N(\mathrm dt\times\mathrm d(B\times i\times U))=\bar N^i\left(\mathrm dt\times\mathrm d(B\times U)\times\left[0,\int_{\mathscr X_i}\lambda(t,x)\ \mathrm d\mu(x)\right]\right).
\end{equation} 
Indeed, for $T>0$ and $D\in\mathcal L_1\otimes\mathcal P(\mathbb N)\otimes\mathcal L_2$ with $(\mathcal Q_1\times\sigma\times Q_2)(D)<\infty$, by Fatou's lemma,
\begin{align}
&\mathbb E\int_{(0,T]\times D}\Bigg|N(\mathrm dt\times\mathrm d(B\times i\times U))-\bar N^i\left(\mathrm dt\times\mathrm d(B\times U)\times\left[0,\int_{\mathscr X_i}\lambda(t,x)\ \mathrm d\mu(x)\right]\right)\Bigg|\nonumber\\
&\leq \lim_{n\to\infty}\mathbb E\int_{(0,T]\times D}\bigg|\bar N^i\left(\mathrm dt\times\mathrm d(B\times U)\times\left[0,\int_{\mathscr X_i}\lambda^n(t,x)\ \mathrm d\mu(x)\right]\right)\nonumber\\
&\phantom{\leq\   \lim_{n\to\infty}\int_{(0,T]\times D}}-\bar N^i\left(\mathrm dt\times\mathrm d(B\times U)\times\left[0,\int_{\mathscr X_i}\lambda(t,x)\ \mathrm d\mu(x)\bigg]\right)\right|\nonumber\\
&\leq T\mathbb Q_2(D)\lim_{n\to\infty}\sup_{t\geq0,i\in\mathbb N}\mathbb E\int_{\mathscr X_i}|\lambda^n(t,x)-\lambda(t,x)| \ \mathrm d\mu(x)=0.
\end{align} 

Finally, we verify that the limiting process $\lambda$ is indeed the intensity density of $N$. 
Let $$\breve\lambda(t,x):=f_x\left(\lambda_\infty(x)+\sum_{\substack{(s,B,y)\in N\\s<t}}B_{xy}(s)W(x,y)h(t-s)\right)$$ be the intensity density that we expect to find. 
For each $i\in\mathbb N$, we have 
\begin{align}
&\mathbb E\int_{\mathscr X_i}|\lambda(t,x)-\breve\lambda(t,x)|\ \mathrm d\mu(x)\nonumber\\
&\leq\mathbb E\int_{\mathscr X_i}|\lambda(t,x)-\lambda^{n}(t,x)|\ \mathrm d\mu(x)\nonumber\\&+\mathbb E\int_{\mathscr X_i}\int_{(0,t)}\int_{\mathscr X}c_xB_{xy}(s)W(x,y)h(t-s)\ |N_s-N_s^{n-1}|(\mathrm ds\times\mathrm dy)\ \mathrm d\mu(x)\nonumber\\
&\leq\mathbb E\int_{\mathscr X_i}|\lambda(t,x)-\lambda^{n}(t,x)|\ \mathrm d\mu(x)\nonumber\\
&+\|h\|_{L^1(\mathbb R_+)}\int_{\mathscr X_i}c_x\mathbb E[B_{xy}]W(x,y)\ \mathrm d\mu(x)\sup_{t\geq0}\mathbb E\int_{\mathscr X}|\lambda(t,y)-\lambda^{n-1}(t,y)|\ \mathrm d\mu(y).
\end{align}
The right-hand side goes to $0$ uniformly in $t$ and in $i\in\mathbb N$, as $n\to\infty$, hence $(N,\lambda)$ is a solution to \eqref{N wrt PRM}, as desired. 
It satisfies $\sup_{t\geq0,i\in\mathbb N}\mathbb E\psi(S_tN_-,i)<\infty$ by  \eqref{bounded mean intensity}, hence this solution is strongly regular, as desired.

\emph{Uniqueness}. We need a counterpart to \cite{Massoulie}, Lemma 4, after which we may simply follow the proof of the uniqueness part of  \cite{Massoulie}, Theorem 2. 
More specifically, we seek a strictly positive function $g$ on $L$ such that $\int_Lg(z)\ \mathcal Q(\mathrm dz)<\infty$ and %, for $\lambda'$ the intensity of another strongly regular solution to (\ref{N wrt PRM}),
%, adopting the notation from the uniqueness part of the proof of \cite[Theorem 2]{Massoulie}, 
%\begin{align*}&\phantom{\leq}\mathbb E\int_{(0,\tau)\times L}g(z)|\lambda(t,y(z))-\lambda'(t,y(z))|\ \mathrm dt\ \mathcal Q(\mathrm dz)\\&\leq r\ \mathbb E\int_{(0,\tau)\times L}g(z)|\lambda(t,y(z))-\lambda'(t,y(z))|\ \mathrm dt\ \mathcal Q(\mathrm dz),\end{align*} for some $r<1$. This inequality can be proven in an analogous manner as in the proof of the uniqueness part of  \cite{Massoulie}, Theorem 2, although here we need the existence of a strictly positive function $g$ on $L$ 
such that the inequality $$\int_{\mathbb R_+\times L}c_{x(z')}h(t)B_{x(z')y(z)}(z)W(x(z'),y(z))g(z)\ \mathrm dt\ \mathcal Q(\mathrm dz)\leq rg(z')$$ holds. We focus on proving this inequality. 
In the previous display, the notations $x(z')$, $y(z)$ and $B(z)$ emphasize the dependency of the location and mark on the random $z'\in L$. 
Let $g_0$ be the identity function on $L$, and for $n\in\mathbb N_0$, set $$g_{n+1}(z')=\int_{\mathbb R_+\times L}c_{x(z')}h(t)B_{x(z')y(z)}(z)W(x(z'),y(z))g_n(z)\ \mathrm dt\ \mathcal Q(\mathrm dz).$$  
By a calculation similar to \eqref{exponential convergence Picard scheme}, combined with Gelfand's formula, given $\tilde r\in(\rho(T_{\mathrm{hom}}),1)$, we can find $N\in\mathbb N$, such that for all $n\geq N$ it holds that $\int_L g_{n+1}(z)\ \mathcal Q(\mathrm dz)\leq \tilde r\int_L g_{n}(z)\ \mathcal Q(\mathrm dz)$. 
Hence, for $n\geq N$, 
$$
\int_L g_{n}(z)\ \mathcal Q(\mathrm dz)\leq \tilde r^{n-N}\|T_{\mathrm{hom}}^{N}\|\int_L g_{0}(z)\ \mathcal Q(\mathrm dz)<\infty.$$ 
\begin{comment}
 $\int_{\mathbb R_+\times L}\bar h(t,z',z)g(z)\ \mathcal Q(\mathrm dz))\leq rg(z')$ for  $z'\in L$. 
%Here $\bar h(t,z',z)=\bar h(t,(B,U,j),i)=h(t)\int_{\mathscr X_i}c_x$
The function $\bar h$ is defined below in (\ref{Lipschitzdominatingfunction}).
To this end, we define the functions $(g_n)_{n\in\mathbb N_0}$ on $L$ recursively by $g_{n+1}(z')=\int_{\mathbb R_+\times L}\bar h(t,z',z)g_n(z)\ \mathrm dt\ \mathcal Q(\mathrm dz)$, where for $g_0$ we may choose any strictly positive function on $L$ satisfying $\int_L g_0(z)\ \mathcal Q(\mathrm dz)$. By the assumption $\rho(T_{\mathrm{hom}})<1$ and an application of Gelfand's formula, we can find some $\tilde r\in(\rho(T_{\mathrm{hom}}),1)$ and some $N\in\mathbb N$ such that for all $n\geq N$ it holds that $\int_L g_{n+1}(z)\ \mathcal Q(\mathrm dz)\leq \tilde r\int_L g_{n}(z)\ \mathcal Q(\mathrm dz)$. Hence, for $n\geq N$, 
$$
\int_L g_{n}(z)\ \mathcal Q(\mathrm dz)\leq \tilde r^{n-N}(\|h\|_{L^1(\mathbb R_+)}C_BC_WC_{\mathrm{Lip}})^N\int_L g_{0}(z)\ \mathcal Q(\mathrm dz).$$ 
\end{comment}
Let $r\in(\tilde r,1)$. 
Set $g(z):=\sum_{n\geq N}r^{-n}g_n(z)$, which converges in $L^1(\mathcal Q)$. 
Then it holds that 
\begin{align*}&\int_{\mathbb R_+\times L}c_{x(z')}h(t)B_{x(z')y(z)}(z)W(x(z'),y(z))g(z)\ \mathrm dt\ \mathcal Q(\mathrm dz)\\&=\sum_{n\geq N}r^{-n}g_{n+1}(z')= r\sum_{n\geq N+1}r^{-n}g_n(z')\leq rg(z').\end{align*}
Using this $g$, the reader can readily verify that the uniqueness part of the proof of \cite{Massoulie}, Theorem~2, applies.
$\hfill\Box$
\newline

\begin{proposition}\label{thstationarity mathscr X}
Work in the setting of Theorem~\ref{thm stability general X spectral radius}.
Assume that the measure corresponding to the Poisson point processes $\bar N^i$ is compatible w.r.t.\ the left shifts $\{\theta_t\}$. 
Then there exists a unique stationary solution $N$ to \eqref{N wrt PRM} such that $\sup_{i\in\mathbb N}\mathbb E\psi(S_0N_-,i)<\infty$. 
Also, if $N'$ is the strongly regular non-stationary solution to \eqref{N wrt PRM} corresponding to an initial condition satisfying 
$$\sup_{t>0,i\in\mathbb N}\int_{\mathscr X_i}\eta(t,x)\ \mathrm d\mu(x)<\infty,\quad\text{and}\quad\forall i\in\mathbb N,\  \lim_{t\to\infty}\int_{\mathscr X_i}\eta(t,x)\ \mathrm d\mu(x)\to0,$$ then the law of $S_tN'$ converges weakly to the stationary law.
\end{proposition}

\begin{proposition}\label{existence mathscr X}
Grant Assumption~\ref{ass1}. % and \ref{ass cw cb}.  
Consider a (possibly nonlinear) graphon Hawkes process $N$  on a $\sigma$-finite measure space $(\mathscr X,\mathscr A,\mu)$, driven by the conditional intensity density specified by (\ref{N wrt PRM}).
Assume that we have a partition of $\mathscr X$ into finite, non-null measure sets $\bigsqcup_{i\in\mathbb N}\mathscr X_i$ that is such that the following conditions hold:
\begin{align}\label{rho mathscr X}
\rho&:=\|h\|_{L^1(\mathbb R_+)}\sup_{i\in\mathbb N}\sum_{j\in\mathbb N}\sup_{y\in\mathscr X_j}\int_{\mathscr X_i}c_x\mathbb E[B_{xy}]W(x,y)\ \mathrm d\mu(x)<1;\\
\label{alpha mathscr X}
\alpha&:=\sup_{i\in\mathbb N}\int_{\mathscr X_i}f_x(\lambda_\infty(x))\ \mathrm d\mu(x)<\infty.
\end{align}
Assume furthermore that $\epsilon(t,i)$ satisfies: $\sup_{t>0,i\in\mathbb N}\epsilon(t,i)<\infty$ and $\lim_{t\to\infty}\epsilon(t,i)\to0$ for all $i\in\mathbb N$, where $\epsilon(t,i)$ describes the effect of the initial condition on site $i$ at time $t$:
\begin{equation}\label{epsilon mathscr X}
\epsilon(t,i):=\sum_{j\in\mathbb N}\mathbb E\left[\int_{(-\infty,0)\times L_1}h(t-s)\sup_{y\in\mathscr X_j}\int_{\mathscr X_i}c_x B_{xy}W(x,y)\ \mathrm d\mu(x)N(\mathrm ds\times\mathrm dB)\right].
\end{equation} 
Then there exists a unique strongly regular solution $N$ to \eqref{N wrt PRM} such that $$\sup_{t\geq0,i\in\mathbb N}\mathbb E\psi(S_tN_-,i)<\infty.$$
\end{proposition}
\begin{proof}
We are working in the general, nonlinear framework described before Theorem~\ref{thm stability general X spectral radius}, with dynamics given by \eqref{N wrt PRM}. 
We fit this into the framework of \cite{Massoulie} as follows. 
Take a process $N$ admitting as its $\mathcal F_t^N$-intensity kernel $$\nu(t,\mathrm dB\times\mathrm di\times\mathrm dU)=\psi(S_tN_-,i)\mathcal Q(\mathrm dB\times\mathrm di\times\mathrm dU),$$ where, with $B_{xy}^i$ and $U^i$ as defined before, $$\mathcal Q(\mathrm dB\times\mathrm di\times\mathrm dU)=\mathcal Q_1(\mathrm dB_{xy}^i)\mathcal Q_2(\mathrm dU^i).$$
Using this notation, we apply \cite{Massoulie}, Theorem~2.

By the Lipschitz condition for $f_x$, and by Fubini's theorem, the Lipschitz condition from \cite{Massoulie}, eqn.\ (6), here reads: for each $i\in\mathbb N$,
\begin{align}\nonumber
&|\psi(S_0N_-,i)-\psi(S_0N'_-,i)|\\ \nonumber&=\Bigg|\int_{\mathscr X_i}\Bigg[
f_x\Bigg(\sum_{j\in\mathbb N}\sum_{\substack{(s,B_{xX_s^j(U_s)},U_s)\in N^j\\s<0}}B_{xX_s^j(U_s)}W(x,X_s^j(U_s))h(-s)\Bigg)\\
&\qquad-f_x\Bigg(\sum_{j\in\mathbb N}\sum_{\substack{(s,B_{xX_s^j(U_s)},U_s)\in (N')^j\\s<0}}B_{xX_s^j(U_s)}W(x,X_s^j(U_s))h(-s)\Bigg)\Bigg]\ \mathrm d\mu(x)\Bigg|\nonumber\\
&\leq \int_{(-\infty,0)\times L}\int_{\mathscr X_i}c_xB_{xX_s^j(U_s)}W(x,X_s^j(U_s))h(-s)|N-N'|(\mathrm ds\times\mathrm dB\times\mathrm dU\times\mathrm dj).\label{Lipschitzconditionpsi}
\end{align}
This equation is valid for all $i\in\mathbb N$, hence the Lipschitz dominating function from \cite{Massoulie}, eqn.\ (6), reads 
\begin{equation}\label{Lipschitzdominatingfunction}
\bar h(t,(B,j,U),i)=h(t)\int_{\mathscr X_i}c_xB_{xX_{-t}^j(U_{-t})}W(x,X_{-t}^j(U_{-t}))\ \mathrm d\mu(x).
\end{equation}
Now condition (10) from \cite{Massoulie}, Theorem~2, is satisfied if $\rho<1$.
%This condition requires, in particular, that some site $\mathscr X_i$ is not connected too strongly to the other sites $(\mathscr X_j)_{j\in\mathbb N}$, i.e., the excitement reaching each site from the whole space is limited. 
Next, condition (11) from \cite{Massoulie}, Theorem~2, requires that $\alpha<\infty$.
Finally, the function involved in \cite{Massoulie}, eqn.\ (12), can be bounded by $\epsilon(t,i)$.
\end{proof}

\section*{Appendix B: Relegated proofs of Section~\ref{section convergence results}}\label{app B}

{\sc Proof of Lemma~\ref{lemma offspring size}.}
We work with $\mathscr X=[\boldsymbol a,\boldsymbol b]$, a graphon Hawkes process $N$ with corresponding integral operator $$T_{\mathrm{hom}}:L^1(\mathscr X)\to L^1(\mathscr X):f(\cdot)\mapsto \|h\|_{L^1(\mathbb R_+)}\int_{\mathscr X}\mathbb E[B_{\cdot y}]W(\cdot, y)f(y)\ \mathrm dy,$$ and prelimit processes $\tilde N^d$ with averaged parameters as described in Section~\ref{section prelimit}, with corresponding integral operators $$\tilde T_{\mathrm{hom}}^{(d)}:L^1(\mathscr X)\to L^1(\mathscr X):f(\cdot)\mapsto \|h\|_{L^1(\mathbb R_+)}\int_{\mathscr X}\mathbb E[\tilde B^d_{\cdot y}]\tilde W^d(\cdot, y)f(y)\ \mathrm dy.$$
We first prove that $\|T_{\mathrm{hom}}-\tilde T_{\mathrm{hom}}^{(d)}\|\to0$ as $d\to\infty$. 
By the total variation bound of Lemma~\ref{TVlemma4}, and by Assumptions~\ref{ass7}, \ref{ass6}, we can select $d$ sufficiently large, i.e., $\mathrm{mesh}(\mathcal P^d)$ sufficiently small, such that both $$\sup_{y\in\mathscr X}\mathbb E\|B_{\cdot y}-\tilde B^d_{\cdot ,y}\|_{L^1(\mathscr X)}\quad\text{and}\quad\sup_{y\in\mathscr X}\|W(\cdot,y)-\tilde W^d(\cdot,y)\|_{L^1(\mathscr X)}$$ are bounded by $$\frac\epsilon{\|h\|_{L^1(\mathbb R_+)}(C_B+C_W)\mathrm{Leb}^m(\mathscr X)}.$$ 
Now let $f\in\ L^1(\mathscr X)$ with $\|f\|_{L^1(\mathscr X)}=1$. 
Then, by Assumption~\ref{ass cw cb}, 
\begin{align}\nonumber&\|(T_{\mathrm{hom}}-\tilde T_{\mathrm{hom}}^{(d)})f\|_{L^1(\mathscr X)}\\\nonumber&=
\left\|\|h\|_{L^1(\mathbb R_+)}\int_{\mathscr X}\left(\mathbb E[B_{\cdot y}]W(\cdot, y)-\mathbb E[\tilde B^d_{\cdot y}]\tilde W^d(\cdot, y)\right)f(y)\ \mathrm dy\right\|_{L^1(\mathscr X)}\\
\nonumber&\leq\|h\|_{L^1(\mathbb R_+)} \int_{\mathscr X}|f(y)|\int_{\mathscr X}\left|\mathbb E[B_{x y}]W(x, y)-\mathbb E[\tilde B^d_{xy}] W(x, y)\right|\mathrm dx\ \mathrm dy\\
\nonumber&+\|h\|_{L^1(\mathbb R_+)} \int_{\mathscr X} |f(y)|\int_{\mathscr X}\left|\mathbb E[\tilde B^d_{xy}] W(x, y)-\mathbb E[\tilde B^d_{xy}] \tilde W^d(x, y)\right|\mathrm dx\ \mathrm dy\\
&< \epsilon,\label{L1 norm operators}
\end{align}
whence $\|T_{\mathrm{hom}}-\tilde T_{\mathrm{hom}}^{(d)}\|\leq\epsilon$.

Now note that for $|\lambda|>\rho(T_{\mathrm{hom}})+\delta$, the resolvent $R(\lambda;T_{\mathrm{hom}})$ exists since $\lambda\notin\sigma(T_{\mathrm{hom}})$.  
Hence, for  $\|T_{\mathrm{hom}}-\tilde T_{\mathrm{hom}}^{(d)}\|<1/\|R(\lambda;T_{\mathrm{hom}})\|$, 
$$
\tilde T_{\mathrm{hom}}^{(d)}-\lambda I=T_{\mathrm{hom}}-\lambda I+\tilde T_{\mathrm{hom}}^{(d)}-T_{\mathrm{hom}}=\left(I-(\tilde T_{\mathrm{hom}}^{(d)}-T_{\mathrm{hom}})R(\lambda;T_{\mathrm{hom}})\right)(T_{\mathrm{hom}}-\lambda I)
$$
is a product of invertible operators, hence is invertible itself. 
We aim to show that for all $d$ sufficiently large such, $\lambda\notin\sigma(\tilde T_{\mathrm{hom}}^{(d)})$ for $|\lambda|>\rho(T_{\mathrm{hom}})+\delta$. 
To this end, we show that 
\begin{equation}\label{infimum norm resolvent}\inf_{|\lambda|>\rho(T_{\mathrm{hom}})+\delta}\left\{\frac1{\|R(\lambda;T_{\mathrm{hom}})\|}\right\}>0.
\end{equation}
By a Neumann series, we can write $$R(\mu;T_{\mathrm{hom}})=\sum_{n\geq0}R(\lambda;T_{\mathrm{hom}})^{n+1}(\lambda-\mu)^n,$$ hence for $|\lambda-\mu|<1/(2\|R(\lambda;T_{\mathrm{hom}})\|)$, it follows that $\|R(\mu;T_{\mathrm{hom}})\|\leq2\|R(\lambda;T_{\mathrm{hom}})\|$. 
Now it follows from the resolvent identity 
$$R(\lambda;T_{\mathrm{hom}})-R(\mu;T_{\mathrm{hom}})=(\lambda-\mu)R(\lambda;T_{\mathrm{hom}})R(\mu;T_{\mathrm{hom}})$$ that the resolvent $R(\lambda;T_{\mathrm{hom}})$ of a bounded linear operator $T_{\mathrm{hom}}$ is a continuous function of $\lambda$, with operator norm tending to $0$, as $\lambda\to\infty$. 
Indeed, again by a Neumann series, for $\lambda$ sufficiently large, $$\|R(\lambda;T_{\mathrm{hom}})\|=\bigg\|\sum_{n\geq0}T_{\mathrm{hom}}^n/\lambda^{n+1}\bigg\|\leq\frac1{\lambda-\|T_{\mathrm{hom}}\|}.$$ 
Now \eqref{infimum norm resolvent} follows. 
Find $D'$ large enough such that 
$$\|T_{\mathrm{hom}}-\tilde T_{\mathrm{hom}}^{(d)}\|<\inf_{|\lambda|>\rho(T_{\mathrm{hom}})+\delta}\left\{\frac1{\|R(\lambda;T_{\mathrm{hom}})\|}\right\}$$ for all $d\geq D'$. 
It follows that $\tilde T_{\mathrm{hom}}^{(d)}-\lambda I$ is invertible for all $|\lambda|>\rho(T_{\mathrm{hom}})+\delta$, hence $\rho(\tilde T_{\mathrm{hom}}^{(d)})\leq \rho(T_{\mathrm{hom}})+\delta<1$, for all $d\geq D'$.

%and by the continuous dependence on initial conditions of resolvents corresponding to Fredholm integral equations of the second kind, see \cite[Corollary 9.3.12]{Integral equations book}, it follows that for $d$ sufficiently large, the operators $T_{\mathrm{hom}}^{(d)}$ corresponding to $\tilde N^d$ admit resolvents that can be written as a Liouville-Neumann series. %For such $d$, by \cite[Corollary 9.3.18]{Integral equations book},  This implies that the spectral radius of the operator $T_{\mathrm{hom}}^{(d)}$ corresponding to the system $\tilde N^d$ is smaller than $1$. Hence for $d$ sufficiently large, the systems $\tilde N^d$ are stable. Suppose we start with an altered system where the kernels are multiplied beforehand by $\frac{\rho(T_{\mathrm{hom}})+1}{2\rho(T_{\mathrm{hom}})}>1$. Since $\frac{\rho(T_{\mathrm{hom}})+1}2<1$ those arguments still apply, and the piecewise constant approximations for the altered systems have corresponding spectral radii eventually bounded by $1$. From this it follows that we can find $D$ such that for all $d\geq D$, the spectral radius of the operator $T_{\mathrm{hom}}^{(d)}$ corresponding to the system $\tilde N^d$ is smaller than $\frac{2\rho(T_{\mathrm{hom}})}{\rho(T_{\mathrm{hom}})+1}<1$. 

We now prove that we can find some $D\in\mathbb N$ sufficiently large such that the expected cluster sizes $\mathbb E[Z_x^d]$ of $\tilde N^d$ of a cluster generated by a particle in $x\in\mathscr X$ are uniformly bounded over $x\in\mathscr X=[\boldsymbol a,\boldsymbol b]$ and $d\geq D$. 
Note that the maximum expected cluster size over $x\in\mathscr X$ for $\tilde N^d$ can be bounded by 
\begin{equation}\sup_{x\in\mathscr X}\mathbb E[Z_x^d]\leq\sum_{n\geq0}\|(\tilde T_{\mathrm{hom}}^{(d)})^n\|,\label{eqclustersize}\end{equation} hence it suffices to find $D\in\mathbb N$ such that the right-hand side of \eqref{eqclustersize} is uniformly bounded over $d\geq D$.
To this end, it suffices to prove that
\begin{equation}\label{eq1 bd cz}
\text{for all } n,\quad \limsup_{d\to\infty}\|(\tilde T_{\mathrm{hom}}^{(d)})^n\|\leq\|T_{\mathrm{hom}}^n\|.%,\quad\text{as }d\to\infty.
\end{equation} 
Indeed, fix $\epsilon>0$. 
Since $\rho(T_{\mathrm{hom}})<1$, we can find $N\in\mathbb N$ such that $\|T_{\mathrm{hom}}^N\|<1-\epsilon$. 
Assuming \eqref{eq1 bd cz}, we can find $D$ such that for all $d\geq D$ and $n\leq N$, $\|(\tilde T_{\mathrm{hom}}^{(d)})^n\|\leq \|T_{\mathrm{hom}}^n\|+\epsilon/2$. 
Hence,
\begin{align*}
\sum_{n\geq0}\|(\tilde T_{\mathrm{hom}}^{(d)})^n\|&=\sum_{k\geq0}\sum_{n=0}^{N-1}\|(\tilde T_{\mathrm{hom}}^{(d)})^{kN+n}\|\\
&\leq\sum_{k\geq0}\|(\tilde T_{\mathrm{hom}}^{(d)})^{N}\|^k\sum_{n=0}^{N-1}\|(\tilde T_{\mathrm{hom}}^{(d)})^n\|\\
&\leq\sum_{k\geq0}(1-\epsilon/2)^k\sum_{n=0}^{N-1}(\|T_{\mathrm{hom}}^n\|+\epsilon/2)<\infty,
\end{align*}
providing a uniform bound over $d\geq D$.

To prove \eqref{eq1 bd cz}, it suffices to prove
\begin{align}\label{eq2 bd cz}
\text{for all }d,n\in\mathbb N,\quad&\|(\tilde T_{\mathrm{hom}}^{(d)})^n\|\leq\|(\tilde T_{\mathrm{hom}}^{(d)})^{n-1}T_{\mathrm{hom}}\|;\\
\label{eq3 bd cz}
\text{for all }n\in\mathbb N,\quad&\|(\tilde T_{\mathrm{hom}}^{(d)})^{n-1}T_{\mathrm{hom}}-T_{\mathrm{hom}}^n\|\to0,\quad\text{as }d\to\infty.
\end{align}

For \eqref{eq2 bd cz}, let $P_d:L^1(\mathscr X)\to L^1(\mathscr X)$ be the conditional expectation operator onto the $\sigma$-algebra generated by the partition $\mathcal P^d$ used to define $\tilde N^d$ and $\tilde T_{\mathrm{hom}}^{(d)}$. 
Then $\|P_d\|\leq1$, see e.g., \cite{Measure theory}, Theorem~27.11(ii). 
Note that we can write $\tilde T_{\mathrm{hom}}^{(d)}=P_dT_{\mathrm{hom}}P_d$, hence it follows that 
$$
\|(\tilde T_{\mathrm{hom}}^{(d)})^n\|=\|(\tilde T_{\mathrm{hom}}^{(d)})^{n-1}P_dT_{\mathrm{hom}}P_d\|\leq\|(\tilde T_{\mathrm{hom}}^{(d)})^{n-1}P_dT_{\mathrm{hom}}\|=\|(\tilde T_{\mathrm{hom}}^{(d)})^{n-1}T_{\mathrm{hom}}\|.
$$
%proving (\ref{eq2 bd cz}).

%For $p,q\in[1,\infty]$, and for a transformation $T:L^p(\mathscr X)\to L^q(\mathscr X)$, write $\|T\|_{p\to q}$ for its operator norm $$\|T\|_{p\to q}:=\sup_{\|f\|_{L^p(\mathscr X)}=1}\|T(f)\|_{L^q(\mathscr X)}.$$ 
%Note that $\|\tilde T_{\mathrm{hom}}^{(d)}-T_{\mathrm{hom}}\|_{\infty\to1}\leq\|\tilde T_{\mathrm{hom}}^{(d)}-T_{\mathrm{hom}}\|_{1\to1}\to0$, as $d\to\infty$, by Assumption \ref{ass6}, Lemma \ref{TVlemma4} below, and by the triangle inequality. 
To prove \eqref{eq3 bd cz}, write 
$$(\tilde T_{\mathrm{hom}}^{(d)})^{n-1}T_{\mathrm{hom}}-T_{\mathrm{hom}}^n=\sum_{k=1}^{n-1}(\tilde T_{\mathrm{hom}}^{(d)})^{n-k-1}(\tilde T_{\mathrm{hom}}^{(d)}-T_{\mathrm{hom}})T_{\mathrm{hom}}^k.$$ 
By Assumption~\ref{ass cw cb}, $\|\tilde T_{\mathrm{hom}}^{(d)}\|,\|T_{\mathrm{hom}}\|\leq C:=\|h\|_{L^1(\mathbb R_+)}C_BC_W<\infty$. 
Hence, for all $k\in[n-1]$, 
\begin{align*}
&\|(\tilde T_{\mathrm{hom}}^{(d)})^{n-k-1}(\tilde T_{\mathrm{hom}}^{(d)}-T_{\mathrm{hom}})T_{\mathrm{hom}}^k\|\\
&\leq\|\tilde T_{\mathrm{hom}}^{(d)}\|^{n-k-1}\|\tilde T_{\mathrm{hom}}^{(d)}-T_{\mathrm{hom}}\|\|T_{\mathrm{hom}}\|^{k}\\
&\leq C^{n-1}\|\tilde T_{\mathrm{hom}}^{(d)}-T_{\mathrm{hom}}\|\to0,
\end{align*}
as $d\to\infty$, by \eqref{L1 norm operators}.%This establishes (\ref{eq3 bd cz}).
$\hfill\Box$  \newline

{\sc Proof of Lemma~\ref{TVlemma4}.}
We prove the result for $f\in\mathcal L_{\mathrm{loc}}^1(\mathbb R^m)$.
If we work on equivalence classes of functions, i.e., $f\in L_{\mathrm{loc}}^1(\mathbb R^m)$, take a representant $g$ of $f$, apply the bound \eqref{bound function on A} to $g$, and note that $\|f-f^d\|_{L^1(\mathscr X)}=\|g-g^d\|_{L^1(\mathscr X)}$. 
Taking the infimum over $g\in\mathcal L_{\mathrm{loc}}^1(\mathbb R^m)$ yields the result for equivalence classes of functions as well.

The assumptions made in the lemma imply that $\mathcal P^d$ partitions $\mathscr X$ into $K=K(d)$ hyperrectangles $\mathscr X_1^d,\ldots, \mathscr X_K^d$. 
We first prove the bound for each such hyperrectangle $\mathscr X_n^d$, $n\in[K]$. 
To this end, fix $\Omega:=\mathscr X_n^d$, and set $u:\Omega\to\mathbb R:x\mapsto f(x)-f^d(x)$; notice that $\int_{\Omega}u=0$. 
We use \cite{Leoni}, Theorem~13.9, to approximate $u$ by a sequence of smooth functions 
$$(\hat u_k)_{k\in\mathbb N}\subset W^{1,1}(\Omega)\cap C^\infty(\Omega)\quad\text{such that}\quad\hat u_k\stackrel{L^1(\Omega)}\longrightarrow u,\quad \mathrm{Var}(\hat u_k,\Omega)\to\mathrm{Var}(u,\Omega),$$ as $k\to\infty$. 
Here, $W^{1,1}(\Omega)\subset L^1(\Omega)$ denotes the Sobolev space of functions having weak derivatives in $L^1(\Omega)$.  
Set $$u_k=\hat u_k-\int_\Omega \hat u_k\in W^{1,1}(\Omega)\cap C^\infty(\Omega).$$ 
Since $\hat u_k\stackrel{L^1(\Omega)}\longrightarrow u$ and $\int_\Omega u=0$, it follows by the reverse triangle inequality on $L^1(\Omega)$ that $\|\hat u_k\|_{L^1(\Omega)}\to0$ as $k\to\infty$. 
Hence, $$\|u-u_k\|_{L^1(\Omega)}\leq\|u-\hat u_k\|_{L^1(\Omega)}+\|\hat u_k\|_{L^1(\Omega)}\to0\quad\text{as }k\to\infty,$$ i.e., $u_k\stackrel{L^1(\Omega)}\longrightarrow u$. 
After replacing $(u_k)_{k\in\mathbb N}$ by a subsequence, if necessary, we may assume that $u_k\to u$ a.s.
Furthermore, since the total variation does not change by adding a constant to a function, $\mathrm{Var}(u_k,\Omega)=\mathrm{Var}(\hat u_k,\Omega)\to\mathrm{Var}(u,\Omega)$, as $k\to\infty$. 
Therefore, we may assume without loss of generality that our approximating sequence satisfies $\int_\Omega u_k=0$ for all $k\in\mathbb N$.

We prove a bound for the approximating smooth functions $(u_k)_{k\in\mathbb N}$, which we then extend to $u$ itself. 
To this end, note that $u_k$ belongs to the Sobolev space $W^{1,1}$, being smooth. Since we are working on a convex domain $\Omega$ and since $\int_\Omega u_k=0$, \cite{Acosta}, Theorem~3.2, gives that 
\begin{equation}\label{bound W11}
\|u_k\|_{L^1(\Omega)}\leq\frac12\mathrm{diam}(\Omega)\|\nabla u_k\|_{L^1(\Omega)},
\end{equation}
where $\nabla$ denotes the gradient operator. 
It holds that $W^{1,1}\subset\mathrm{BV}(\Omega)$, and for $u_k\in W^{1,1}$ we have $\|\nabla u_k\|_{L^1(\Omega)}=\mathrm {Var}(u_k,\Omega)$, see \cite{Leoni}, \S13.2. 
Now Fatou's lemma and (\ref{bound W11}) imply
\begin{align}\label{bound BV}\nonumber
\|u\|_{L^1(\Omega)}&\leq\liminf_{k\to\infty}\|u_k\|_{L^1(\Omega)}\leq\lim_{k\to\infty}\frac12\mathrm{diam}(\Omega)\|\nabla u_k\|_{L^1(\Omega)}\\&=\lim_{k\to\infty}\frac12\mathrm{diam}(\Omega)\mathrm {Var}(u_k,\Omega)=\frac12\mathrm{diam}(\Omega)\mathrm {Var}(u,\Omega).
\end{align}
Again, since the total variation is insensitive to adding a constant to a function, this bound implies that 
\begin{equation}\label{bound BV f}
\|f-f^d\|_{L^1(\Omega)}\leq\frac12\mathrm{diam}(\Omega)\mathrm {Var}(f-f^d,\Omega)=\frac12\mathrm{diam}(\Omega)\mathrm {Var}(f,\Omega).
\end{equation}
Next, because the $L^1$-norm is additive and the total variation is superadditive over disjoint sets,
\begin{align}\nonumber
\|f-f^d\|_{L^1(\mathscr X)}&=\sum_{i=1}^K\|f-f^d\|_{L^1(\mathscr X_i^d)}\leq\sum_{i=1}^K\frac12\mathrm{diam}(\mathscr X_i^d)\mathrm {Var}(f,\mathscr X_i^d)\\
&\leq\frac12\mathrm{mesh}(\mathcal P^d)\sum_{i=1}^K\mathrm {Var}(f,\mathscr X_i^d)\leq\frac12\mathrm{mesh}(\mathcal P^d)\mathrm{Var}(f,\mathscr X),\label{bound function on A}
\end{align}
since we defined the mesh of a partition as the largest diameter among its elements.
$\hfill\Box$  \newline

\begin{lemma}\label{TVlemma3}
Let $[a,b]\subset\mathbb R$ be some closed interval. 
Assume that $a,b$ are endpoints of intervals $\mathscr X_n^d,\mathscr X_m^d$ from the partition $\mathcal P^d$ of $\mathbb R$ into intervals; if this is not the case, replace $\mathcal P^d$ by its coarsest refinement containing $a,b$ as endpoints. 
Assume that $\mathbb R$ is equipped with the Borel $\sigma$-algebra and a measure $\mu$ dominated by the Lebesgue measure. 
Consider a piecewise constant approximation $f^d$ of the $\mu$-locally integrable function $f:\mathbb R\to\mathbb R$, where $f^d$ is constant on each $\mathscr X_n^d$, with $f^d$ calculated as a measurable functional of $f|_{\mathscr X_n^d}$, and $$f^d(\mathscr X_n^d)\in\left[\essinf_{x\in\mathscr X_n^d}f(x),\esssup_{x\in\mathscr X_n^d}f(x)\right].$$ 
Let $\mathrm{mesh}(\mathcal P^d)$ be the maximum $\mu$-measure of all elements of $\mathcal P^d$. 
Then it holds that $f^d\in L^1([a,b],\mu)$, $f^d\to f$ a.e.\  as $\mathrm{mesh}(\mathcal P^d)\to0$, and finally, with $\|\cdot\|_{\mathrm{TV}[a,b]}$ the usual total variational norm over $[a,b]$, 
\begin{equation}\label{L1 TV ineq ab}
\|f^d-f\|_{L^1([a,b],\mu)}\leq \|f\|_{\mathrm{TV}[a,b]}\mathrm{mesh}(\mathcal P^d).
\end{equation}
\end{lemma}
%In this lemma, $f^d|_{\mathscr X_n^d}$ might be, for example, the infimum, the average, or the supremum over $\mathscr X_n^d$.
\begin{proof}
The first two claims are easy. 
For the final claim, the total variation norm of $f\in L^1([a,b],\mu)$ may depend on the representant $g\in[f]:=\{\bar f\in\mathcal L([a,b],\mu):\bar f=f \ \mu\text{-a.e.}\}$. 
Therefore, we set $\|f\|_{\mathrm{TV[a,b]}}:=\inf_{g\in[f]}\mathrm{var}_{[a,b]}(g)$, where $$\mathrm{var}_{[a,b]}(g):=\sup\left\{\sum_{i=1}^n|g(x_{i+1})-g(x_i)|:n\in\mathbb N,a=x_0<x_1<\cdots<x_{n-1}<x_n=b\right\}.$$
Select a $g\in[f]$. 
Then, $g^d\equiv f^d$ on $[a,b]$. 
Note that for $x\in \mathscr X_n^d$, it holds that $$|g^d(x)-g(x)|\leq\sup_{u,v\in \mathscr X_n^d}|g(u)-g(v)|=:\Delta_n^d.$$ 
It follows that 
\begin{align*}
\|f^d-f\|_{L^1([a,b],\mu)}&=\int_a^b|f^d(x)-f(x)|\ \mathrm d\mu(x)=\int_a^b|g^d(x)-g(x)| \ \mathrm d\mu(x)\\&=\sum_{n\in\mathbb N}\int_{\mathscr X_n^d\cap[a,b]}|g^d(x)-g(x)|\ \mathrm d\mu(x)\leq\sum_{n\in\mathbb N}\int_{\mathscr X_n^d\cap[a,b]}\Delta_n^d\ \mathrm d\mu(x)\\
&=\sum_{\substack{n\in\mathbb N\\ \mu(\mathscr X_n^d\cap[a,b])\neq0}}\mu(\mathscr X_n^d)\Delta_n^d
\leq\mathrm{mesh}(\mathcal P^d)\sum_{\substack{n\in\mathbb N\\ \mu(\mathscr X_n^d\cap[a,b])\neq0}}\Delta_n^d\\&\leq\mathrm{mesh}(\mathcal P^d)\mathrm{var}_{[a,b]}(g).
\end{align*}
Taking the infimum over $g\in[f]$ over both sides of this inequality, \eqref{L1 TV ineq ab} follows.
\end{proof}

{\sc Proof of Proposition~\ref{graphon + process convergence simultaneously}.}
Suppose that the partitions $(\mathcal P^d)_{d\in\mathbb N}$ are induced by  the first $d-1$ hyperplanes of a random sample $(\mathcal V_i)_{i\in\mathbb N}$ of $(m-1)$-dimensional hyperplanes orthogonal to one of the coordinates. 
More specifically, the $\mathcal V_i$s are sampled as follows. 
First, the coordinate $n\in[m]$ to which $\mathcal V_i$ is orthogonal is picked uniformly at random, i.e., $n\sim\mathrm{Uni}(\{1,\ldots,m\})$. 
Next, we draw $z_n\sim\mathrm{Uni}(a_n,b_n)$, and set $\mathcal V_i:=\{\mathbf x\in\mathbb R^m:x_n=z_n\}$. 
For $d=1$, we simply split $\mathscr X$ into two parts, obtaining $\mathcal P^2$. 
Next, for $d\geq2$, $\mathcal V_{d-1}$ intersects $\mathscr X$, splitting $S\in[d-1]$ sets $\mathscr X_{n(1)}^{d-1},\ldots,\mathscr X_{n(S)}^{d-1}$ of $\mathcal P^{d-1}$ into two parts ($S\geq1$ a.s.). 
Now obtain $\mathcal P^d$ as a refinement of $\mathcal P^{d-1}$ by splitting only the set $\mathscr X_{n(i)}^{d-1}$ ($i\in[S]$) having the largest diameter into two parts; in case of a tie, choose randomly between the ties.

Next, take $d\in\mathbb N$, and for each $i\in[d]$, select a point $x_i\in \mathscr X_i^d$. % according to \emph{any} (possibly degenerate) law. 
Let $M_d=\{x_i:i\in[d]\}\subset\mathscr X$ be the set of those points. 
Then $(M_d)_{d\in\mathbb N}$ constitutes a sequence of well-distributed sets in the sense of \cite{Lovasz}, p.\ 185; i.e., $M_d$ converges weakly to a uniform measure on $\mathscr X$. 
By a straightforward modification of the proof of \cite{Lovasz}, Lemma 11.33 --- for which we need Assumption \ref{ass2}, i.e., a.s.\ continuity of $W$ ---, it follows that the random connectivity graph $\mathbb G^d$ on $[d]$ with edge set $\{Z_{ij}^d:i,j\in[d]\}$ converges to $W$ in the cut norm, a.s. 
Since we use possibly non-symmetric graphons and in any case sample looped directed graphs, we use the framework of \cite{decoratedgraphons, decoratedgraphonscompact}, which allows to go beyond simple, symmetric graphs and graphons.

For the ucp convergence, to invoke the proof of Theorem~\ref{quenched theorem}, we only need $\mathrm{mesh}(\mathcal P^d)\to0$ %=\mathcal O(d^{-1})
 as $d\to\infty$, which is evident. 
\hfill$\Box$

\section*{Appendix C: Supplement to Section~\ref{LT stable case}}\label{app C}
Before Assumption~\ref{ass9}, we claimed the following.
\begin{proposition}\label{app3prop} Let $T_{\mathrm{hom}}\in B(L^1[0,1])$ be a bounded operator with $\rho(T_{\mathrm{hom}})>1$. 
Then there exists some $g\in L^1_+[0,1]$ such that $\|T_{\mathrm{hom}}^ng\|_{L^1[0,1]}\to\infty$, as $n\to\infty$.
\end{proposition}
\begin{proof}
Observe that $T^n_{\mathrm{hom}}f$ records the spatial density of expected $n$th-generation offspring in a cluster with an immigrant distributed according to density $f\in L^1[0,1]$. 
Informally, letting $\delta_x$ be the Dirac delta function centered at $x$, $T^n_{\mathrm{hom}}\delta_x$ records the spatial density of eventual expected $n$th-generation offspring within a cluster generated by an immigrant in $x$. 
If $\rho(T_{\mathrm{hom}})>1$, then $\|T_{\mathrm{hom}}^n\|\to\infty$ as $n\to\infty$, hence there exists $(f_n)_{n\in\mathbb N}\subset L^1[0,1]$ with $\|f_n\|_{L^1[0,1]}=1$ for all $n\in\mathbb N$, such that $\|T_{\mathrm{hom}}^nf_n\|_{L^1[0,1]}\to\infty$, as $n\to\infty$. 
Set $\rho':=1+\frac12(\rho(T_{\mathrm{hom}})-1)$ and $\rho'':=1+\frac13(\rho(T_{\mathrm{hom}})-1)$. 

We present a construction for a function $g$, independent of $n$, such that $\|T_{\mathrm{hom}}^ng\|_{L^1[0,1]}\to\infty$, as $n\to\infty$.
For each $n\in\mathbb N$, we can find $f_n\in L^1[0,1]$ with $\|f_n\|_{L^1[0,1]}=1$ such that $\|T_{\mathrm{hom}}^nf_n\|_{L^1[0,1]}\geq(\rho')^n$. 
Approximate $f_n^\pm\in L^1_+[0,1]$ from below by simple functions.
By continuity of norms, we can find a simple function $s_n\leq f_n$ pointwise, satisfying $\|T_{\mathrm{hom}}^ns_n\|_{L^1[0,1]}\geq(\rho'')^n$. 
 %Set $s'_{nl}(\cdot)=s_n(\cdot)\mathbf1\left\{\int_0^1W^{(k)}(\cdot,y)\ \mathrm dy\geq\frac1l\right\}$. It holds that $s'_{nl}\to s_n$ pointwise, and by monotone convergence in $L^1[0,1]$ as well. 
Consequently, for each $n$, we can find a set $A_n\in\mathcal B[0,1]$ of positive measure, corresponding to a normalized indicator function $\chi_n=(\mathrm{Leb}(A_n))^{-1}\mathbf1\{A_n\}$ satisfying $\|T_{\mathrm{hom}}^n\chi_n\|_{L^1[0,1]}\geq(\rho'')^n$. %Note that $\rho''>1$ and
%  set $\gamma=\sqrt{1+\frac{\rho(T_{\mathrm{hom}})-1}4}>1$
Set $\gamma=(1+\rho'')/(2\rho'')\in(\frac12,1)$; then $\gamma\rho''>1$. 
We define $g\in L^1_+[0,1]$ by $g=(1-\gamma)\sum_{n\geq1}\gamma^{n-1}\chi_n$; it is easy to see that $\|g\|_{L^1[0,1]}\leq1$, while for each $n\in\mathbb N$,
\[
\|T_{\mathrm{hom}}^ng\|_{L^1[0,1]}\geq\frac{1-\gamma}{\gamma}\gamma^n\|T_{\mathrm{hom}}^n\chi_n\|_{L^1[0,1]}\geq\frac{1-\gamma}{\gamma}(\gamma\rho'')^n\to\infty,\quad\text{as }n\to\infty.\qedhere
\]
\end{proof}

The existence of $g$ indicates instability of our system. 
To prove $N_T(A)/T\to\infty$ a.s.\ using this function $g$, we need to argue that there is a positive stream of immigrants leading to a point with infinite expected offspring. 
However, it may be possible that $g$ has an infinite spike (of finite $L^1$ size) creating the unstable behavior, and there is no way of bounding below the size of a set of $x\in[0,1]$ having infinite offspring given only the existence of such a function $g$. 
The next proposition indicates that this approach is not fruitful without using further properties of $(T_{\mathrm{hom}}^n)_{n\geq0}$.
% Similarly, if we approximate $g$ by simple functions $\tilde g$, we may loose the property $\|T_{\mathrm{hom}}^n\tilde g\|_{L^1[0,1]}\to\infty$. 

\begin{proposition}\label{heuristic does not work}
There exist a sequence $(L_{n})_{n\in\mathbb N}\subset B(L^1[0,1])$ and a $g\in L^1[0,1]$ satisfying $\|L_n g\|_{L^1[0,1]}\to\infty$, but such that there does not exist a simple function $\tilde g$ with $\|L_n\tilde g\|_{L^1[0,1]}\to\infty$.
\end{proposition}
\begin{proof}
The existence of such a simple function $\tilde g$ would imply the existence of some $A\in\mathcal B[0,1]$ for which $\|L_n\mathbf1_A\|_{L^1[0,1]}\to\infty$. 
Note that $\|L_n(\cdot)\|\in L^1[0,1]^*$, the dual of $L^1[0,1]$, which is isometrically isomorphic to $L^\infty[0,1]$. 
In this dual space, to prove our result it suffices to find a sequence $(\phi_n)_{n\in \mathbb N}\subset L^\infty[0,1]$ and a $g\in L^1[0,1]$ such that $\int_0^1\phi_n(x)g(x)\ \mathrm dx\to\infty$ as $n\to\infty$, and such that for any indicator function $\mathbf1_A$ it holds that $\int_0^1\phi_n(x)\mathbf1_A(x)\ \mathrm dx=\int_A\phi_n(x)\ \mathrm dx\to0$. 
To this end, let $g(x)=1/(2\sqrt x)$ and $\phi_n(x)=n^{3/4}\mathbf1_{[0,n^{-1}]}$. 
For any $A\in \mathcal B[0,1]$, \[\int_A\phi_n(x)\ \mathrm dx=n^{3/4}\ \mathrm{Leb}([0,n^{-1}]\cap A)\leq n^{-1/4}\to0,\quad\text{as }n\to\infty.\qedhere\] 
%i.e., there does not exist an indicator function $\mathbf 1_A$ such that $\liminf_{n\to\infty}\int_0^1\phi_n(x)\mathbf1_A(x)\ \mathrm dx>0$.
\end{proof}

\section*{Appendix D: Relegated proofs of Section~\ref{section6 fixed point theorems}}\label{app D}

{\sc Proof of Theorem~\ref{transform Q S_x}}. 
Conditionally upon having $n$ immigrants in $[0,t]$, those immigrants arrive uniformly over time, since \emph{ignoring the spatial coordinate}, they arrive according to a homogeneous Poisson process of rate $\alpha:=\|\lambda_\infty(\cdot)\|_{L^1[0,1]}$. 
Furthermore, the spatial locations are i.i.d.\ with Lebesgue densities proportional to $\lambda_\infty(\cdot)$. 
Let $p^n$ denote the Lebesgue density of the spatial coordinates in $[0,1]^n$ for a sample of $n$ immigrants. 
We calculate the Laplace functional of $Q$ by conditioning on the number of immigrants and on their spatial coordinates. 
Let $I_t$ be the number of immigrants that have arrived by time $t$. 
Then, with the understanding that $\int_{[0,1]^0}f\equiv1$,
\begin{align*}\pushQED{\qed} 
\mathcal L_Q(f,t)&=\mathbb E\left[\exp\left(-\sum_{z_i\in Q(t)}f(z_i)\right)\right]\\
&=\sum_{n\geq0}\mathbb E\left[\mathbb E\left[\exp\left(-\sum_{z_i\in Q(t)}f(z_i)\right)\right]\Bigg|I_t=n\right]\mathbb P(I_t=n)\\
&=\sum_{n\geq0}\int_{[0,1]^n}\mathbb E\left[\mathbb E\left[\exp\left(-\sum_{z_i\in Q(t)}f(z_i)\right)\right]\Bigg|I_t=n,\text{ locations }x_1,\ldots,x_n\right]\\
&\phantom=\qquad \times p^n(x_1,\ldots,x_n)\ \mathrm dx_1\cdots\mathrm d x_n \ \mathbb P(I_t=n)\\
&=\sum_{n\geq0}\int_{[0,1]^n}\prod_{j=1}^n\left(\frac1t\int_0^t\eta_{x_j}(f,u)\ \mathrm du\right)\frac{\lambda_\infty(x_j)}{\|\lambda_\infty(\cdot)\|_{L^1[0,1]}}\ \mathrm dx_1\cdots\mathrm d x_n\ \frac{e^{-\alpha t}(\alpha t)^n}{n!}\\
&=\sum_{n\geq0}\frac{e^{-\alpha t}}{n!}\prod_{j=1}^n\left\{\int_0^1\int_0^t\eta_{x_j}(f,u)\ \mathrm du\ \lambda_\infty(x_j)\ \mathrm dx_j\right\}\\
&=\sum_{n\geq0}\frac{e^{-\alpha t}}{n!}\left(\int_0^1\int_0^t\eta_x(f,u)\lambda_\infty(x)\ \mathrm du\mathrm dx\right)^n\\
&=\exp\left(\int_0^1\int_0^t\left(\eta_x(f,u)-1\right)\lambda_\infty(x)\ \mathrm du\ \mathrm dx\right).
\end{align*}
This proves the stated result.
$\hfill\Box$  
\newline

{\sc Proof of Lemma~\ref{lemmawelldefined}}.
Take $\xi\in\mathbb L^{[0,1]}$, which is the transform of some $[0,1]$-dimensional spatiotemporal point process $Z\in\mathscr Z$. 
We aim to define the probabilistic analog $\hat\Phi$ in the process domain of the operator $\Phi$ in the transform domain, meaning that the diagram depicted in Figure \ref{comdiagram} commutes. 
\begin{figure}[!htbp]
\centering
\begin{tikzcd}
Z \arrow{r}{\hat\Phi} \arrow[swap]{d}{\mathcal L} & \hat\Phi(Z) \arrow{d}{\mathcal L} \\
\xi \arrow{r}{\Phi} & \Phi(\xi)
\end{tikzcd}
\caption{Diagram showing the relation between $\Phi$ and $\hat\Phi$.}\label{comdiagram}
\end{figure}
To this end, for $x\in[0,1]$, let $K_x$ be an inhomogeneous Poisson process as defined after the proof of Theorem~\ref{transform Q S_x}. 
For $Z\in\mathscr Z$, $x\in[0,1]$ and $t\in\mathbb R_+$, we set
\begin{equation}
(\hat\Phi(Z))_x(t):=\delta_x\cdot\mathbf1\{J_x>t\}+\sum_{i=1}^{K_x(t)}Z_{x_i}(t-t_i).
\end{equation}
Then $\hat\Phi(Z)\in\mathscr Z$. 
By the same steps given in the proof of Theorem~\ref{fixedpointtheorem} below,
$$\mathcal L_{(\hat\Phi(Z))_x}(f,t)=\Phi_x(\xi)(f,t),$$ 
whence $\mathcal L_{\hat\Phi(Z)}=\Phi(\mathcal L_Z)$. 
As $\Phi(\xi)$ is the transform of $\hat\Phi(Z)\in\mathscr Z$, $\Phi(\xi)\in\mathbb L^{[0,1]}$.
%We have proven that the following diagram commutes.
$\hfill\Box$  \newline

{\sc Proof of Theorem~\ref{fixedpointtheorem}}.
Note that by the multiplicative structure of the excitation term $B_{zx}(0)W(z,x)h(t)$, it follows that conditional on the arrival of a particle in the time interval $[0,u]$, the spatial placement in $[0,1]$ is independent of the temporal placement in $[0,u]$. 
Let $P_t(s)$ be the probability that an offspring event, conditionally upon being born before time $t$, is born before time $s$. 
Let $H(u):=\int_0^uh(v)\ \mathrm dv$ be the integrated excitation function.
As in \cite{Infinite server queues}, p.\ 12, it follows that $P_t(s)=H(s)/H(t)$, with density $p_t(s)=P'_t(s)=h(s)/H(t)$.

The distributional equality \eqref{distributionaleq} for $S_x$ implies that 
\begin{align}
&\eta_x(f,u)=\mathbb E\left[\exp\left(-\sum_{z_i\in S_x(u)}f(z_i)\right)\right]\nonumber\\
&=\gamma_x(f,u)\mathbb E_B\left[\sum_{n\geq0}\mathbb E\left[\exp\left(-\sum_{z_i\in S_x(u),z_i\neq x}f(z_i)\right)\Bigg|K_x(u)=n,B\right]\mathbb P(K_x(u)=n|B)\right],\label{appendixeq1}
\end{align}
where we exclude $x$ from the first summation since, with probability $1$, there is no event with spatial coordinate $x$, and since we factored out $\gamma_x(f,u)=\left(\bar{\mathscr J}_x(u)+\mathscr J_x(u)e^{-f(x)}\right)$. 
By also conditioning on the locations of the $n$ children of the initial event in $x$, \eqref{appendixeq1} equals
\begin{align}
&\gamma_x(f,u)\mathbb E_B\Bigg[\sum_{n\geq0}\mathbb P(K_x(u)=n|B)\int_{[0,1]^n}p^n(y_1,\ldots,y_n)\nonumber\\
&\times\mathbb E\left[\exp\left(-\sum_{\substack{z_i\in S_x(u)\\z_i\neq x}}f(z_i)\right)\Bigg| K_x(u)=n;\text{ locations }y_1,\ldots y_n;B\right]\ \mathrm dy_1\cdots\mathrm dy_n\Bigg]\nonumber\\
&=\gamma_x(f,u)\mathbb E_B\Bigg[\sum_{n\geq0}e^{-\|B_{\cdot x}(0)W(\cdot,x)\|_{L^1[0,1]}H(u)}\frac{\|B_{\cdot x}(0)W(\cdot,x)\|_{L^1[0,1]}^nH(u)^n}{n!}\nonumber\\
&\phantom=\times\int_{[0,1]^n}\prod_{j=1}^n\eta_{y_j}^{(1)}(f,u)\frac{B_{y_jx}(0)W(y_j,x)}{\|B_{\cdot x}(0)W(\cdot,x)\|_{L^1[0,1]}}\ \mathrm dy_1\cdots\mathrm dy_n\Bigg],\label{appendixeq2}
\end{align}
where $\eta_y^{(1)}$ is the transform of a cluster in $y$ that is born between time $0$ and $u$; its input is $u$, which is the time that has elapsed since the event in $x$ occurred, and is therefore different from $\eta_y$, which takes $u-v$ as input when the event in $y$ is born at time $v$. 
The distribution of the temporal location of such an event is given by $P_t(s)$ defined in the first part of this proof. 
Hence, \eqref{appendixeq2} equals
\begin{align}
&\gamma_x(f,u)\mathbb E_B\Bigg[\sum_{n\geq0}e^{-\|B_{\cdot x}(0)W(\cdot,x)\|_{L^1[0,1]}H(u)}\frac{\|B_{\cdot x}(0)W(\cdot,x)\|_{L^1[0,1]}^nH(u)^n}{n!}\nonumber\\
&\times\prod_{j=1}^n\int_0^1\int_0^up_u(s)\eta_{y_j}(f,u-s)\ \mathrm ds\frac{B_{y_jx}(0)W(y_j,x)}{\|B_{\cdot x}(0)W(\cdot,x)\|_{L^1[0,1]}}\ \mathrm dy_j\Bigg]\nonumber\\
&=\gamma_x(f,u)\mathbb E_B\Bigg[e^{-\|B_{\cdot x}(0)W(\cdot,x)\|_{L^1[0,1]}H(u)}\nonumber\\
&\phantom=\times
\sum_{n\geq0}\frac1{n!}\left(\int_0^1\int_0^uh(s)\eta_y(f,u-s)\ \mathrm dsB_{yx}(0)W(y,x)\ \mathrm dy\right)^n\Bigg]\nonumber\\
&=\gamma_x(f,u)\mathbb E\left[\exp\left(\int_0^1\int_0^u(\eta_y(f,u-s)-1)B_{yx}(0)W(y,x)h(s)\ \mathrm ds\mathrm dy\right)\right]\nonumber\\
&=\gamma_x(f,u)\beta_x\left(\left\{y\mapsto\int_0^u(1-\eta_y(f,u-s))W(y,x)h(s)\ \mathrm ds\right\}\right),
\end{align}
recognizing the Laplace functional $\beta_x$ of $B_{\cdot x}$.
$\hfill\Box$  \newline

{\sc Proof of Lemma~\ref{continuity}}.
We present the proof for $p\in[1,\infty)$; the proof for $p=\infty$ follows along the same lines, and is no more difficult.

Take $\xi,\tilde\xi\in\mathbb L^{[0,1]}$, and fix $\epsilon>0$. We aim to find $\delta>0$ such that 
$$d_{\mathbb L^{[0,1]}}^p(\xi,\tilde\xi)<\delta\implies d_{\mathbb L^{[0,1]}}^p(\Phi(\xi),\Phi(\tilde\xi))<\epsilon,$$ so we even prove uniform continuity.
By the definition of $d_{\mathbb L^{[0,1]},p}$, it holds that 
\begin{equation}\label{distance d^p}
d_{\mathbb L^{[0,1]},p}^p(\Phi(\xi),\Phi(\tilde\xi))=\sup_{\substack{u\in[0,t]\\ f\in\mathrm{BM}_+[0,1]}}\int_0^1|\Phi_x(\xi)(f,u)-\Phi_x(\tilde\xi)(f,u)|^p\ \mathrm dx. 
\end{equation}
We first bound the integrand. 
Since $f$ is positive-valued, $\bar{\mathscr J}_x(u)+\mathscr J_x(u)e^{-f(x)}\in[0,1]$. 
Furthermore, $\xi_y$ is a Laplace functional, hence $\xi_y(f,u-s)\in[0,1]$ as well. 
Therefore, by the mean value theorem applied to $x\mapsto e^x$,
\begin{align*}
&|\Phi_x(\xi)(f,u)-\Phi_x(\tilde\xi)(f,u)|\\
&\leq\bigg|\mathbb E\bigg[\exp\left(\int_0^1\int_0^u(\xi_y(f,u-s)-1)B_{yx}(0)W(y,x)h(s)\ \mathrm ds\mathrm dy\right)\\
&\phantom=-\exp\left(\int_0^1\int_0^u(\tilde\xi_y(f,u-s)-1)B_{yx}(0)W(y,x)h(s)\ \mathrm ds\mathrm dy\right)\bigg]\bigg|\\
&\leq\left|\mathbb E\left[\int_0^1\int_0^u(\xi_y(f,u-s)-\tilde\xi_y(f,u-s))B_{yx}(0)W(y,x)h(s)\ \mathrm ds\mathrm dy\right]\right|\\
&\leq u\|h\|_\infty C_BC_W\int_0^1|\xi_y(f,u)-\tilde\xi_y(f,u)|\ \mathrm dy,
\end{align*} 
using Assumption~\ref{ass cw cb}, stating that $\mathbb E[B_{yx}]\leq C_B$ a.s., and $W\leq C_W$.
Since $p\geq1$, the function $x\mapsto x^p$ is convex. 
Jensen's inequality gives 
\begin{align*}
\int_0^1|\Phi_x(\xi)(f,u)-\Phi_x(\tilde\xi)(f,u)|^p\ \mathrm dx&\leq u^p\|h\|_\infty^p C_B^pC_W^p\int_0^1\int_0^1|\xi_y(f,u)-\tilde\xi_y(f,u)|^p\ \mathrm dy\mathrm dx.
\end{align*}
It now follows that 
\begin{align*}
\sup_{\substack{u\in[0,t]\\ f\in\mathrm{BM}_+[0,1]}}\left(\int_0^1|\Phi_x(\xi)(f,u)-\Phi_x(\tilde\xi)(f,u)|^p\ \mathrm dx\right)^{1/p}\leq t\|h\|_\infty C_BC_W\delta.
\end{align*}
Hence, selecting $\delta<\epsilon/(t\|h\|_\infty C_BC_W)$ guarantees $d_{\mathbb L^{[0,1]},p}(\Phi(\xi),\Phi(\tilde\xi))<\epsilon$.
$\hfill\Box$  \newline

\emph{\sc Proof of Lemma~\ref{contractionlemma}}.
We present the proof for $p\in[1,\infty)$, the proof for $p=\infty$ is no more difficult.

By the bounds appearing in the proof of Lemma~\ref{continuity}, it holds that 
\begin{equation*}
\|\Phi_{x}(\xi)(f,u)-\Phi_{x}(\zeta)(f,u)\|_{L^p[0,1]}\leq u\|h\|_\infty C_BC_W\|\xi_x(f,u)-\zeta_x(f,u)\|_{L^p[0,1]}.
\end{equation*}
Since the Laplace functional takes on values in $[0,1]$,
$$\left\|\xi_x^{(1)}(f,u)-\zeta_x^{(1)}(f,u)\right\|_{L^p[0,1]}\leq 2u\|h\|_\infty C_BC_W, $$ verifying \eqref{contractioneq} for $n=1$ and $C:=2\|h\|_\infty C_BC_W$. 
Now suppose the result holds for $n\in\mathbb N$. 
Then again by the bounds appearing in the proof of Lemma~\ref{continuity}, and using Jensen's inequality,
\begin{align*}
\left|\xi_x^{(n+1)}(f,u)-\zeta_x^{(n+1)}(f,u)\right|^p&\leq \|h\|_\infty^p C_B^pC_W^p\int_0^1\left(\int_0^u\left|\xi_y^{(n)}(f,s)-\zeta_y^{(n)}(f,s)\right| \ \mathrm ds\right)^p\mathrm dy\\
&\leq C^p\int_0^1\left(\int_0^u \frac{C^n s^n}{n!}\ \mathrm ds\right)^p\mathrm dy=\left(\frac{C^{n+1}u^{n+1}}{(n+1)!}\right)^p.
\end{align*}
Integrating over $x\in[0,1]$ proves the bound for $n+1$. 
The stated result follows by induction.
%By Jensen's inequality,
%\begin{align*}
%\left|\xi_x^{(n+1)}(f,u)-\zeta_x^{(n+1)}(f,u)\right|^p&\leq C^p  \int_0^u\int_0^1\left|\xi_y^{(n)}(f,s)-\zeta_y^{(n)}(f,s)\right|^p\ \mathrm dy\mathrm ds\\
%&\leq C^p \int_0^1\int_0^u\frac{C^ns^n}{n!}\ \mathrm ds\ \mathrm dy\leq\frac{C^{n+1}s^{n+1}}{(n+1)!}.
%\end{align*} 
$\hfill\Box$  %\newline

\end{document}